\documentclass[12pt,a4paper]{amsart}
\usepackage[czech,english]{babel}
\usepackage{amssymb}
\usepackage[all,cmtip]{xy}
\usepackage{hyperref}

\calclayout

\newcommand{\+}{\protect\nobreakdash-}
\newcommand{\<}{\protect\nobreakdash--}
\renewcommand{\:}{\colon}

\newcommand{\rarrow}{\longrightarrow}
\newcommand{\ot}{\otimes}
\newcommand{\bec}{\natural}
\newcommand{\dd}{\partial}

\DeclareFontFamily{U}{mathb}{\hyphenchar\font45}
\DeclareFontShape{U}{mathb}{m}{n}{
      <5> <6> <7> <8> <9> <10> gen * mathb
      <10.95> mathb10 <12> <14.4> <17.28> <20.74> <24.88> mathb12
      }{}
\DeclareSymbolFont{mathb}{U}{mathb}{m}{n}
\DeclareFontSubstitution{U}{mathb}{m}{n}
\DeclareMathSymbol{\blackdiamond}{0}{mathb}{"0C}

\newcommand{\bu}{{\text{\smaller\smaller$\scriptstyle\bullet$}}}
\newcommand{\cu}{{\text{\smaller$\scriptstyle\blackdiamond$}}}

\newcommand{\lrarrow}{\mskip.5\thinmuskip\relbar\joinrel\relbar\joinrel
 \rightarrow\mskip.5\thinmuskip\relax}

\DeclareMathOperator{\Hom}{Hom}
\DeclareMathOperator{\cHom}{\mathcal H\mathit{om}}
\DeclareMathOperator{\Ext}{Ext}
\DeclareMathOperator{\Tot}{Tot}
\DeclareMathOperator{\Id}{Id}
\DeclareMathOperator{\id}{id}

\DeclareMathOperator{\cone}{cone}

\DeclareMathOperator{\Fil}{\mathsf{Fil}}

\newcommand{\Sets}{\mathsf{Sets}}
\newcommand{\Ab}{\mathsf{Ab}}

\newcommand{\sA}{\mathsf A}
\newcommand{\sB}{\mathsf B}
\newcommand{\sC}{\mathsf C}
\newcommand{\sD}{\mathsf D}
\newcommand{\sE}{\mathsf E}
\newcommand{\sF}{\mathsf F}
\newcommand{\sG}{\mathsf G}
\newcommand{\sH}{\mathsf H}
\newcommand{\sK}{\mathsf K}
\newcommand{\sL}{\mathsf L}
\newcommand{\sP}{\mathsf P}
\newcommand{\sR}{\mathsf R}
\newcommand{\sS}{\mathsf S}
\newcommand{\sT}{\mathsf T}
\newcommand{\sW}{\mathsf W}
\newcommand{\sX}{\mathsf X}
\newcommand{\sY}{\mathsf Y}
\newcommand{\sZ}{\mathsf Z}

\newcommand{\bA}{\mathbf A}
\newcommand{\bB}{\mathbf B}
\newcommand{\bC}{\mathbf C}
\newcommand{\bE}{\mathbf E}
\newcommand{\bF}{\mathbf F}

\newcommand{\biJ}{\boldsymbol J}
\newcommand{\biL}{\boldsymbol L}
\newcommand{\biM}{\boldsymbol M}
\newcommand{\biN}{\boldsymbol N}
\newcommand{\biP}{\boldsymbol P}
\newcommand{\biR}{\boldsymbol R}

\newcommand{\cC}{\mathcal C}
\newcommand{\cE}{\mathcal E}
\newcommand{\cG}{\mathcal G}
\newcommand{\cH}{\mathcal H}
\newcommand{\cK}{\mathcal K}
\newcommand{\cL}{\mathcal L}
\newcommand{\cP}{\mathcal P}
\newcommand{\cR}{\mathcal R}
\newcommand{\cS}{\mathcal S}
\newcommand{\cW}{\mathcal W}
\newcommand{\cZ}{\mathcal Z}

\newcommand{\Mono}{{\operatorname{-\mathcal M\mathit{ono}}}}
\newcommand{\Epi}{{\operatorname{-\mathcal E\mathit{pi}}}}
\newcommand{\Cof}{\mathcal C\mathit{of}}
\newcommand{\Cell}{\mathcal C\mathit{ell}}

\newcommand{\boZ}{\mathbb Z}
\newcommand{\boI}{\mathbb I}

\newcommand{\sModl}{{\operatorname{\mathsf{--Mod}}}}
\newcommand{\cModl}{{\operatorname{\mathcal{--M}
 \mskip-.7\thinmuskip\mathit{od}}}}
\newcommand{\bModl}{{\operatorname{\mathbf{--Mod}}}}
\newcommand{\sqcoh}{{\operatorname{\mathsf{--qcoh}}}}

\newcommand{\rQcoh}{{\operatorname{\mathrm{--Qcoh}}}}
\newcommand{\rcoh}{{\operatorname{\mathrm{--coh}}}}

\newcommand{\bb}{{\mathsf{b}}}
\newcommand{\abs}{{\mathsf{abs}}}
\newcommand{\bco}{{\mathsf{bco}}}
\newcommand{\bctr}{{\mathsf{bctr}}}
\newcommand{\proj}{{\mathsf{proj}}}
\newcommand{\inj}{{\mathsf{inj}}}
\newcommand{\fp}{{\mathsf{fp}}}
\newcommand{\ac}{{\mathsf{ac}}}

\newcommand{\qc}{{\mathrm{qc}}}
\newcommand{\coh}{{\mathrm{coh}}}

\newcommand{\bproj}{{\mathbf{proj}}}
\newcommand{\binj}{{\mathbf{inj}}}
\newcommand{\bfp}{{\mathbf{fp}}}

\newcommand\hathatRbu{\widehat{\widehat\biR}
 {\vphantom{\widehat\biR^t}}^\bu}

\newcommand{\Section}[1]{\bigskip\section{#1}\medskip}
\theoremstyle{plain}
\newtheorem{thm}{Theorem}[section]

\newtheorem{lem}[thm]{Lemma}
\newtheorem{prop}[thm]{Proposition}
\newtheorem{cor}[thm]{Corollary}
\theoremstyle{definition}
\newtheorem{ex}[thm]{Example}
\newtheorem{exs}[thm]{Examples}
\newtheorem{rem}[thm]{Remark}

\begin{document}

\title{Coderived and contraderived categories \\
of locally presentable abelian DG-categories}

\author{Leonid Positselski}

\address{Leonid Positselski, Institute of Mathematics, Czech Academy
of Sciences \\ \v Zitn\'a~25, 115~67 Praha~1 \\ Czech Republic}

\email{positselski@math.cas.cz}

\author{Jan \v S\v tov\'\i\v cek}

\address{Jan {\v S}{\v{t}}ov{\'{\i}}{\v{c}}ek, Charles University,
Faculty of Mathematics and Physics, Department of Algebra,
Sokolovsk\'a 83, 186 75 Praha, Czech Republic}

\email{stovicek@karlin.mff.cuni.cz}

\begin{abstract}
 The concept of an abelian DG\+category, introduced by the first-named
author in~\cite{Pedg}, unites the notions of abelian categories and
(curved) DG\+modules in a common framework.
 In this paper we consider coderived and contraderived categories in
the sense of Becker.
 Generalizing some constructions and results from the preceding papers
by Becker~\cite{Bec} and by the present authors~\cite{PS4}, we define
the contraderived category of a locally presentable abelian
DG\+category $\bB$ with enough projective objects and the coderived
category of a Grothendieck abelian DG\+category~$\bA$.
 We construct the related abelian model category structures and show
that the resulting exotic derived categories are well-generated.
 Then we specialize to the case of a locally coherent Grothendieck
abelian DG\+category $\bA$, and prove that its coderived category is
compactly generated by the absolute derived category of finitely
presentable objects of~$\bA$, thus generalizing a result from
the second-named author's preprint~\cite{Sto2}.
 In particular, the homotopy category of graded-injective left
DG\+modules over a DG\+ring with a left coherent underlying graded ring
is compactly generated by the absolute derived category of DG\+modules
with finitely presentable underlying graded modules.
 We also describe compact generators of the coderived categories of
quasi-coherent matrix factorizations over coherent schemes.
\end{abstract}

\maketitle

\tableofcontents

\section*{Introduction}
\medskip

 To any abelian category\/ $\sE$, one can assign its unbounded derived
category $\sD(\sE)$.
 To any DG\+ring $\biR^\bu=(R^*,d)$, one can assign the derived category
$\sD(\biR^\bu\bModl)$ of DG\+modules over~$\biR^\bu$.
 Is there a natural common generalization of these two constructions?
 In this paper, following a preceding paper~\cite{Pedg}, we suggest
a framework providing an answer to this question, but there is a caveat.
 We consider \emph{derived categories of the second kind} instead of
the conventional derived categories
(the terminology is inspired by related constructions of derived functors in~\cite{HMS74}).

 Simply put, this means that the homotopy category
$\sH^0(\bC(\sE_\inj))$ of unbounded complexes of injective objects in
$\sE$ and the homotopy category $\sH^0(\bC(\sE_\proj))$ of unbounded
complexes of projective objects in $\sE$ are our two versions of
the derived category of~$\sE$.
 Similarly, the homotopy category $\sH^0(\biR^\bu\bModl_\binj)$ of
DG\+modules $\biJ^\bu=(J^*,d_J)$ over $\biR^\bu$ with injective
underlying graded $R^*$\+modules $J^*$ and the homotopy category
$\sH^0(\biR^\bu\bModl_\bproj)$ of DG\+modules $\biP^\bu=(P^*,d_P)$ over
$\biR^\bu$ with projective underlying graded $R^*$\+modules $P^*$ are
our two versions of the derived category of DG\+modules over~$\biR^\bu$.

Derived categories of the second kind were, in fact, recently
studied in numerous contexts under various names. In connection with
representation theory of finite groups and commutative algebra, they
were studied by J\o rgensen~\cite{Jor} and Krause~\cite{Kra}. They
also naturally appear in the context of Grothendieck duality in the
work of Krause and Iyengar~\cite{IK06} and Gaitsgory
and Rozenblyum~\cite{Gaitsgory,GR}
(the derived ind-completion of the bounded derived category of
coherent sheaves on a nice enough scheme $X$ is equivalent to
$\sH^0(\bC(\sE_\inj))$ for $\sE=X\sqcoh$, essentially by~\cite{Kra}).
Lurie studied their universal properties under the name
``unseparated derived category'' \cite[Appendix C.5]{Lur-SAG}.
They play an essential role in understanding the Koszul duality
(see e.~g.\ \cite{LH03,Pkoszul,Pksurv}).
Last but not least, Becker~\cite{Bec} significantly contributed to 
the theory of derived categories of the second kind for curved
DG\+modules, motivated by the special case of 
categories of matrix factorizations
(whose theory, going back to~\cite{Eisenbud}, requires derived
categories of the second kind for its full development, as
shown in~\cite{Or,EP}).

We will use the terminology from the paper by Becker~\cite{Bec} and 
name the homotopy category of unbounded complexes of injective
objects or the homotopy category of graded-injective DG\+modules
the \emph{coderived category in the sense of Becker}.
 Dually, we call the homotopy category of unbounded complexes
of projective objects or the homotopy category of graded-projective
DG\+modules the \emph{contraderived category in the sense of Becker}.
 To be precise, the definitions of Becker's contraderived and coderived
categories in Sections~\ref{becker-contraderived-subsecn}
and~\ref{becker-coderived-subsecn} are spelled out in a more fancy way:
these derived categories of the second kind are defined as certain
Verdier quotient categories of the homotopy category of arbitrary
complexes or (curved) DG\+modules, rather than subcategories.
 But then Corollaries~\ref{becker-contraderived-corollary}
and~\ref{becker-coderived-corollary} tell that, in particular, in
the case of DG\+modules, the Becker contraderived and coderived
categories as defined in Sections~\ref{becker-contraderived-subsecn}
and~\ref{becker-coderived-subsecn} agree with the homotopy categories
of graded-projective and graded-injective DG\+modules, respectively.

\smallskip


Let us now explain what we believe is a main contribution of the 
present paper. The second-named author proved the following
fact in~\cite{Sto2} (which was one of the main results there):

\begin{thm}[{\cite[Corollary~6.13]{Sto2}}] \label{locally-coherent-abelian-theorem}
 Let\/ $\sA$ be a locally coherent Grothendieck abelian category and\/
$\sH^0(\bC(\sA_\inj))$ be the homotopy category of unbounded complexes
of injective objects in\/~$\sA$.
 Then\/ $\sH^0(\bC(\sA_\inj))$ is a compactly generated triangulated
category.
 The full subcategory of compact objects in\/ $\sH^0(\bC(\sA_\inj))$
is equivalent to the bounded derived category\/ $\sD^\bb(\sA_\fp)$
of the abelian category of finitely presentable objects in\/~$\sA$.
\end{thm}

 This was already a substantial generalization
of Krause's~\cite[Theorem 1.1(2)]{Kra}, where the same result
was obtained for $\sA$ locally Noetherian.
Here we go further and prove
the following generalization of
Theorem~\ref{locally-coherent-abelian-theorem}, where by a locally
coherent abelian DG\+category, we mean
a DG\+category $\bA$ with finite direct sums, shifts and cones
whose underlying additive category $\sZ^0(\bA)$ is a locally
coherent Grothendieck category.

\begin{thm}[Theorem~\ref{becker-coderived-compactly-generated} below]
\label{locally-coherent-abelian-DG-category-theorem}
 Let\/ $\bA$ be a locally coherent abelian DG\+category and\/
$\sD^\bco(\bA)=\sH^0(\bA_\binj)$ be its Becker's coderived category,
that is,
the homotopy category of graded-injective objects in\/~$\bA$.
 Then\/ $\sD^\bco(\bA)$ is a compactly generated triangulated
category.
 Up to adjoining direct summands, the full subcategory of compact
objects in\/ $\sD^\bco(\bA)$ is equivalent to the absolute
derived category\/ $\sD^\abs(\bA_\bfp)$ of the abelian DG\+category
of finitely presentable objects in\/~$\bA$.
\end{thm}

 We will explain details in the following paragraphs. 
At the moment, we just point out that when specializing to the
abelian DG\+categories of
curved DG\+modules, we obtain the following corollary,
which seems to be new as well.
 (The reader would not lose much by assuming further that $h=0$,
in which case $\biR^\bu=(R^*,d)$ is a DG\+ring and
the curved DG\+modules mentioned in the corollary are simply
the DG\+modules; but curved DG\+modules are a more natural context for
derived categories of the second kind.)

\begin{cor}[Corollary~\ref{CDG-modules-becker-coderived-comp-gen} below]
\label{graded-coherent-CDG-ring-cor}
 Let $\biR^\cu=(R^*,d,h)$ be a curved DG\+ring whose underlying graded
ring $R^*$ is graded left coherent (i.~e., all finitely generated
homogeneous left ideals in $R^*$ are finitely presentable).
 Then the homotopy category\/ $\sH^0(\biR^\cu\bModl_\binj)$ of left
curved DG\+modules over $\biR^\cu$ with injective underlying graded
$R^*$\+modules is compactly generated.
 Up to adjoining direct summands, the full subcategory of compact
objects in\/ $\sH^0(\biR^\cu\bModl_\binj)$ is equivalent to
the absolute derived category\/ $\sD^\abs(\biR^\cu\bModl_\bfp)$ of
left curved DG\+modules over $\biR^\cu$ with finitely presentable
underlying graded $R^*$\+modules.
\end{cor}

 We also work out the following application of
Theorem~\ref{locally-coherent-abelian-DG-category-theorem}
to quasi-coherent matrix factorizations on coherent schemes,
generalizing the previously known result for Noetherian
schemes~\cite[Proposition~1.5(d) and Corollary~2.3(l)]{EP}.
If $X$ is affine and regular and $L=O_X$, the absolute derived category below is none other than the homotopy category of classical matrix factorizations with projective components (there, compact generators were studied in detail in special cases e.g.\ in \cite{Dyc11}).

\begin{cor}[Corollary~\ref{qcoh-factorizations-coderived-comp-gen}
below] \label{matrix-factorizations-cor}
 Let $X$ be coherent scheme (i.~e., $X$ is quasi-compact,
quasi-separated, and locally coherent).
 Let $L$ be a line bundle (invertible sheaf) on $X$ and $w\in L(X)$ be
a global section.
 Then the homotopy category\/ $\sH^0(\bF_\qc(X,L,w)_\binj)$ of injective
quasi-coherent factorizations of the potential~$w$ on $X$ is compactly
generated.
 Up to adjoining direct summands, the full subcategory of compact
objects in\/ $\sH^0(\bF_\qc(X,L,w)_\binj)$ is equivalent to the absolute
derived category of coherent factorizations\/
$\sD^\abs(\bF_\coh(X,L,w))$.
\end{cor}

Along with developing general theory of the derived categories 
of the second kind for abelian DG\+categories,
we also clarify certain rather basic aspects of their construction.
For example, the ordinary derived
category $\sD(\sE)$ of an abelian category $\sE$ is usually
constructed as a Verdier quotient
$\sH^0(\bC(\sE))/\sH^0(\bC(\sE))_\ac$ of the homotopy category of
complexes modulo the full subcategory of acyclic complexes. It has 
been known for some time \cite{Gil,PS4}
that one can construct Becker's coderived category
of a Grothendieck abelian category $\sA$ in 
the same way as $\sH^0(\bC(\sA))/\sH^0(\bC(\sA))_\ac^\bco$.
However, the corresponding  subcategory $\sH^0(\bC(\sA))_\ac^\bco$
of the so-called \emph{Becker-coacyclic objects} did not have an easy
description.
Here we prove in
Corollary~\ref{coacyclics-as-closure-extensions-directed-colimits}
(in fact, in a much more general setting) that the class of
Becker-coacyclic objects viewed in the Grothendieck category of 
complexes $\sZ^0(\bC(\sA))$ is given precisely as the closure
of the class of contractible complexes under extensions and directed
colimits. If we turn back to the locally coherent case, we obtain
the following consequence.

\begin{cor}
 Let\/ $\sA$ be a locally coherent Grothendieck abelian category.
Then the class of Becker-coacyclic objects coincides in
$\sZ^0(\bC(\sA))$ with the direct limit closure of
the subcategory
of acyclic bounded complexes of finitely presentable objects.
\end{cor}

This is an amazingly simple description, and the general form for
locally coherent abelian DG\+categories
(see Proposition~\ref{coacyclics-as-varinjlim-of-abs-acyclics}),
which in particular includes the case of (curved) DG\+modules
studied by Becker~\cite{Bec}, is equally nice.
Again, this appears to be completely new.

Now we are going to introduce our aims and the our setting 
more in detail.


\subsection{The starting point}
Speaking of DG\+modules, one observes that the category of DG\+rings
is a (nonfull) subcategory in a larger category of so-called
\emph{curved DG\+rings} (\emph{CDG\+rings}), and the assignment of
the DG\+category of DG\+modules to a DG\+ring extends naturally to
an assignment of the DG\+category of \emph{CDG\+modules} to any
curved DG\+ring~\cite{Pcurv,Pkoszul}.
 As mentioned, CDG\+rings and CDG\+modules, rather than DG\+rings and DG\+modules,
form a natural context for the homotopy categories of graded-injective
and graded-projective differential modules.
 The theory of the homotopy categories of graded-injective and
graded-projective CDG\+modules over a CDG\+ring, together with
the related complete cotorsion pairs and abelian model structures,
was developed in Becker's paper~\cite{Bec}.
 The relevant result in~\cite{Bec} is~\cite[Proposition~1.3.6]{Bec}. 

 Speaking of abelian categories, a natural generality for
the derived categories of the second kind in the sense of Becker is
achieved in certain classes of locally presentable abelian categories.
 The natural context for the homotopy category of unbounded complexes
of injective objects $\sH^0(\bC(\sA_\inj))$ seems to be that of
Grothendieck abelian categories~$\sA$.
 Dually, the natural generality level for the homotopy category of
unbounded complexes of projective objects $\sH^0(\bC(\sB_\proj))$ is
that of locally presentable abelian categories $\sB$ with enough
projective objects.
 This is the point of view suggested in the present authors'
paper~\cite{PS4}.

 The coderived abelian model structure on the category of complexes in
a Grothendieck abelian category $\sA$ was constructed by
Gillespie in~\cite[Theorem~4.2]{Gil}; another exposition can be
found in~\cite[Section~9]{PS4}.
 Inverting the weak equivalences in this model category structure
produces the homotopy category of unbounded complexes of injectives 
in~$\sA$.
 For discussions of the homotopy category of unbounded complexes of
injective objects in a Grothendieck category \emph{not} based on
the notions of cotorsion pairs and abelian model structures,
see~\cite{Neem2,Kra3}.
 The contraderived abelian model structure on the category of complexes
in a locally presentable abelian category $\sB$ with enough
projective objects was constructed in~\cite[Section~7]{PS4}.
 Inverting the weak equivalences in this model category structure
produces the homotopy category of unbounded complexes of projectives
in~$\sB$.


\subsection{Abelian DG\+categories}
 The unifying framework of \emph{abelian} (and \emph{exact})
\emph{DG\+categories} was suggested in the first-named author's
paper~\cite{Pedg}.
 A simple definition of an abelian DG\+category $\bE$ is that $\bE$
is a DG\+category with finite direct sums, shifts, and cones such
that the additive category $\sZ^0(\bE)$ of closed morphisms of
degree~$0$ in $\bE$ is abelian.

 A more substantial approach is based on the idea that there are
\emph{two} abelian categories naturally associated with an abelian
DG\+category.
 For example, to the DG\+category $\biR^\bu\bModl$ of DG\+modules over
a DG\+ring $\biR^\bu=(R^*,d)$ one can assign the abelian category
$\sZ^0(\biR^\bu\bModl)$ of DG\+modules over $\biR^\bu$ and closed
morphisms of degree~$0$ between them, and also the abelian category
$R^*\sModl$ of graded $R^*$\+modules and homogeneous morphisms of
degree~$0$ between them.
 The proper theory of $\biR^\bu\bModl$ as an abelian DG\+category
must involve both the abelian categories $\sZ^0(\biR^\bu\bModl)$ and
$R^*\sModl$, as well as naturally defined functors acting between
these abelian categories.
 Such functors include not only the forgetful functor
$\sZ^0(\biR^\bu\bModl)\rarrow R^*\sModl$ assigning to a DG\+module its
underlying graded module, but also the two functors $G^+$
and $G^-\:R^*\sModl\rarrow\sZ^0(\biR^\bu\bModl)$ left and right adjoint
to the forgetful functor, which happen to differ by a shift:
$G^-=G^+[1]$ (see the left-hand side of
Figure~\ref{the-underlying-categories-fig}).

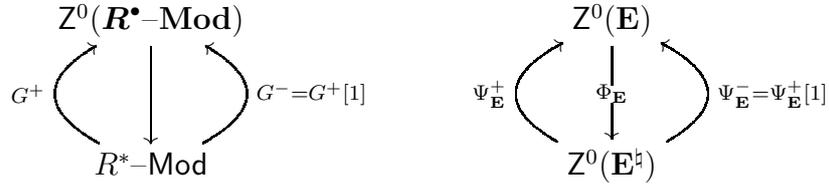
\begin{figure}
\[
\xymatrix@R=3em{
\sZ^0(\biR^\bu\bModl)
\ar@{<-}@/_3pc/[d]_{G^+}
\ar[d]
\ar@{<-}@/^3pc/[d]^{G^-=G^+[1]}
& \qquad\qquad\quad &
\sZ^0(\bE)
\ar@{<-}@/_3pc/[d]_{\Psi_\bE^+}
\ar|{\Phi_\bE}[d]
\ar@{<-}@/^3pc/[d]^{\Psi_\bE^-=\Psi_\bE^+[1]}
\\
R^*\sModl
&&
\sZ^0(\bE^\bec)
}
\]
\caption{Two ``underlying'' additive categories of a DG\+category.}
\label{the-underlying-categories-fig}
\end{figure}

 Given a DG\+category $\bE$ (with finite direct sums), one would like
to construct an additive category which would serve as ``the category
of underlying graded objects'' for the objects of~$\bE$.
 A na\"\i ve approach would be to consider the additive category $\bE^0$
whose objects are the objects of $\bE$ and whose morphisms are
the arbitrary (not necessarily closed) morphisms of degree~$0$ in~$\bE$.
 The problem with this construction is that, taking $\bE=\biR^\bu\bModl$
to be the DG\+category of DG\+modules over a DG\+ring
$\biR^\bu=(R^*,d)$, the additive category $\bE^0$ is usually \emph{not}
abelian, and generally not well-behaved.
 The reason is that, while the morphisms in $\bE^0$ are what
the morphisms in the category of underlying graded objects are supposed
to be, there are too many ``missing objects'' in~$\bE^0$.
 Simply put, an arbitrary graded $R^*$\+module does \emph{not}
admit a DG\+module structure over $(R^*,d)$ in general.

 The main construction of the paper~\cite{Pedg} (first suggested in
the memoir~\cite[Section~3.2]{Pkoszul}) assigns to a DG\+category $\bE$
another DG\+category~$\bE^\bec$.
 Assuming, as we will always do, that the DG\+category $\bE$ has
finite direct sums, shifts, and cones, the DG\+category $\bE^\bec$
comes endowed with an additive functor $\Phi_\bE\:\sZ^0(\bE)\rarrow
\sZ^0(\bE^\bec)$ acting between the additive categories of closed
morphisms of degree~$0$ in $\bE$ and~$\bE^\bec$.
 The functor $\Phi_\bE$ has adjoint functors on both sides, denoted
by $\Psi^+_\bE$ and $\Psi^-_\bE\:\sZ^0(\bE^\bec)\rarrow\sZ^0(\bE)$.
 The two functors $\Psi^+$ and $\Psi^-$ only differ by a shift: one has
$\Psi^-(X)=\Psi^+(X)[1]$ for all $X\in\bE^\bec$ (see the right-hand
side of Figure~\ref{the-underlying-categories-fig}; so the functor
$\Phi_\bE$ is actually a part of a doubly infinite ladder of adjoint
functors acting between $\sZ^0(\bE)$ and $\sZ^0(\bE^\bec)$).
 Furthermore, both the functors $\Phi$ and $\Psi^+$ can be naturally
extended to fully faithful functors acting from the additive
categories of arbitrary morphisms of degree~$0$ in $\bE$ and~$\bE^\bec$;
so there are fully faithful functors $\widetilde\Phi_\bE\:\bE^0
\rarrow\sZ^0(\bE^\bec)$ and $\widetilde\Psi^+_\bE\:(\bE^\bec)^0
\rarrow\sZ^0(\bE)$.

 The existence of a fully faithful functor $\widetilde\Phi_\bE$ is
one way of saying that the additive category $\sZ^0(\bE^\bec)$ is
a candidate for the role of the category of underlying graded objects
for $\bE$, having the correct groups of morphisms between objects
coming from $\bE$ but also some ``missing objects'' added.
 The functor $\Phi_\bE$ plays the role of the forgetful functor
assigning to a (curved) DG\+module its underlying graded module.
%
%
%
This in fact works very well in the case of abelian DG\+categories,
where $\sZ^0(\bE)$ being abelian automatically implies that also
$\sZ^0(\bE^\bec)$ is abelian.
If one wishes to generalize the theory from abelian to exact
DG\+categories (which we are not going to do here),
the category $\sZ^0(\bE^\bec)$ may turn out be  too big
and one may want to consider so-called ``exact DG\+pairs'' 
rather than just exact DG\+categories (we refer to~\cite{Pedg}
for details).

 The construction assigning the DG\+category $\bE^\bec$ to
a DG\+category $\bE$ is an ``almost involution''.
 For any DG\+category $\bE$ with finite direct sums, shifts,
and cones, there is a fully faithful DG\+functor
$\bec\bec\:\bE\rarrow\bE^{\bec\bec}$.
 The passage from a DG\+category $\bE$ to the DG\+category
$\bE^{\bec\bec}$ adjoins all twists of objects with respect to
Maurer--Cartan cochains in their complexes of endomorphisms
(and some direct summands of such twists).
 In fact, for any DG\+category $\bE$, all twists exist in
the DG\+category~$\bE^\bec$.
 For an idempotent-complete DG\+category $\bE$ with finite direct
sums, shifts, and twists, the DG\+functor $\bec\bec\:\bE\rarrow
\bE^{\bec\bec}$ is an equivalence of DG\+categories.

 In particular, any abelian DG\+category $\bE$ is idempotent-complete
and has twists.
 In other words, for any abelian DG\+category $\bE$,
the DG\+functor $\bec\bec\:\bE\rarrow\bE^{\bec\bec}$ is an equivalence
of DG\+categories.
 So the passage from $\bE$ to $\bE^\bec$ is indeed an involution on
abelian DG\+categories.
 The reader should be warned that the assignment of $\bE^\bec$ to
$\bE$ does \emph{not} preserve quasi-equivalences of DG\+categories,
however.
 A DG\+category being abelian is also a property of this DG\+category
considered up to equivalence, but \emph{not} up to quasi-equivalence.


\subsection{Derived categories of the second kind}
This is a common name for
several constructions, most notably the coderived, contraderived, and
\emph{absolute derived} categories.
 The constructions of derived categories of the second kind are
well-defined for abelian (and exact) DG\+categories, which seem to
form their natural context.
 Once again, we should however warn the reader that the construction
of the conventional derived category does \emph{not} seem to make
sense for a general abelian DG\+category.

 In particular, for any DG\+ring $\biR^\bu=(R^*,d)$, and in fact for
any CDG\+ring $\biR^\cu=(R^*,d,h)$, one can easily construct
an \emph{acyclic} DG\+ring $\hathatRbu$ such that the DG\+category of
DG\+modules over $\hathatRbu$ is equivalent to the DG\+category of
(C)DG\+modules over~$\biR^\cu$.
 In this sense, one can say that any CDG\+ring, and in particular
any DG\+ring, is Morita equivalent to an acyclic DG\+ring.
 This kind of Morita equivalence preserves the DG\+category of
(C)DG\+modules, and in fact it preserves the coderived, contraderived,
and absolute derived categories.
 But, of course, the conventional derived category of DG\+modules
over a DG\+ring vanishes if and only if the DG\+ring is acyclic.

 The time has come for us to say explicitly that the coderived and
contraderived categories in the sense of Becker have to be
distinguished from the coderived and contraderived category in
the sense of the first-named author of this paper, as defined
in the monograph~\cite{Psemi}, the memoir~\cite{Pkoszul}, and
subsequent publications such as~\cite{EP,Pps}.
 In fact, it is an open question whether derived categories of
the second kind in the sense of Becker are \emph{ever} different
from the ones in the sense of Positselski within their common
domain of definition.
 But, at least, as long as the question is open, the distinction
has to be made.
 The coderived and contraderived categories in the sense of Positselski
tend to be defined in a greater generality than the ones in the sense
of Becker, but the latter may have better properties in the more narrow
contexts for which they are suited.
 We refer to~\cite[Remark~9.2]{PS4} or~\cite[Section~7]{Pksurv} for
a discussion of the history and philosophy of derived categories
of the second kind.

 For abelian (and exact) DG\+categories, derived categories of
the second kind in the sense of Positselski are defined and their
theory is developed in the paper~\cite{Pedg}.
 The aim of the present paper is to define and study the coderived
and contraderived categories in the sense of Becker for abelian
DG\+categories.
 The notion of the absolute derived category (which has no known
separate Becker's version, but only one in the sense of the first-named
author of the present paper) turns out to play an important role
in this study, as we have seen in Theorem~\ref{locally-coherent-abelian-DG-category-theorem}.


\subsection{Locally presentable and Grothendieck DG\+categories}

Let us say a few words about locally presentable and Grothendieck
DG\+categories.
 Given a DG\+category $\bB$ with shifts and cones such that all
colimits exist in the additive category $\sZ^0(\bB)$, one can show that
the additive category $\sZ^0(\bB)$ is locally $\lambda$\+presentable
(for a fixed regular cardinal~$\lambda$) if and only if the additive
categry $\sZ^0(\bB^\bec)$ is locally $\lambda$\+presentable.
 If the DG\+category $\bB$ is abelian, then the abelian category
$\sZ^0(\bB)$ has enough projective objects if and only if
the abelian category $\sZ^0(\bB^\bec)$ has enough projective objects.
 Hence the class of ``locally presentable abelian DG\+categories
$\bB$ with enough projective objects'', for which the contraderived
category in the sense of Becker is well-defined.

 Given an abelian DG\+category $\bA$ with infinite coproducts,
the abelian category $\sZ^0(\bA)$ is Grothendieck if and only if
the abelian category $\sZ^0(\bA^\bec)$ is Grothendieck.
 Hence the class of ``Grothendieck abelian DG\+categories'', for
which the coderived category in the sense of Becker is well-behaved.
 These are the contraderived and coderived categories which we study
in the present paper, constructing the contraderived abelian model
structure on the abelian category $\sZ^0(\bB)$ in the former case
and the coderived abelian model structure on the abelian category
$\sZ^0(\bA)$ in the latter case.
 In particular, we obtain the semiorthogonal decompositions of
the homotopy categories $\sH^0(\bB)$ and $\sH^0(\bA)$ associated with
the Becker contraderived and coderived categories, and prove that
such derived categories of the second kind $\sD^\bctr(\bB)=
\sH^0(\bB_\bproj)$ and $\sD^\bco(\bA)=\sH^0(\bA_\binj)$ are
well-generated triangulated categories.

 Let us reiterate the warning that, similarly to~\cite{Pedg},
all the theory of DG\+categories developed in this paper is
a \emph{strict} theory.
 We work with complexes of morphisms in DG\+categories up to (natural)
isomorphism, \emph{not} up to quasi-isomorphism; and accordingly
consider DG\+categories up to equivalence, \emph{not} up to
quasi-equivalence.
 Similarly, a fully faithful DG\+functor for us is a DG\+functor
inducing isomorphisms and \emph{not} just quasi-isomorphisms of
the complexes of morphisms.


\subsection{Structure of the paper}
 This paper is based on an abundance of preliminary material.
 The theory of abelian DG\+categories, as developed in
the paper~\cite{Pedg}, is summarized for our purposes
in Sections~\ref{DG-preliminaries-secn}--\ref{abelian-DG-secn}.
 A summary of the theory of abelian model structures, following
the exposition in~\cite{PS4} (which in turn is based on
the results of~\cite{Hov,Bec,PR}), occupies
Sections~\ref{cotorsion-pairs-secn}--%
\ref{abelian-model-structures-secn}.

 A construction of the contraderived abelian model structure for
a locally presentable abelian DG\+category $\bB$ with enough projective
objects is the main result of Section~\ref{contraderived-secn}.
 In Section~\ref{coderived-secn} we not only construct the coderived
model structure for a Grothendieck abelian DG\+category $\bA$, but
also show that the class of all Becker-coacyclic objects is closed
under directed colimits in~$\sZ^0(\bA)$.

 Various descriptions of the class of all Becker-coacyclic objects
in $\bA$ and (in the case of a locally coherent DG\+category~$\bA$)
of the class of all absolutely acyclic objects in $\bA_\bfp$ are
obtained in Sections~\ref{coderived-secn}
and~\ref{locally-coherent-secn}.
 The main results about compact generators of Becker's coderived
category in the locally coherent case, formulated above as
Theorem~\ref{locally-coherent-abelian-DG-category-theorem}
and Corollary~\ref{graded-coherent-CDG-ring-cor}, are proved
at the end of Section~\ref{locally-coherent-secn}.
 The application to algebraic geometry (matrix factorizations on
coherent schemes) is discussed in the last 
Section~\ref{coherent-schemes-matrix-factorizations-secn}, with
Corollary~\ref{matrix-factorizations-cor} proved at the end of the paper.

\subsection*{Acknowledgement}
 We are grateful to an anonymous referee for reading the manuscript
carefully and suggesting relevant corrections and improvements.
 This research is supported by GA\v CR project 20-13778S.
 The first-named author is also supported by the Institute of
Mathematics, Czech Academy of Sciences (RVO:~67985840).

\Section{Preliminaries on DG-Categories} \label{DG-preliminaries-secn}

 In this paper, we will presume all our graded objects and complexes
to be graded by the group of integers~$\boZ$.
 The differentials on complexes raise the degree by~$1$.
 This follows the convention in~\cite{PS4} and constitutes
a restriction of generality as compared to the paper~\cite{Pedg},
where a grading group $\Gamma$ is considered.
 In practice, this is more of a notational convention than an actual
change of setting, because all the definitions and arguments can be
generalized to a grading group $\Gamma$ in a straightforward way.

 The material below in this section is mostly an extraction
from~\cite[Sections~1--2]{Pedg}.
 The reader can consult with~\cite[Section~1.2]{Pkoszul}
and/or~\cite[Sections~1--2]{Pedg} for details.

\subsection{Enriched categories} \label{enriched-subsecn}
 Let $\sC$ be an (associative, unital) monoidal category with
the tensor product operation $\ot\:\sC\times\sC\rarrow\sC$ and
the unit object $\boI\in\sC$.
 Then a \emph{$\sC$\+enriched category} $\cK$ consists of a class
of objects, an object $\cHom_\cK(X,Y)\in\sC$ defined for every
pair of objects $X$, $Y\in\cK$, a multiplication
(also known as composition) morphism
$\cHom_\cK(Y,Z)\ot\cHom_\cK(X,Y)\rarrow\cHom_\cK(X,Z)$ in~$\sC$
defined for every triple of objects $X$, $Y$, $Z\in\cK$, and a unit
morphism $\boI\rarrow\Hom_\cK(X,X)$ in~$\sC$ defined for every
object $X\in\cK$.
 The usual associativity and unitality axioms are
imposed~\cite[Section~1.2]{Kel}.

 The functor $\Hom_\sC(\boI,{-})$ is a lax monoidal functor from
$\sC$ to the monoidal category of sets $\Sets$ (with the Cartesian
product in the role of the tensor product in $\Sets$).
 Consequently, the functor $\Hom_\sC(\boI,{-})$ takes monoids in $\sC$
to monoids in $\Sets$, and $\sC$\+enriched categories to
$\Sets$\+enriched categories, which means the usual categories.
 In other words, the \emph{underlying category} $\sK=\cK_0$ of
a $\sC$\+enriched category $\cK$ is defined by the rule
$\Hom_\sK(X,Y)=\Hom_\sC(\boI,\cHom_\cK(X,Y))$ (where the class
of objects of $\sK$ coincides with the class of objects of~$\cK$)
\cite[Section~1.3]{Kel}.

\subsection{Graded categories}
 A \emph{preadditive category} $\sE$ is a category enriched in
(the monoidal category of) abelian groups, with the operation of
tensor product of abelian groups~$\otimes_\boZ$ defining
the monoidal structure.

 A \emph{graded category} $\cE$ is a category enriched in (the monoidal
category of) graded abelian groups.
 So, for every pair of objects $X$, $Y\in\cE$, a graded abelian group
$\Hom^*_\cE(X,Y)=\bigoplus_{n\in\boZ}\Hom^n_\cE(X,Y)$ is defined,
together with the related multiplication maps $\Hom^*_\cE(Y,Z)\ot_\boZ
\Hom^*_\cE(X,Y)\rarrow\Hom^*_\cE(X,Z)$ for all objects $X$, $Y$,
$Z\in\cE$ and unit elements $\id_X\in\Hom^0_\cE(X,X)$ for all objects
$X\in\cE$.

 The underlying preadditive category $\sE=\cE^0$ of a graded category
$\cE$ is defined by the rule $\Hom_{\cE^0}(X,Y)=\Hom_\cE^0(X,Y)$.
 This is the definition one obtains by specializing to graded categories
the construction of the underlying category of a $\sC$\+enriched
category explained above in Section~\ref{enriched-subsecn}.

 Given an integer $i\in\boZ$ and an object $X\in\cE$, the \emph{shift}
$Y=X[i]$ of the object $X$ is an object of $\cE$ endowed with
a pair of morphisms $f\in\Hom^{-i}_\cE(X,Y)$ and $g\in\Hom^i_\cE(Y,X)$
such that $gf=\id_X$ and $fg=\id_Y$.
 Given a finite collection of objects $X_1$,~\dots, $X_n\in\cE$,
the \emph{direct sum} $\bigoplus_{i=1}^n X_i\in\cE$ can be simply
defined as the direct sum of the same objects in the preadditive
category~$\cE^0$.
 A graded category $\cE$ is called \emph{additive} if all finite
direct sums exist in $\cE$, or equivalently, if the preadditive
category $\cE^0$ is additive (i.~e., has finite direct sums).

 Let $X_\alpha$ be a (possibly infinite) collection of objects in~$\cE$.
 Then the \emph{coproduct} $Y=\coprod_\alpha X_\alpha$ is
an object of $\cE$ such that a natural isomorphism of graded
abelian groups $\Hom_\cE^*(Y,Z)\simeq\prod_\alpha\Hom_\cE^*(X_\alpha,Z)$
is defined for all objects $Z\in\cE$ in a way functorial with respect
to all morphisms $f\in\Hom^*_\cE(Z,Z')$ in~$\cE$.
 Here $\prod_\alpha$ denotes the product functor in the category of
graded abelian groups.
 Dually, the \emph{product} $W=\prod_\alpha X_\alpha$ is
an object of $\cE$ such that a functorial isomorphism of graded abelian
groups $\Hom_\cE^*(Z,W)\simeq\prod_\alpha\Hom_\cE^*(Z,X_\alpha)$ is
defined for all objects $Z\in\cE$.
 
\subsection{DG-categories and related graded/preadditive
categories}
 A \emph{DG\+cat\-e\-gory} $\bE$ is a category enriched in (the monoidal 
category of) complexes of abelian groups.
 So, for every pair of objects $X$, $Y\in\bE$, a complex of abelian
groups $\Hom^\bu_\bE(X,Y)$ with the components $\Hom^n_\bE(X,Y)$,
\,$n\in\boZ$, and a differential $d\:\Hom^n_\bE(X,Y)\rarrow
\Hom^{n+1}_\bE(X,Y)$ is defined, together with the related
multiplication maps $\Hom^\bu_\bE(Y,Z)\otimes_\boZ\Hom^\bu_\bE(X,Y)
\rarrow\Hom^\bu_\bE(X,Z)$, which must be morphisms of complexes of
abelian groups for all objects $X$, $Y$, $Z\in\bE$.
 The unit elements (identity morphisms) $\id_X\in\Hom_\bE^0(X,X)$
must be closed morphisms of degree~$0$, i.~e., $d(\id_X)=0$ for all
objects $X\in\bE$. {\hbadness=1150\par}

 There are three natural functors from the category of complexes of
abelian groups to the category of graded abelian groups: to any
complex of abelian groups, one can assign its underlying graded abelian
group, or its graded subgroup of cocycles, or its graded abelian
group of cohomology.
 The former functor is monoidal, while the latter two functors are
lax monoidal; so all the three functors, applied to the complexes of
morphisms, transform DG\+categories into graded categories.

 The underlying graded category $\cE=\bE^*$ of a DG\+category $\bE$
can be na\"\i vely defined by the rule that the objects of $\cE$ are
the objects of $\bE$ and $\Hom_\cE^*(X,Y)$ is the underlying graded
abelian group of the complex $\Hom_\bE^\bu(X,Y)$ for all objects
$X$ and~$Y$.
 For the proper (non-na\"\i ve) construction of the underlying graded
category of a DG\+category, see Section~\ref{almost-involution-secn}.
 In particular, the underlying preadditive category $\bE^0$ of
a DG\+category $\bE$ is na\"\i vely defined as $\bE^0=(\bE^*)^0$;
so $\Hom_{\bE^0}(X,Y)=\Hom_\bE^0(X,Y)$.

 A morphism $f\in\Hom^n_\bE(X,Y)$ is said to be \emph{closed} if
$d(f)=0$.
 The graded category whose objects are the objects of $\bE$ and
whose graded abelian group of morphisms $X\rarrow Y$ is
the homogeneous subgroup of closed elements in $\Hom_\cE^*(X,Y)$
is denoted by $\cZ(\bE)$.
 In particular, $\sZ^0(\bE)=(\cZ(\bE))^0$ is the preadditive category
whose objects are the objects of $\bE$ and whose morphisms are
the closed morphisms of degree~$0$ in~$\bE$.
 The construction of the underlying category of a $\sC$\+enriched
category from Section~\ref{enriched-subsecn} assigns to
a DG\+category $\bE$ the preadditive category $\sZ^0(\bE)$.

 The graded category whose objects are the objects of $\bE$ and
whose graded abelian group of morphisms $X\rarrow Y$ is the graded
group of cohomology of the complex $\Hom_\bE^\bu(X,Y)$ is denoted
by $\cH(\bE)$.
 In particular, $\sH^0(\bE)=(\cH(\bE))^0$ is the preadditive category
whose objects are the objects of $\bE$ and whose morphisms are
the cochain homotopy classes of degree~$0$ morphisms in~$\bE$.
 The preadditive category $\sH^0(\bE)$ is known as the \emph{homotopy
category} of a DG\+category~$\bE$.
 Morphisms in $\sZ^0(\bE)$ whose images in $\sH^0(\bE)$ agree are
called \emph{homotopic}, and isomorphisms in $\sH^0(\bE)$ are called
\emph{homotopy equivalences} of objects of~$\bE$.

\subsection{Shifts, direct sums, products, and coproducts in
DG-categories} \label{co-products-in-DG-categories}
 Two objects $X$ and $Y$ in a DG\+category $\bE$ are said to be
\emph{isomorphic} if they are isomorphic in $\sZ^0(\bE)$.
 This is the strict notion of isomorphism in DG\+categories.

 The \emph{shifts} and \emph{finite direct sums} in a DG\+category
$\bE$ can be simply defined as the shifts in the graded category
$\cZ(\bE)$ and the finite direct sums in the preadditive category
$\sZ^0(\bE)$, respectively.
 In particular, given an integer $i\in\boZ$ and an object $X\in\bE$,
the \emph{shift} $Y=X[i]$ of the object $X$ is an object of $\bE$
endowed with a pair of closed morphisms $f\in\Hom_\bE^{-i}(X,Y)$ and
$g\in\Hom_\bE^i(Y,X)$ such that $gf=\id_X$ and $fg=\id_Y$.
 So a DG\+category $\bE$ has shifts if and only if the graded category
$\cZ(\bE)$ does.
 In this case, the graded categories $\bE^*$ and $\cH(\bE)$ also
have shifts.

 The definition of the direct sum of a finite collection of objects
in $\bE$ can be spelled out similarly.
 So the direct sum of a finite collection of objects in $\bE$ is
their direct sum in $\sZ^0(\bE)$.
 A DG\+category $\bE$ is called \emph{additive} if all finite direct
sums exists in $\bE$, or equivalently, if the preadditive category
$\sZ^0(\bE)$ is additive.
 In this case, the graded categories $\cZ(\bE)$, $\bE^*$, $\cH(\bE)$
and the preadditive categories $\bE^0$, $\sH^0(\bE)$ are also
additive.

 Let $X_\alpha$ be an (infinite) collection of objects in~$\bE$.
 Then the \emph{coproduct} $Y=\coprod_\alpha X_\alpha$ is
an object of $\bE$ such that a natural isomorphism of complexes of
abelian groups $\Hom_\bE^\bu(Y,Z)\simeq\prod_\alpha
\Hom_\bE^\bu(X_\alpha,Z)$ is defined for all objects $Z\in\bE$ in
a way functorial with respect to all morphisms $f\in\Hom^n_\bE(Z,Z')$,
\,$n\in\boZ$.
 Here $\prod_\alpha$ denotes the product functor in the category of
complexes of abelian groups.
 Dually, the \emph{product} $W=\prod_\alpha X_\alpha$ is an object of
$\bE$ such that a functorial isomorphism of complexes of abelian
groups $\Hom_\bE^\bu(Z,W)\simeq\prod_\alpha\Hom_\bE^\bu(Z,X_\alpha)$ is
defined for all objects $Z\in\bE$.
 The products and coproducts in $\bE$, when they exist, are defined
uniquely up to a unique isomorphism (i.~e., closed isomorphism
of degree~$0$).
Under mild assumptions on $\bE$, the existence of infinite products
and coproducts in $\bE$ is again equivalent to the existence of such products 
and coproducts, respectively, in $\sZ^0(\bE)$. This will be
made precise in Lemma~\ref{coproducts-in-DG-categories}.

 Notice that the functors assigning to a complex of abelian groups its
grading components, groups of cocycles, and groups of cohomology
commute with infinite products in the respective categories.
 Therefore, any (co)product of a family of objects in a DG\+category
$\bE$ is also their (co)product in the graded categories $\cZ(\bE)$,
$\bE^*$, $\cH(\bE)$ and the preadditive categories $\sZ^0(\bE)$,
$\bE^0$, $\sH^0(\bE)$.

\subsection{Twists and cones in DG-categories}
\label{twists-and-cones-subsecn}
 A \emph{Maurer--Cartan cochain} in the complex of endomorphisms of
an object $X\in\bE$ is an element $a\in\Hom_\bE^1(X,X)$ satisfying
the equation $d(a)+a^2=0$ in $\Hom_\bE^2(X,X)$.
 The \emph{twist} $Y=X(a)$ of the object $X$ by a Maurer--Cartan cochain
$a\in\Hom_\bE^1(X,X)$ is an object of $\bE$ endowed with a pair of
morphisms $f\in\Hom_\bE^0(X,Y)$ and $g\in\Hom_\bE^0(Y,X)$ such that
$gf=\id_X$, \ $fg=\id_Y$, and $d(f)=fa$.
 Assuming the former two equations, the latter one is equivalent to
$d(g)=-ag$ (therefore, $Y=X(a)$ implies $X=Y(-fag)$).

 An object $Y\in\bE$ is a twist of an object $X\in\bE$ (by some
Maurer--Cartan cochain in the complex of endomorphisms) if and only
if the objects $X$ and $Y$ are isomorphic in the preadditive
category~$\bE^0$.
 For any pair of mutually inverse isomorphisms $f\:X\rarrow Y$
and $Y\rarrow X$ in $\bE^0$, the elements $a=gd(f)=-d(g)f\in
\Hom_\bE^1(X,X)$ and $-fag=fd(g)=-d(f)g\in\Hom_\bE^1(Y,Y)$ are
Maurer--Cartan cochains.
 The twist of a given object in $\bE$ by a given Maurer--Cartain
cochain in its complex of endomorphisms, when it exists, is defined
uniquely up to a unique closed isomorphism of degree~$0$.

 An object $C\in\bE$ is said to be the \emph{cone} of a closed
morphism $f\:X\rarrow Y$ of degree~$0$ in $\bE$ if an isomorphism
between the complex of abelian groups $\Hom_\bE^\bu(Z,C)$ and the cone
of the morphism of complexes $\Hom_\bE^\bu(Z,f)\:\Hom_\bE^\bu(Z,X)
\rarrow\Hom_\bE^\bu(Z,Y)$ is specified for all objects $Z\in\bE$,
in a way functorial with respect to all morphisms $g\in
\Hom^n_\bE(Z',Z)$, \,$n\in\boZ$.
 Equivalently, $C$ is the cone of~$f$ if an isomorphism between
the complex of abelian groups $\Hom_\bE^\bu(C[-1],Z)$ and the cone
of the morphism of complexes $\Hom_\bE^\bu(f,Z)\:\Hom_\bE^\bu(Y,Z)
\rarrow\Hom_\bE^\bu(X,Z)$ is specified for all objects $Z\in\bE$,
in a way functorial with respect to all morphisms
$g\in\Hom^n_\bE(Z,Z')$.
 The object $C[-1]$ is called the \emph{cocone} of~$f$.
 The (co)cone of a closed morphism in $\bE$, if it exists, is defined
uniquely up to a unique closed isomorphism of degree~$0$.

 If the shift $X[1]$ and the direct sum $Y\oplus X[1]$ exist in $\bE$,
then the cone $C=\cone(f)$ can be constructed as the twist
$(Y\oplus X[1])(a_f)$ of the object $Y\oplus X[1]$ by a suitable
Maurer--Cartan cochain~$a_f$ produced from the morphism~$f$.
 Consequenly, any additive DG\+category with shifts and twists
has cones.
 Conversely, any DG\+category with shifts, cones, and a zero object
is additive.

 Whenever the shift $X[1]$ and the cone $C=\cone(f)$ exist in $\bE$,
there is a natural short sequence $0\rarrow Y\rarrow C\rarrow X[1]
\rarrow0$ in $\sZ^0(\bE)$ which is split exact in~$\bE^0$.
 Conversely, if $\bE$ is a DG\+category with shifts, then any
short sequence $0\rarrow B\rarrow C\rarrow A\rarrow0$ of closed
morphisms of degree~$0$ in $\bE$ which is split exact in $\bE^0$
arises from a closed morphism $f\:A\rarrow B[1]$; so $C=\cone(f)[-1]$.

 For any additive DG\+category $\bE$ with shifts and cones,
the homotopy category $\sH^0(\bE)$ has a natural structure of
triangulated category.

\subsection{Totalizations of complexes in DG-categories}
\label{totalizations-subsecn}
 Let $X^\bu$ be a complex in the preadditive category $\sZ^0(\bE)$.
 So, for every $n\in\boZ$, the differential $d_n\:X^n\rarrow X^{n+1}$
is a closed morphism of degree~$0$ in $\bE$, and the composition
$X^{n-1}\rarrow X^n\rarrow X^{n+1}$ vanishes as a closed morphism
in~$\bE$.

 Then an object $T\in\bE$ is said to be the \emph{product
totalization} of the complex $X^\bu$ and denoted by
$T=\Tot^\sqcap(X^\bu)$ if an isomorphism between the complex of
abelian groups $\Hom_\bE^\bu(Z,T)$ and the totalization of
the bicomplex of abelian groups $\Hom_\bE^\bu(Z,X^\bu)$, constructed
by taking the countable products of abelian groups along the diagonals,
is specified for all objects $Z\in\bE$, in a way functorial with
respect to all morphisms $g\in\Hom_\bE^i(Z',Z)$, \,$i\in\boZ$.
 Dually, an object $T\in\bE$ is said to be the \emph{coproduct
totalization} of the complex $X^\bu$ and denoted by
$T=\Tot^\sqcup(X^\bu)$ if an isomorphism between the complex of
abelian groups $\Hom_\bE^\bu(T,Z)$ and the totalization of the bicomplex
of abelian groups $\Hom_\bE^\bu(X^\bu,Z)$, constructed by taking
the countable products of abelian groups along the diagonals,
is specified for all objects $Z\in\bE$, in a way functorial with
respect to all morphisms $g\in\Hom_\bE^i(Z,Z')$.

 For a finite complex $X^\bu$ in $\bE$, there is no difference between
the product and the coproduct totalization; so we will write simply
$T=\Tot(X^\bu)$.
 The cone of a morphism $f\:X\rarrow Y$ in $\sZ^0(\bE)$
is the totalization of the two-term complex $X\rarrow Y$.
 Conversely, in a DG\+category $\bE$ with shifts and cones,
the totalization of any finite complex can be expressed as
a finitely iterated cone.

\subsection{DG-functors}
 A \emph{DG-functor} is a functor of categories enriched in complexes
of abelian groups.
 More explicitly, given two DG\+categories $\bA$ and $\bB$,
a DG\+functor $F\:\bB\rarrow\bA$ is a rule assigning to every object
$X\in\bB$ an object $F(X)\in\bA$ and to every pair of objects
$X$, $Y\in\bB$ a morphism of complexes $F_{X,Y}\:\Hom_\bB^\bu(X,Y)
\rarrow\Hom_\bA^\bu(F(X),F(Y))$ in such a way that the compositions
of morphisms and the identity morphisms are preserved.

 Any DG\+functor preserves finite direct sums and shifts of objects,
twists of objects by Maurer--Cartan cochains, and cones of closed
morphisms of degree~$0$.
 In other words, one can say that these operations are examples of
``absolute weighted (co)limits'' in DG\+categories, in the sense
of~\cite[Section~5]{NST}.

 A DG\+functor $F\:\bB\rarrow\bA$ is said to be \emph{fully faithful}
if the map $F_{X,Y}$ is an isomorphism of complexes of abelian groups
for all objects $X$, $Y\in\bB$.
 A DG\+functor $F$ is said to be an \emph{equivalence of DG\+categories}
if it is fully faithful and essentially surjective.
 The latter condition means that for every object $W\in\bA$ there exists
an object $X\in\bB$ together with a closed isomorphism $F(X)\simeq W$
of degree~$0$ in~$\bA$.
 The datum of a fully faithful DG\+functor $\bB\rarrow\bA$ means that,
up to an equivalence, $\bB$ can be viewed as a \emph{full
DG\+subcategory} in~$\bA$.

 Given two additive DG\+categories $\bA$ and $\bB$ with shifts and
cones, any DG\+functor $F\:\bB\rarrow\bA$ induces a triangulated
functor $\sH^0(F)\:\sH^0(\bB)\rarrow\sH^0(\bA)$.

\subsection{Example: DG-category of complexes}
\label{DG-category-of-complexes-defined-subsecn}
 Let $\sE$ be a preadditive category.
 Then the preadditive category $\sG(\sE)$ of graded objects in $\sE$
is constructed as follows.
 The objects of $\sG(\sE)$ are collections of objects $X^*=
(X^n\in\sE)_{n\in\boZ}$ in the category $\sE$ indexed by
the integers~$n$.
 The abelian group of morphisms $\Hom_{\sG(\sE)}(X^*,Y^*)$ is defined
as the product $\prod_{n\in\boZ}\Hom_\sE(X^n,Y^n)$.
 The composition of morphisms and the identity morphisms in $\sG(\sE)$
are constructed in the obvious way.
 So $\sG(\sE)=\sE^\boZ$ is simply the Cartesian product of $\boZ$
copies of~$\sE$.

 The shift functor on $\sG(\sE)$ is defined by the rule $X^*[i]^n=
X^{n+i}$ for all $i$, $n\in\boZ$.
 The graded category $\cG(\sE)$ is constructed as follows.
 The objects of $\cG(\sE)$ are the objects of $\sG(\sE)$, and
the graded abelian group of morphisms $\Hom^*_{\cG(\sE)}(X^*,Y^*)$
has the grading components $\Hom^i_{\cG(\sE)}(X^*,Y^*)=
\Hom_{\sG(\sE)}(X^*,Y^*[i])$ for all $i\in\boZ$.
 Once again, we omit the obvious construction of the composition
of morphisms and identity morphisms in $\cG(\sE)$.
 By the definition, one has $(\cG(\sE))^0=\sG(\sE)$.

 The DG\+category $\bC(\sE)$ of complexes in $\sE$ is constructed
as follows.
 The objects of $\bC(\sE)$ are complexes $X^\bu$ with the components
$X^n\in\sE$; the differentials $d_{X,n}\:X^n\rarrow X^{n+1}$ are
morphisms in~$\sE$.
 Let us denote by $X^*\in\cG(\sE)$ the underlying graded object
of~$X^\bu$.
 Given two complexes $X^\bu$ and $Y^\bu\in\bC(\sE)$, the underlying
graded abelian group of the complex of morphisms $\Hom_{\bC(\sE)}^\bu
(X^\bu,Y^\bu)$ is the graded abelian group $\Hom_{\cG(\sE)}^*(X^*,Y^*)$.
 The differential~$d$ in the complex $\Hom_{\bC(\sE)}^\bu
(X^\bu,Y^\bu)$ is defined by the usual formula $d(f)=d_Y\circ f-
(-1)^{|f|}f\circ d_X$, where $f\in\Hom_{\bC(\sE)}^{|f|}(X^\bu,Y^\bu)$.

 By the definition, the composition of morphisms and the identity
morphisms in the DG\+category $\bC(\sE)$ agree with those in the graded
category $\cG(\sE)$.
 Notice that any graded object in $\sE$ admits a differential making
it a complex in~$\sE$ (e.~g., the zero differential).
 So one has $(\bC(\sE))^*=\cG(\sE)$ and $(\bC(\sE))^0=\sG(\sE)$.

 All shifts exist in the DG\+category $\bC(\sE)$; the shift $X^\bu[i]$
of a complex $X^\bu\in\bC(\sE)$ is constructed by the well-known
rules $X^\bu[i]^n=X^{n+i}$ and $d_{X[i],n}=(-1)^i d_{X,n+i}$.
 Furthermore, all twists exist in the DG\+category $\bC(\sE)$.
 When the category $\sE$ is additive, so is the DG\+category $\bC(\sE)$.
 In this case, the DG\+category $\bC(\sE)$ also has cones.
 When the category $\sE$ has infinite (co)products, so does
the DG\+category $\bC(\sE)$.

 The notation for the categories of closed morphisms in $\bC(\sE)$ is
$\cC(\sE)=\cZ(\bC(\sE))$ and $\sC(\sE)=\sZ^0(\bC(\sE))$.
 These are the (respectively, graded and preadditive) \emph{categories
of complexes} in $\sE$, i.~e., the categories of complexes and closed
morphisms between them.
 The usual notation for the homotopy category is
$\sK(\sE)=\sH^0(\bC(\sE))$.

\subsection{Example: DG-category of CDG-modules}
\label{DG-category-of-CDG-modules-defined-subsecn}
 The concept of a \emph{CDG\+algebra} goes back to
the paper~\cite{Pcurv}.
 A more advanced discussion of CDG\+rings can be found
in~\cite[Section~3]{Pkoszul}, and of CDG\+coalgebras,
in~\cite[Section~4]{Pkoszul}.

 By the definition, a \emph{CDG\+ring} $\biR^\cu=(R^*,d,h)$ is
the following set of data:
\begin{itemize}
\item $R^*=\bigoplus_{n\in\boZ}R^n$ is a graded ring;
\item $d$~is an odd derivation of degree~$1$ on $R^*$, i.~e.,
an additive map $d_n\:R^n\rarrow R^{n+1}$ defined for all $n\in\boZ$
and satisfying the Leibniz rule with signs $d(rs)=d(r)s+(-1)^{|r|}rd(s)$
for all $r\in R^{|r|}$ and $s\in R^{|s|}$;
\item $h\in R^2$ is an element of degree~$2$.
\end{itemize}

 Two equations involving~$d$ and~$h$ must be satisfied:
\begin{itemize}
\item $d^2(r)=hr-rh$ for all $r\in R^*$;
\item $d(h)=0$.
\end{itemize}

 So $\biR^\cu=(R^*,d,h)$ is not a complex: the square of
the differential~$d$ is not equal to zero, but rather to
the commutator with~$h$.
 The element $h\in R^2$ is called the \emph{curvature element} of
a CDG\+ring $(R^*,d,h)$.
 We refer to~\cite[Section~3.2]{Prel}, \cite[Section~6]{Pksurv},
or~\cite[Sections~2.2 and~2.4]{Pedg} for further details on CDG\+rings,
including the (nontrivial!) definition of a morphism of CDG\+rings.

 A \emph{left CDG\+module} $\biM^\cu=(M^*,d_M)$ over a CDG\+ring $\biR^\cu$
is the following set of data:
\begin{itemize}
\item $M^*=\bigoplus_{n\in\boZ}M^n$ is a graded left $R^*$\+module;
\item $d_M$~is an odd derivation of degree~$1$ on the graded left
$R^*$\+module $M^*$ compatible with the derivation~$d$ on the graded
ring $R^*$, i.~e., an additive map $d_{M,n}\:M^n\rarrow M^{n+1}$ is
defined for all $n\in\boZ$ and satisfies the Leibniz rule with
signs $d_M(rx)=d(r)x+(-1)^{|r|}rd_M(x)$ for all $r\in R^{|r|}$ and
$x\in M^{|x|}$.
\end{itemize}
 The following equation describing the square of the differential~$d_M$
must be satisfied:
\begin{itemize}
\item $d_M^2(x)=hx$ for all $x\in M^*$.
\end{itemize}

 Similarly, a \emph{right CDG\+module} $\biN^\cu=(N^*,d_N)$ over
a CDG\+ring $\biR^\cu$ is the following set of data:
\begin{itemize}
\item $N^*=\bigoplus_{n\in\boZ}N^n$ is a graded right $R^*$\+module;
\item an additive map $d_{N,n}\:N^n\rarrow N^{n+1}$ is defined for
all $n\in\boZ$ and satisfies the Leibniz rule with signs
$d_N(yr)=d_N(y)r+(-1)^{|y|}yd(r)$ for all $r\in R^{|r|}$ and
$y\in N^{|y|}$.
\end{itemize}
 The next equation describing the square of the differential~$d_N$
must be satisfied:
\begin{itemize}
\item $d_N^2(y)=-yh$ for all $y\in N^*$.
\end{itemize}

 A \emph{DG\+ring} $\biR^\bu=(R^*,d)$ is a CDG\+ring with vanishing
curvature element, $h=0$.
 DG\+modules over a DG\+ring $(R^*,d)$ are the same things as
CDG\+modules over the CDG\+ring $(R^*,d,0)$.
 The definitions below show that the rather familiar construction of
the DG\+category $\biR^\bu\bModl$ of DG\+modules over a DG\+ring
$\biR^\bu$ can be extended to a construction of the DG\+category
of CDG\+modules $\biR^\cu\bModl$ over an arbitrary CDG\+ring $\biR^\cu$ in
a quite natural way.

 Let $R^*$ be a graded ring.
 We denote by $R^*\sModl$ the abelian category of graded left
$R^*$\+modules (and homogeneous morphisms of degree zero between them).
 Let us define the graded category $R^*\cModl$ of graded left
$R^*$\+modules.

 The objects of $R^*\cModl$ are the graded left $R^*$\+modules.
 Given two objects $L^*$, $M^*\in R^*\cModl$, the graded abelian group
of morphisms $\Hom^*_{R^*\cModl}(L^*,M^*)=\Hom_{R^*}^*(L^*,M^*)$ is
constructed by the following rule.
 For every $i\in\boZ$, the degree~$i$ component $\Hom_{R^*}^i(L^*,M^*)$
is the group of all homogeneous maps of graded abelian groups
$f\:L^*\rarrow M^*$ of degree~$i$ which are compatible with left
$R^*$\+module structures in the sense of the sign rule
$f(rz)=(-1)^{i|r|}rf(z)$ for all $r\in R^{|r|}$ and $z\in L^{|z|}$.
 The composition of morphisms and the unit morphisms in the graded
category $R^*\cModl$ are defined in the obvious way.

 Let $\biR^\cu=(R^*,d,h)$ be a CDG\+ring.
 The DG\+category $\biR^\cu\bModl$ of left CDG\+modules over $\biR^\cu$ is
constructed as follows.
 The objects of $\biR^\cu\bModl$ are the left CDG\+modules over~$\biR^\cu$.
 Given two left CDG\+modules $\biL^\cu=(L^*,d_L)$ and $\biM^\cu=(M^*,d_M)
\in\biR^\cu\bModl$, the underlying graded abelian group of the complex of 
morphisms $\Hom^\bu_{\biR^\cu\bModl}(\biL^\cu,\biM^\cu)=
\Hom^\bu_{\biR^\cu}(\biL^\cu,\biM^\cu)$ is the graded abelian group
$\Hom^*_{R^*}(L^*,M^*)$.
 The differential~$d$ in the complex
$\Hom^\bu_{\biR^\cu}(\biL^\cu,\biM^\cu)$ is defined by
the usual formula $d(f)(z)=d_M(f(z))-(-1)^{|f|}f(d_L(z))$
for all $f\in\Hom_{R^*}^{|f|}(L^*,M^*)$ and $z\in L^{|z|}$.

 By the definition, the composition of morphisms and the identity
morphisms in the DG\+category $\biR^\cu\bModl$ agree with those in
the graded category $R^*\cModl$.
 So $(\biR^\cu\bModl)^*$ is the full graded subcategory in $R^*\cModl$
whose objects are all the graded left $R^*$\+modules \emph{that
admit a structure of CDG\+module over~$\biR^\cu$}, and similarly,
$(\biR^\cu\bModl)^0$ is the full additive subcategory in $R^*\sModl$
whose objects are all the graded $R^*$\+modules \emph{that
admit a structure of CDG\+module over~$\biR^\cu$}.
 Not all the graded $R^*$\+modules can be endowed with such
a structure in general; see counterexamples
in~\cite[Examples~3.2 and~3.3]{Pedg}.

 All shifts, twists, and infinite products and coproducts (hence
also all finite direct sums and cones) exist in the DG\+category
$\biR^\cu\bModl$.
 The category $\sZ^0(\biR^\cu\bModl)$ of CDG\+modules over $\biR^\cu$ with
closed morphisms of degree~$0$ is abelian.
 In fact, $\sZ^0(\biR^\cu\bModl)$ is the abelian category of graded modules
over a graded ring denoted by $\widehat\biR^*$, as we will discuss
below in Section~\ref{almost-involution-CDG-modules-subsecn}.

\subsection{Example: DG-category of factorizations}
\label{DG-category-of-factorizations-defined-subsecn}
 The concept of a \emph{matrix factorization} of a polynomial, or
more generally, of a section of a line bundle on a scheme goes back
to the paper~\cite{Eisenbud}; see~\cite{Or,EP} for the nonaffine case.
 The following abstract category-theoretic version was suggested in
the papers~\cite[Section~6 and Appendix~A]{Ef} and~\cite{BDFIK}
(see~\cite[Section~2.5]{Pedg} for an even more general approach).

 We follow~\cite[Remark~2.7]{Pedg} (restricting ourselves to the case
of the grading group $\Gamma=\boZ$).
 Let $\sE$ be a preadditive category and $\Delta\:\sE\rarrow\sE$ be
an autoequivalence.
 Then the preadditive category $\sP(\sE,\Delta)$ of
\emph{$2$\+$\Delta$-periodic objects} in $\sE$ is constructed
as follows.
 An object $X^\circ\in\sP(\sE,\Delta)$ is a collection of objects
$(X^n\in\sE)_{n\in\boZ}$ endowed with isomorphisms
$\delta_X^{n+2,n}\:\Delta(X^n)\overset\simeq\rarrow X^{n+2}$ defined
for all $n\in\boZ$.
 The collection of all objects $(X^n\in\sE)_{n\in\boZ}$ with
the periodicity isomorphisms~$\delta_X^{n+2,n}$ forgotten defines
the underlying graded object $X^*\in\sG(\sE)$ of
a $2$\+$\Delta$-periodic object $X^\circ\in\sP(\sE,\Delta)$.

 For any two objects $X^\circ$ and $Y^\circ\in\sP(\sE,\Delta)$,
the abelian group $\Hom_{\sP(\sE,\Delta)}(X^\circ,Y^\circ)$ is
defined as the subgroup in $\Hom_{\sG(\sE)}(X^*,Y^*)$ consisting of
all the morphisms $(f_n\:X^n\to Y^n)_{n\in\boZ}$ satisfying
the equations $\delta_Y^{n+2,n}\circ\Delta(f_n)=
f_{n+2}\circ\delta_X^{n+2,n}$ for all $n\in\boZ$.
 The category $\sP(\sE,\Delta)$ is naturally equivalent to
the Cartesian product $\sE\times\sE$ of two copies of
the category $\sE$, the equivalence being provided by the functor
taking a $2$\+$\Delta$-periodic object $X^\circ\in\sP(\sE,\Delta)$
to the pair of objects $(X^0,X^1)\in\sE\times\sE$.

 The shift functor on $\sP(\sE,\Delta)$ is defined by the rule
$X^\circ[i]^n=X^{n+i}$ and $\delta^{n,n+2}_{X[i]}=\delta_X^{n+i,n+i+2}$
for all $i$, $n\in\boZ$.
 The graded category $\cP(\sE,\Delta)$ is constructed as follows.
 The objects of $\cP(\sE,\Delta)$ are the objects of $\sP(\sE,\Delta)$,
i.~e., the $2$\+$\Delta$-periodic objects in~$\sE$.
 The graded abelian group of morphisms $\Hom_{\cP(\sE,\Delta)}^*
(X^\circ,Y^\circ)$ has the grading components $\Hom^i_{\cP(\sE,\Delta)}
(X^\circ,Y^\circ)=\Hom_{\sP(\sE,\Delta)}(X^\circ,Y^\circ[i])$ for
all $i\in\boZ$.
 By the definition, one has $(\cP(\sE,\Delta))^0=\sP(\sE,\Delta)$.
 The composition of morphisms in the graded category $\cP(\sE,\Delta)$
agrees with the one in the graded category~$\cG(\sE)$.
 \emph{Unlike} the preadditive category $\sP(\sE,\Delta)$, the graded
category $\cP(\sE,\Delta)$ is not determined by the category $\sE$
alone; it depends on the autoequivalence~$\Delta$.

 A \emph{potential}~$v$ for an autoequivalence $\Delta\:\sE\rarrow\sE$
is a natural transformation $v\:\Id_\sE\rarrow\Delta$ satisfying
the equation $v_{\Delta(E)}=\Delta(v_E)$ for all $E\in\sE$.
 Given a potential~$v$, the DG\+category $\bF(\sE,\Delta,v)$ of 
\emph{factorizations of~$v$} in $\sE$ is constructed as follows.
 A factorization $\biN^\cu=(N^\circ,d_N)$ (in the sense of~\cite{BDFIK})
is a $2$\+$\Delta$-periodic object endowed with a homogeneous
endomorphism $d_N\in\Hom_{\cP(\sE,\Delta)}^1(N,N)$ such that $d_N^2=
\delta_Nv_N$, that is, for every $n\in\boZ$, the composition
$d_{N,n+1}d_{N,n}\:N^{n}\rarrow N^{n+1}\rarrow N^{n+2}$ is equal to
the composition $\delta_N^{n,n+2}v_{N^n}\:N^n\rarrow\Delta(N^n)
\rarrow N^{n+2}$.
 Given two factorizations $L^\cu=(L^\circ,d_L)$ and $M^\cu=
(M^\circ,d_M)$ of the same potential~$v$, the underlying graded abelian
group of the complex $\Hom^\bu_{\bF(\sE,\Delta,v)}(L^\cu,M^\cu)$
is the graded abelian group $\Hom^*_{\cP(\sE,\Delta)}(L^\circ,M^\circ)$.
 The differential~$d$ in the complex
$\Hom^\bu_{\bF(\sE,\Delta,v)}(L^\cu,M^\cu)$ is given by the usual
formula $d(f)=d_M\circ f-(-1)^{|f|}f\circ d_L$.
 One can check that $d^2=0$.

 By the definition, the composition of morphisms and the identity
morphisms in the DG\+category $\bF(\sE,\Delta,v)$ agree with those
in the graded category $\cP(\sE,\Delta)$.
 So $\bF(\sE,\Delta,v)^*$ is the full graded subcategory in
$\cP(\sE,\Delta)$ whose objects are all the $2$\+$\Delta$-periodic
objects in~$\sE$ \emph{that admit a differential defining a structure
of factorization of~$v$}, and similarly, $\bF(\sE,\Delta,v)^0$
is the full preadditive subcategory in $\sP(\sE,\Delta)$ whose objects
are all the $2$\+$\Delta$-periodic objects in~$\sE$ \emph{that admit
a differential defining a structure of factorization of~$v$}.
 Not all the $2$\+$\Delta$-periodic objects in~$\sE$ can be endowed
with such a structure in general, as illustrated
by~\cite[Examples~3.2]{Pedg} interpreted in the context of
factorizations; see~\cite[Example~3.19]{Pedg} for a discussion.

 All shifts and twists exist in the DG\+category $\bF(\sE,\Delta,v)$.
 When the category $\sE$ is additive, so is the DG\+category
$\bF(\sE,\Delta,v)$; hence this DG\+category also has cones.
 When the category $\sE$ has infinite (co)products, so does
the DG\+category $\bF(\sE,\Delta,v)$.

 The notation for the category of closed morphisms in
$\bF(\sE,\Delta,v)$ is $\sF(\sE,\Delta,v)=\sZ^0(\bF(\sE,\Delta,v))$.
 This is the preadditive category of factorizations of~$v$ in~$\sE$.
 The notation for the homotopy category is $\sK(\sE,\Delta,v)=
\sH^0(\bF(\sE,\Delta,v))$.

\Section{The Almost Involution on DG-Categories}
\label{almost-involution-secn}

 This section is an extraction from~\cite[Section~3]{Pedg}.
 Most proofs are omitted and replaced with references to~\cite{Pedg}.

\subsection{The DG-category $\bE^\bec$}
\label{bec-construction-subsecn}
 Let $\bE$ be a DG\+category.
 The DG\+category $\bE^\bec$ is constructed as follows.
 The objects of $\bE^\bec$ are pairs $X^\bec=(X,\sigma_X)$, where
$X$ is an object of $\bE$ and $\sigma_X\in\Hom^{-1}_\bE(X,X)$ is
an endomorphism of degree~$-1$ such that $d(\sigma_X)=\id_X$ and
$\sigma_X^2=0$.
 In other words, $\sigma_X$~is a contracting homotopy for the object
$X\in\bE$ satisfying the additional condition that the square
of~$\sigma_X$ vanishes in $\Hom^{-2}_\bE(X,X)$.

 The construction of the complex of morphisms
$\Hom_{\bE^\bec}^\bu(X^\bec,Y^\bec)$ for a pair of objects $X^\bec$,
$Y^\bec\in\bE^\bec$ involves a change of the sign of the cohomological
grading.
 For every $n\in\boZ$, the group $\Hom_{\bE^\bec}^n(X^\bec,Y^\bec)$
consists of all the \emph{cocycles} in the group
$\Hom_\bE^{-n}(X,Y)$, i.~e., all elements $f\in\Hom_\bE^{-n}(X,Y)$
annihilated by the differential $d_{-n}\:\Hom_\bE^{-n}(X,Y)\rarrow
\Hom_\bE^{-n+1}(X,Y)$.
 The differential~$d^\bec$ in the complex $\Hom_{\bE^\bec}^\bu
(X^\bec,Y^\bec)$ is defined as the commutator with the contracting
homotopies~$\sigma$, that is, $d^\bec(f)=\sigma_Yf-(-1)^nf\sigma_X$.
 The composition of morphisms in the DG\+category $\bE^\bec$ is
induced by the composition of morphisms in $\bE$ in the obvious way,
and the identity morphisms in $\bE^\bec$ are the identity morphisms
in $\bE$, i.~e., $\id_{X^\bec}=\id_X$.

 All twists exist in the DG\+category~$\bE^\bec$.
 Given an object $X^\bec=(X,\sigma_X)\in\bE^\bec$ and
a Maurer--Cartan cochain $a\in\Hom^1_{\bE^\bec}(X^\bec,X^\bec)\subset
\Hom^{-1}_\bE(X,X)$, the twisted object $X^\bec(a)$ is constructed
as $X^\bec(a)=(X,\sigma_X+a)$.
 Finite direct sums, as well as infinite products and/or coproducts,
exist in the DG\+category $\bE^\bec$ whenever they exist in
the DG\+category~$\bE$, and are constructed in the obvious way.
 Shifts exists in $\bE^\bec$ whenever they exist in~$\bE$; here one
has to observe the grading sign change: given an object $X^\bec=
(X,\sigma_X)\in\bE^\bec$, the object $X^\bec[1]$ is given by
the rule $X^\bec[1]=(X[-1],-\sigma_X[-1])$.
 Consequently, the DG\+category $\bE^\bec$ has cones whenever
the DG\+category $\bE$ has shifts and finite direct
sums~\cite[Section~3.2]{Pedg}.

 To be more precise, let us spell out the sign rule (or lack
thereof) in the definition of the morphism
$\sigma_X[-1]\in\Hom_\bE^{-1}(X[-1],X[-1])$ in the formula
for $X^\bec[1]$ above.
 Given an object $A\in\bE$, the object $A[-1]\in\bE$ comes together
with a pair of closed morphisms $s_A\in\Hom_\bE^1(A,A[-1])$
and $t_A\in\Hom_\bE^{-1}(A[-1],A)$ such that $t_As_A=\id_A$
and $s_At_A=\id_{A[-1]}$.
 For any morphism $f\in\Hom^n_\bE(A,B)$ in $\bE$, \,$n\in\boZ$,
we put $f[-1]=s_Bft_A\in\Hom^n_\bE(A[-1],B([-1])$.

\subsection{The functors $\Phi$ and $\Psi^{\pm}$}
\label{Phi-and-Psi-subsecn}
 There is a doubly infinite ladder of adjoint functors acting between
the additive categories $\sZ^0(\bE)$ and $\sZ^0(\bE^\bec)$ of closed
morphisms of degree~$0$ in the DG\+categories $\bE$ and~$\bE^\bec$.

 The functor $\Phi_\bE\:\sZ^0(\bE)\rarrow\sZ^0(\bE^\bec)$ is
\emph{interpreted} as the forgetful functor assigning to a complex,
CDG\+module, or factorization its underlying graded object/module,
while the functors $\Psi^+_\bE$ and $\Psi^-_\bE\:\sZ^0(\bE^\bec)
\rarrow\sZ^0(\bE)$ left and right adjoint to $\Phi_\bE$ are interpreted
as assigning to a graded module the CDG\+module (co)freely (co)generated
by it (see the discussion of examples in
Sections~\ref{almost-involution-complexes-subsecn}--%
\ref{almost-involution-factorizations-subsecn} below).
 The constructions of the three functors do not immediately suggest
this somewhat counterintuitive interpretation, as it is the functor
$\Psi^+_\bE$ that is \emph{constructed} as a forgetful functor,
while the construction of the functor $\Phi_\bE$ is more complicated.

 In fact, up to adjoining twists and direct summands to
the DG\+categories, the roles of the three functors are completely
symmetric, as we will see in Section~\ref{DG-functor-becbec-subsecn}.

 Let $\bE$ be a DG\+category with shifts and cones.
 Let $A\in\bE$ be an object, and let $L$ be the cone of the identity
endomorphism of the object $A[-1]\in\bE$.
 The object $L$ is defined in terms of its structure morphisms
$\iota\in\Hom_\bE^1(A,L)$, \, $\pi\in\Hom^0_\bE(L,A)$, \
$\pi'\in\Hom_\bE^{-1}(L,A)$, and $\iota'\in\Hom_\bE^0(A,L)$
satisfying the equations
\begin{gather*}
 \pi'\iota'=0=\pi\iota, \quad
 \pi'\iota=\id_A=\pi\iota', \quad
 \iota\pi'+\iota'\pi=\id_L, \\
 d(\iota)=0=d(\pi), \quad d(\pi')=\pi, \quad d(\iota')=\iota.
\end{gather*}

 Put $\sigma_L=\iota'\pi'\in\Hom_\bE^{-1}(L,L)$.
 Then $L^\bec=(L,\sigma_L)$ is an object of
the DG\+cat\-e\-gory~$\bE^\bec$.
 We put $\Phi_\bE(A)=L^\bec$.
 The action of the functor $\Phi_\bE$ on morphisms in $\sZ^0(\bE)$
(i.~e., on the closed morphisms of degree~$0$ in~$\bE$) is defined
in the obvious way.

 The functors $\Psi^+_\bE$ and $\Psi^-_\bE$ are defined by the rules
$\Psi^+(X^\bec)=X$ and $\Psi^-(X^\bec)=X[1]$ for all objects
$X^\bec=(X,\sigma_X)\in\bE^\bec$.
 It is explained in~\cite[proof of Lemma~3.4]{Pedg} that the functor
$\Psi^+$ is left adjoint to $\Phi$, while the functor $\Psi^-$ is
right adjoint to~$\Phi$.

 All the three functors $\Psi^+$, $\Psi^-$, and $\Phi$ are faithful.
 They are also conservative~\cite[Lemma~3.12]{Pedg}.
 The functors $\Phi$ and $\Psi^\pm$ transform the shift functors~$[n]$,
\,$n\in\boZ$ on the DG\+categories $\bE$ and $\bE^\bec$ into
the inverse shift functors~\cite[Lemma~3.11]{Pedg}:
$$
 \Phi\circ[n]=[-n]\circ\Phi, \quad
 \Psi^+\circ[n]=[-n]\circ\Psi^+, \quad
 \Psi^-\circ[n]=[-n]\circ\Psi^-.
$$

 The compositions of the functors $\Phi$ and $\Psi^\pm$ are computable
as follows.
 For any DG\+category $\bE$ with shifts and cones, denote by $\Xi=
\Xi_\bE\:\sZ^0(\bE)\rarrow\sZ^0(\bE)$ the additive functor taking
an object $A\in\bE$ to the object $\Xi(A)=\cone(\id_A[-1])$ (and
acting on the morphisms in the obvious way).
 Then there are natural isomorphisms of
additive functors~\cite[Lemma~3.8]{Pedg}
$$
 \Psi^+_\bE\circ\Phi_\bE=\Xi_\bE
$$
and
$$
 \Phi_\bE\circ\Psi^-_\bE=\Xi_{\bE^\bec}.
$$

\subsection{The functors $\widetilde\Phi$ and $\widetilde\Psi^{\pm}$}
\label{tilde-Phi-and-Psi-subsecn}
 The additive functor $\Psi^+\:\sZ^0(\bE^\bec)\rarrow\sZ^0(\bE)$ can be
naturally extended to a fully faithful additive functor
$$
 \widetilde\Psi^+\:(\bE^\bec)^0\lrarrow\sZ^0(\bE).
$$
 The action of the functor $\widetilde\Psi^+$ on the objects is given
by the obvious rule $\widetilde\Psi^+(X^\bec)=X$ for any $X^\bec=
(X,\sigma_X)\in\bE^\bec$.
 The functor $\widetilde\Psi^+$ acts on the morphisms by the natural
isomorphism
$$
 \Hom_{(\bE^\bec)^0}(X^\bec,Y^\bec)=\Hom_{\sZ^0(\bE)}(X,Y)
 \quad\text{for all $X^\bec=(X,\sigma_X)$ and
 $Y^\bec=(Y,\sigma_Y)\in\bE^\bec$},
$$
which is a part of the definition of the DG\+category~$\bE^\bec$.

 Similarly, the additive functor $\Psi^-\:\sZ^0(\bE^\bec)\rarrow
\sZ^0(\bE)$ can be naturally extended to a fully faithful additive
functor
$$
 \widetilde\Psi^-\:(\bE^\bec)^0\lrarrow\sZ^0(\bE).
$$
 The functor $\widetilde\Psi^-$ is constructed as $\widetilde\Psi^-=
\widetilde\Psi^+[1]$.
 
 Moreover, the additive functor $\Phi\:\sZ^0(\bE)\rarrow\sZ^0(\bE^\bec)$
can be naturally extended to a fully faithful additive functor
$$
 \widetilde\Phi\:\bE^0\lrarrow\sZ^0(\bE^\bec).
$$

 On objects, the functor $\widetilde\Phi$ is defined by the obvious
rule $\widetilde\Phi(A)=\Phi(A)=\cone(\id_A[-1])$ for all
$A\in\bE$.
 To construct the action of the functor $\widetilde\Phi$ on morphisms,
suppose given a morphism $f\in\Hom^0_\bE(A,B)$ for some objects
$A$, $B\in\bE$.
 Put $L=\cone(\id_A[-1])$ and $M=\cone(\id_B[-1])$.
 Let $\iota_A$, $\pi_A$, $\iota'_A$, $\pi'_A$ and $\iota_B$, $\pi_B$,
$\iota'_B$, $\pi'_B$ be the related morphisms
(as in Section~\ref{Phi-and-Psi-subsecn}).
 Then we put
$$
 \widetilde\Phi(f)=g=\iota'_Bf\pi_A+\iota_Bf\pi'_A+\iota'_Bd(f)\pi'_A
 \,\in\,\Hom_{\sZ^0(\bE^\bec)}(L^\bec,M^\bec)\subset\Hom_\bE^0(L,M),
$$
where $L^\bec=(L,\sigma_L)=\Phi(A)$ and $M^\bec=(M,\sigma_M)=\Phi(B)$.
 One has to check that $d(g)=0$, \,$d^\bec(g)=0$, \
$\widetilde\Phi(\id_A)=\id_{\Phi(A)}$, and
$\widetilde\Phi(f_1\circ f_2)=
\widetilde\Phi(f_1)\circ\widetilde\Phi(f_2)$ for any pair of composable
morphisms $f_1$, $f_2$ in~$\bE^0$.
 Then one also has to check that the assignment $f\longmapsto g$ is
an isomorphism of the Hom groups
$$
 \Hom_\bE^0(A,B)\simeq\Hom_{\sZ^0(\bE)}(L^\bec,M^\bec).
$$
 This is the result of~\cite[Lemma~3.9]{Pedg}.

 It follows from the existence of the functors $\widetilde\Phi$ and
$\widetilde\Psi^\pm$ that the functors $\Phi$ and $\Psi^\pm$ transform
twists into isomorphisms~\cite[Lemma~3.11]{Pedg}.

 Notice that the fully faithful functors $\widetilde\Phi$ and
$\widetilde\Psi^\pm$ usually do \emph{not} have either left or right
adjoints.
 Indeed, these functors are embeddings of pretty badly behaved full
subcategories, in general~\cite[Examples~3.2 and~3.3]{Pedg}.

\subsection{The DG-functor $\bec\bec$}
\label{DG-functor-becbec-subsecn}
 For any DG\+category $\bE$ with shifts and cones, there is a naturally
defined fully faithful DG\+functor $\bec\bec\:\bE\rarrow\bE^{\bec\bec}$.
 In order to construct the DG\+functor~$\bec\bec$, let us first
describe the DG\+category $\bE^{\bec\bec}$ more explicitly.

 The objects of $\bE^{\bec\bec}$ are triples $W^{\bec\bec}=
(W,\sigma,\tau)$, where $W$ is an object of $\bE$ endowed with
two endomorphisms $\sigma\in\Hom_\bE^{-1}(W,W)$ and
$\tau\in\Hom_\bE^1(W,W)$.
 The morphisms $\sigma$ and~$\tau$ must satisfy the equations
$$
 \sigma^2=0=\tau^2, \quad \sigma\tau+\tau\sigma=\id_W, \quad
 d(\sigma)=\id_W, \quad d(\tau)=0.
$$
 Here $W^\bec=(W,\sigma)$ is an object of $\bE^\bec$, and
$\tau\in\Hom^{-1}_{\bE^\bec}(W^\bec,W^\bec)$ is an endomorphism with
$d^\bec(\tau)=\sigma\tau+\tau\sigma=\id_W$ and $\tau^2=0$.

 Given two objects $U^{\bec\bec}=(U,\sigma_U,\tau_U)$ and
$V^{\bec\bec}=(V,\sigma_V,\tau_U)\in\bE^{\bec\bec}$, the complex
of morphisms $\Hom_{\bE^{\bec\bec}}^\bu(U^{\bec\bec},V^{\bec\bec})$
is described as follows.
 For every $n\in\boZ$, the group $\Hom_{\bE^{\bec\bec}}^n
(U^{\bec\bec},V^{\bec\bec})$ is a subgroup in $\Hom_\bE^n(U,V)$
consisting of all the morphisms $f\:U\rarrow V$ of degree~$n$ in $\bE$
such that $d(f)=0$ and $d^\bec(f)=\sigma_V f-(-1)^nf\sigma_U=0$.
 The differential~$d^{\bec\bec}$ on $\Hom_{\bE^{\bec\bec}}^\bu
(U^{\bec\bec},V^{\bec\bec})$ is given by the rule
$d^{\bec\bec}(f)=\tau_V f-(-1)^nf\tau_U$.
 The composition of morphisms in $\bE^{\bec\bec}$ is induced by
the composition of morphisms in $\bE$ in the obvious way.

 The DG\+functor~$\bec\bec$ assigns to an object $A\in\bE$
the object $L=\cone(\id_A[-1])$ endowed with the endomorphisms
$\sigma_L=\iota'\pi'\in\Hom_\bE^{-1}(L,L)$ and
$\tau_L=\iota\pi\in\Hom_\bE^1(L,L)$, in the notation from
Section~\ref{Phi-and-Psi-subsecn}.
 So $\bec\bec(A)=(L,\sigma_L,\tau_L)$.

 To define the action of the DG\+functor~$\bec\bec$ on morphisms,
consider two objects $A$ and $B\in\bE$.
 Put $L=\cone(\id_A[-1])$ and $M=\cone(\id_B[-1])$, and denote
by $\iota_A$, $\pi_A$, $\iota'_A$, $\pi'_A$ and $\iota_B$, $\pi_B$,
$\iota'_B$, $\pi'_B$ the related morphisms as in
Sections~\ref{Phi-and-Psi-subsecn}--\ref{tilde-Phi-and-Psi-subsecn}.
 Then the DG\+functor~$\bec\bec$ assigns to a morphism
$f\in\Hom_\bE^n(A,B)$ the morphism $g=\bec\bec(f)\in
\Hom_{\bE^{\bec\bec}}^n(\bec\bec(A),\bec\bec(B))\subset
\Hom_\bE^n(L,M)$ given by the formula
$$
 \bec\bec(f)=g=(-1)^n\iota'_Bf\pi_A+\iota_Bf\pi'_A+\iota'_Bd(f)\pi'_A.
$$
 One can check that $d(g)=0$, \,$d^\bec(g)=0$, \, $d^{\bec\bec}(g)
=\bec\bec(df)$, and $\bec\bec(f_1\circ f_2)=\bec\bec(f_1)\circ
\bec\bec(f_2)$ for any composable pair of morphisms $f_1$, $f_2$
in~$\bE$.
 It is a result of~\cite[Proposition~3.5]{Pedg} that
the DG\+functor~$\bec\bec$ is fully faithful.

 An additive DG\+category $\bE$ is said to be \emph{idempotent-complete}
if the additive category $\sZ^0(\bE)$ is idempotent-complete (i.~e.,
any idempotent endomorphism in $\sZ^0(\bE)$ arises from a direct sum
decomposition).
 Given an additive DG\+category $\bE$, the additive DG\+category
$\bE^\bec$ is idempotent-complete whenever the additive DG\+category
$\bE$~is.
 If $\bE$ is an idempotent-complete additive DG\+category with twists,
then the DG\+functor~$\bec\bec$ is an equivalence of DG\+categories.
 Generally speaking, the DG\+cat\-e\-gory $\bE^{\bec\bec}$ is obtained
from the DG\+category $\bE$ by adjoining all twists and some of their
direct summands~\cite[Proposition~3.14]{Pedg}.

 For any DG\+category $\bE$ with shifts and cones, there are natural
isomorphisms of additive functors
$$
 \Psi^+_{\bE^\bec}\circ\sZ^0(\bec\bec)\simeq\Phi_\bE
$$
and
$$
 \sZ^0(\bec\bec)\circ\Psi^-_\bE\simeq\Phi_{\bE^\bec}
$$
showing that the ``almost involution'' $\bE\longmapsto\bE^{\bec\bec}$
interchanges the roles of the functors $\Phi$ and~$\Psi^\pm$
\cite[Lemmas~3.6 and~3.7]{Pedg}.
 Moreover, one has~\cite[Remark~3.10]{Pedg}
$$
 \widetilde\Psi^+_{\bE^\bec}\circ(\bec\bec)^0\simeq
 \widetilde\Phi_\bE
$$
and
$$
 \sZ^0(\bec\bec)\circ\widetilde\Psi_\bE^-\simeq
 \widetilde\Phi_{\bE^\bec}.
$$

\subsection{Example: DG-category of complexes}
\label{almost-involution-complexes-subsecn}
 Let $\sE$ be an additive category.
 Then the DG\+category $\bC(\sE)$ of complexes in~$\sE$
(see Section~\ref{DG-category-of-complexes-defined-subsecn})
is an additive DG\+category with shifts, twists, and cones.
 Recall the notation $\sG(\sE)=\sE^\boZ$ for the additive category
of graded objects in~$\sE$ and $\sC(\sE)=\sZ^0(\bC(\sE))$ for
the additive category of complexes in~$\sE$.

 The forgetful functor $X^\bu\longmapsto X^\bu{}^\#\:\sC(\sE)\rarrow
\sG(\sE)$, assigning to a complex $X^\bu$ in $\sE$ its underlying
graded object $X^*=X^\bu{}^\#$, has adjoints on both sides.
 The functor $G^+\:\sG(\sE)\rarrow\sC(\sE)$ left adjoint to
$X^\bu\longmapsto X^\bu{}^\#$ assigns to a graded object
$E^*\in\sG(\sE)$ the complex $G^+(E^*)$ freely generated by~$E^*$.
 Explicitly, one has $G^+(E^*)^n=E^n\oplus E^{n-1}$ for all $n\in\boZ$,
and the differential $d_{G^+,n}\:E^n\oplus E^{n-1}\rarrow E^{n+1}
\oplus E^n$ is given by the $2\times2$ matrix of morphisms whose only
nonzero entry is the identity map $E^n\rarrow E^n$.

 The functor $G^-\:\sG(\sE)\rarrow\sC(\sE)$ left adjoint to
$X^\bu\longmapsto X^\bu{}^\#$ assigns to a graded object
$E^*\in\sG(\sE)$ the complex $G^-(E^*)$ cofreely cogenerated by~$E^*$.
 Explicitly, $G^-(E^*)^n=E^n\oplus E^{n+1}$ for all $n\in\boZ$,
and the differential $d_{G^-,n}\:E^n\oplus E^{n+1}\rarrow E^{n+1}
\oplus E^{n+2}$ is given by the $2\times2$ matrix of morphisms whose
only nonzero entry is the identity map $E^{n+1}\rarrow E^{n+1}$.
 So one has $G^-(E^*)\simeq G^+(E^*)[1]$.

 Let us assume that the additive category $\sE$ is idempotent-complete;
then so is the DG\+category $\bC(\sE)$.
 In this case, there is a natural equivalence of additive categories
$$
 \Upsilon=\Upsilon_\sE\:\sG(\sE)\simeq\sZ^0(\bC(\sE)^\bec)
$$
\cite[Example~3.16]{Pedg}.
 The equivalence of categories $\Upsilon$ forms the following
commutative diagrams of additive functors with the functors $\#$,
$\Phi$, $\widetilde\Phi$, $G^\pm$, and~$\Psi^\pm$:
\begin{equation} \label{complexes-Phi-diagram}
\begin{gathered}
 \xymatrix{
  \sC(\sE) \ar@{=}[rr] \ar[d]_{\#}
  && \sZ^0(\bC(\sE)) \ar[d]^{\Phi_{\bC(\sE)}} \\
  \sG(\sE) \ar@<-0.4ex>[rr]_-{\Upsilon_\sE}
  && \sZ^0(\bC(\sE)^\bec) \ar@<-0.4ex>@{-}[ll]
 }
\end{gathered}
\end{equation}
\begin{equation} \label{complexes-tilde-Phi-diagram}
\begin{gathered}
 \xymatrix{
  & \bC(\sE)^0 \ar@<-0.4ex>[ld]
  \ar@<0.4ex>[rd]^-{\widetilde\Phi_{\bC(\sE)}} \\
  \sG(\sE) \ar@<-0.4ex>[rr]_-{\Upsilon_\sE} \ar@<-0.4ex>@{-}[ru]
  && \sZ^0(\bC(\sE)^\bec) \ar@<-0.4ex>@{-}[ll] \ar@<0.4ex>@{-}[lu]
 }
\end{gathered}
\end{equation}
and
\begin{equation} \label{complexes-Psi-diagram}
\begin{gathered}
 \xymatrix{
  \sG(\sE) \ar@<0.4ex>[rr]^-{\Upsilon_\sE} \ar[d]_{G^\pm}
  && \sZ^0(\bC(\sE)^\bec) \ar[d]^{\Psi_{\bC(\sE)}^\pm}
  \ar@<0.4ex>@{-}[ll]
  \\ \sC(\sE)\ar@{=}[rr] && \sZ^0(\bC(\sE))
 }
\end{gathered}
\end{equation}
 Here the upper horizontal double line
in~\eqref{complexes-Phi-diagram} and the lower horizontal double
line in~\eqref{complexes-Psi-diagram} is essentially the definition
of the additive category of complexes~$\sC(\sE)$.
 The leftmost diagonal double arrow
in~\eqref{complexes-tilde-Phi-diagram} is the obvious equivalence
of additive categories mentioned in
Section~\ref{DG-category-of-complexes-defined-subsecn}.
 The leftmost vertical arrow in~\eqref{complexes-Phi-diagram}
is the forgetful functor $X^\bu\longmapsto X^\bu{}^\#=X^*$.
 There are, actually, two commutative diagrams depicted
in~\eqref{complexes-Psi-diagram}: one has to choose either
$G^+$ for the leftmost vertical arrow and $\Psi^+$ for the rightmost
one, or $G^-$ for the leftmost vertical arrow and $\Psi^-$ for
the rightmost one.
 All the double lines and double arrows are category equivalences;
all the ordinary arrows are faithful functors~\cite[Example~3.16]{Pedg}.

 All objects of the DG\+category $\bC(\sE)^\bec$ are
contractible~\cite[Example~3.16]{Pedg}; so the DG\+category
$\bC(\sE)^\bec$ is quasi-equivalent to the zero DG\+category.
 Still the DG\+category $\bC(\sE)^{\bec\bec}$ is equivalent to
the DG\+category $\bC(\sE)$.
 So the passage from a DG\+category $\bE$ to the DG\+category
$\bE^\bec$ does \emph{not} preserve quasi-equivalences.

\subsection{Example: DG-category of CDG-modules}
\label{almost-involution-CDG-modules-subsecn}
 Let $\biR^\cu=(R^*,d,h)$ be a CDG\+ring.
 Then the DG\+category $\biR^\cu\bModl$ of left CDG\+modules
over $\biR^\cu$
(see Section~\ref{DG-category-of-CDG-modules-defined-subsecn})
is an additive DG\+category with shifts, twists, and cones
(as well as infinite products and coproducts).
 Recall the notation $R^*\sModl$ for the abelian category of graded
left modules over the graded ring~$R^*$.

 The forgetful functor $\biM^\cu\longmapsto\biM^\cu{}^\#\:
\sZ^0(\biR^\cu\bModl)\rarrow R^*\sModl$, assigning to a CDG\+module
$\biM^\cu=(M^*,d_M)$ over $\biR^\cu$ its underlying graded $R^*$\+module
$M^*=\biM^\cu{}^\#$, has adjoints on both sides.
 The functor $G^+\:R^*\sModl\rarrow\sZ^0(\biR^\cu\bModl)$ left adjoint
to $\biM^\cu\longmapsto\biM^\cu{}^\#$ assigns to a graded $R^*$\+module
$M^*$ the CDG\+module $G^+(M^*)$ freely generated by~$M^*$.
 The functor $G^-\:R^*\sModl\rarrow\sZ^0(\biR^\cu\bModl)$ right adjoint
to $\biM^\cu\longmapsto\biM^\cu{}^\#$ assigns to a graded $R^*$\+module
$M^*$ the CDG\+module $G^-(M^*)$ cofreely cogenerated by~$M^*$.

 Explicit constructions of the functors $G^+$ and $G^-$ for
CDG\+modules over a CDG\+ring can be found in~\cite[proof of
Theorem~3.6]{Pkoszul} or~\cite[Proposition~3.1]{Pedg}
(see also~\cite[Proposition~1.3.2]{Bec}).
 For any graded left $R^*$\+module $M^*$, there is a natural
isomorphism of CDG\+modules $G^-(M)\simeq G^+(M)[1]$
\,\cite[Proposition~3.1(c)]{Pedg}.

 The abelian category $\sZ^0(\biR^\cu\bModl)$ of left CDG\+modules over
a CDG\+ring $\biR^\cu=(R^*,d,h)$ can be described as the category of graded
modules over the following graded ring~$R^*[\delta]$.
 The ring $R^*[\delta]$ is obtained by adjoining to the graded ring
$R^*$ a new element~$\delta$, which is homogeneous of degree~$1$ and
subject to the relations
\begin{itemize}
 \item $\delta r-(-1)^{|r|}r\delta=d(r)$ for all $r\in R^{|r|}$,
\,$|r|\in\boZ$;
 \item $\delta^2=h$.
\end{itemize}

 The ring $R^*$ is a graded subring in~$R^*[\delta]$.
 Viewed as either a graded left $R^*$\+module or a graded right
$R^*$\+module, the graded ring $R^*[\delta]$ is a free graded $R^*$\+module
with two generators~$1$ and~$\delta$.
 We refer to~\cite[Section~3.1]{Pedg} or~\cite[Section~4.2]{Prel} for
a further discussion of this construction.
 Given a left CDG\+module $\biM^\cu=(M^*,d_M)$ over $\biR^\cu$, the action
of the graded ring $R^*$ in $M^*$ is extended to an action of
the graded ring $R^*[\delta]$ by the obvious rule $\delta\cdot x=d_M(x)$
for all $x\in M^*$.

 The graded ring $R^*[\delta]$ is endowed with an odd derivation
$\dd=\dd/\dd\delta$ of degree~$-1$ defined by the rules $\dd(r)=0$
for all $r\in R^*$ and $\dd(\delta)=1$.
 In order to make~$\dd$ a cohomological differential, we follow
the notation in~\cite{Pedg} and denote by $\widehat\biR^*$
the graded ring $R^*[\delta]$ with the sign of the grading changed:
$\widehat\biR^n=R^*[\delta]^{-n}$ for all $n\in\boZ$.
 Then the graded ring $\widehat\biR^*$ with the differential~$\dd$
becomes a DG\+ring, which we will denote by $\widehat\biR^\bu
=(\widehat\biR^*,\dd)$.
 Notice that $\widehat\biR^\bu$ is an \emph{acyclic} DG\+ring:
one has $H_\dd(\widehat\biR^\bu)=0$, since the unit element vanishes
in the cohomology ring $H_\dd(\widehat\biR^\bu)$.

 The DG\+category $(\biR^\cu\bModl)^\bec$ is naturally equivalent to
the DG\+category $\widehat\biR^\bu\bModl$ of DG\+modules over
the DG\+ring $\widehat\biR^\bu$.
 Furthermore, there is a natural equivalence of abelian categories
$$
 \Upsilon=\Upsilon_{\biR^\cu}\:R^*\sModl\simeq\sZ^0((\biR^\cu\bModl)^\bec)
$$
\cite[Example~3.17]{Pedg}.
 The equivalence of categories $\Upsilon_{\biR^\cu}$ transforms
the functor~$\#$ into the functor $\Phi$, the functors $G^\pm$ into
the functors $\Psi^\pm$, and the natural inclusion
$(\biR^\cu\bModl)^0\rarrow R^*\sModl$ into the functor~$\widetilde\Phi$.
 In other words, there are the following commutative diagrams of
additive functors:
\begin{equation} \label{cdg-modules-Phi-diagram}
\begin{gathered}
 \xymatrix{
  & \sZ^0(\biR^\cu\bModl) \ar[ld]_-{\#}
  \ar[rd]^-{\Phi_{\biR^\cu\bModl}} \\
  R^*\sModl \ar@<-0.4ex>[rr]_-{\Upsilon_{\biR^\cu}}
  && \sZ^0((\biR^\cu\bModl)^\bec) \ar@<-0.4ex>@{-}[ll]
 }
\end{gathered}
\end{equation}
\begin{equation} \label{cdg-modules-tilde-Phi-diagram}
\begin{gathered}
 \xymatrix{
  & (\biR^\cu\bModl)^0 \ar@{>->}[ld]
  \ar@{>->}[rd]^-{\widetilde\Phi_{\biR^\cu\bModl}} \\
  R^*\sModl \ar@<-0.4ex>[rr]_-{\Upsilon_{\biR^\cu}}
  && \sZ^0((\biR^\cu\bModl)^\bec) \ar@<-0.4ex>@{-}[ll]
 }
\end{gathered}
\end{equation}
and
\begin{equation} \label{cdg-modules-Psi-diagram}
\begin{gathered}
 \xymatrix{
  R^*\sModl \ar@<0.4ex>[rr]^-{\Upsilon_{\biR^\cu}} \ar[rd]_-{G^\pm}
  && \sZ^0((\biR^\cu\bModl)^\bec) \ar@<0.4ex>@{-}[ll]
  \ar[ld]^-{\Psi_{\biR^\cu\bModl}^\pm} \\ & \sZ^0(\biR^\cu\bModl)
 }
\end{gathered}
\end{equation}
 Here the leftmost diagonal arrow in~\eqref{cdg-modules-Phi-diagram}
is the forgetful functor $\biM^\cu\longmapsto\biM^\cu{}^\#=M^*$.
 The leftmost diagonal arrow in~\eqref{cdg-modules-tilde-Phi-diagram}
is the obvious fully faithful inclusion of additive categories
mentioned in Section~\ref{DG-category-of-CDG-modules-defined-subsecn}.
 There are, actually, two commutative diagrams depicted
in~\eqref{cdg-modules-Psi-diagram}: one has to choose either $G^+$
for the leftmost diagonal arrow and $\Psi^+$ for the rightmost one,
or $G^-$ for the leftmost diagonal arrow and $\Psi^-$ for
the rightmost one.
 The horizontal double arrow (which is the same on all the three
diagrams) is an abelian category equivalence.
 The arrows with tails are fully faithful functors; the ordinary
arrows are faithful exact functors~\cite[Example~3.17]{Pedg}.

 Iterating the passage from $\biR^\cu$ to $\widehat\biR^\bu$, we obtain
an acyclic DG\+ring~$\hathatRbu$.
 Iterating the assertion above, we conclude that the DG\+category
$(\biR^\cu\bModl)^{\bec\bec}$ is naturally equivalent to
$\hathatRbu\bModl$.
 On the other hand, $\biR^\cu\bModl$ is an idempotent-complete additive
DG\+category with twists (indeed, the additive category
$\sZ^0(\biR^\cu\bModl)$ is abelian, hence idempotent-complete).
 So Section~\ref{DG-functor-becbec-subsecn} tells that the DG\+functor
$\bec\bec\:\biR^\cu\bModl\rarrow(\biR^\cu\bModl)^{\bec\bec}$ is an equivalence
of DG\+categories.
 Thus we obtain an equivalence of DG\+categories
$$
 \biR^\cu\bModl\,\simeq\,\hathatRbu\bModl
$$
showing that the DG\+category of CDG\+modules over any CDG\+ring
$\biR^\cu=(R^*,d,h)$ is equivalent to the DG\+category of DG\+modules
over an acyclic DG\+ring~$\hathatRbu$.

\subsection{Example: DG-category of factorizations}
\label{almost-involution-factorizations-subsecn}
 Let $\sE$ be an additive category, $\Delta\:\sE\rarrow\sE$ be
an auto-equivalence, and $v\:\Id_\sE\rarrow\Delta$ be a potential
(see Section~\ref{DG-category-of-factorizations-defined-subsecn}).
 Then the DG\+category $\bF(\sE,\Delta,v)$ of factorizations of~$v$
in $\sE$ is an additive DG\+category with shifts, twists, and cones.
 Recall the notation $\sP(\sE,\Delta)\simeq\sE\times\sE$ for
the additive category of $2$\+$\Delta$-periodic objects in $\sE$
and $\sF(\sE,\Delta,v)=\sZ^0(\bF(\sE,\Delta,v))$ for the additive
category of factorizations of~$v$.

 The forgetful functor $X^\cu\longmapsto X^\cu{}^\#\:\sF(\sE,\Delta,v)
\rarrow\sP(\sE,\Delta)$, assigning to a factorization $X^\cu$ its
underlying $2$\+$\Delta$-periodic object $X^\circ=X^\cu{}^\#$, has
adjoints on both sides.
 The functor $G^+\:\sP(\sE,\Delta)\rarrow\sF(\sE,\Delta,v)$ left adjoint
to $X^\cu\longmapsto X^\cu{}^\#$ assigns to a $2$\+$\Delta$-periodic 
object $E^\circ\in\sP(\sE,\Delta)$ the factorization $G^+(E^\circ)$
freely generated by~$E^\circ$.
 Explicitly, one has $G^+(E^\circ)^n=E^n\oplus E^{n-1}$ for all
$n\in\boZ$, and the differential $d_{G^+,n}\:E^n\oplus E^{n-1}\rarrow
E^{n+1}\oplus E^n$ is given by the $2\times2$ matrix of morphisms whose
only nonzero entries are the identity morphism $E^n\rarrow E^n$ and
the composition $\delta^{n+1,n-1}_E\circ v_{E^{n-1}}\:E^{n-1}\rarrow
\Delta(E^{n-1})\rarrow E^{n+1}$.
 The functor $G^-\:\sP(\sE,\Delta)\rarrow\sF(\sE,\Delta,v)$ right
adjoint to $X^\cu\longmapsto X^\cu{}^\#$ assigns to
a $2$\+$\Delta$-periodic  object $E^\circ\in\sP(\sE,\Delta)$
the factorization $G^-(E^\circ)$ cofreely cogenerated by~$E^\circ$.
 Explicitly, $G^-(E^\circ)^n=E^n\oplus E^{n+1}$ for all $n\in\boZ$,
and $G^-(E^\circ)\simeq G^+(E^\circ)[1]$.

 Assume that the additive category $\sE$ is idempotent-complete; then
so is the DG\+category $\bF(\sE,\Delta,v)$.
 In this case, there is a natural equivalence of additive categories
$$
 \Upsilon=\Upsilon_{\sE,\Delta,v}\:\sP(\sE,\Delta)\simeq
 \sZ^0(\bF(\sE,\Delta,v)^\bec)
$$
(cf.~\cite[Example~3.19 and Remark~2.7]{Pedg}).
 The equivalence of categories $\Upsilon$ transforms
the functor~$\#$ into the functor $\Phi$, the functors $G^\pm$ into
the functors $\Psi^\pm$, and the natural inclusion
$\bF(\sE,\Delta,v)^0\rarrow\sP(\sE,\Delta)$ into the
functor~$\widetilde\Phi$.
 In other words, there are the following commutative diagrams of
additive functors:
\begin{equation} \label{factorizations-Phi-diagram}
\begin{gathered}
 \xymatrix{
  \sF(\sE,\Delta,v) \ar@{=}[rr] \ar[d]_{\#}
  && \sZ^0(\bF(\sE,\Delta,v)) \ar[d]^{\Phi_{\bF(\sE,\Delta,v)}} \\
  \sP(\sE,\Delta) \ar@<-0.4ex>[rr]_-{\Upsilon_{\sE,\Delta,v}}
  && \sZ^0(\bF(\sE,\Delta,v)^\bec) \ar@<-0.4ex>@{-}[ll]
 }
\end{gathered}
\end{equation}
\begin{equation} \label{factorizations-tilde-Phi-diagram}
\begin{gathered}
 \xymatrix{
  & \bF(\sE,\Delta,v)^0 \ar@{>->}[ld]
  \ar@{>->}[rd]^-{\widetilde\Phi_{\bF(\sE,\Delta,v)}} \\
  \sP(\sE,\Delta) \ar@<-0.4ex>[rr]_-{\Upsilon_{\sE,\Delta,v}}
  && \sZ^0(\bF(\sE,\Delta,v)^\bec) \ar@<-0.4ex>@{-}[ll]
 }
\end{gathered}
\end{equation}
and
\begin{equation} \label{factorizations-Psi-diagram}
\begin{gathered}
 \xymatrix{
  \sP(\sE,\Delta) \ar@<0.4ex>[rr]^-{\Upsilon_{\sE,\Delta,v}}
  \ar[d]_{G^\pm}
  && \sZ^0(\bF(\sE,\Delta,v)^\bec) \ar[d]^{\Psi_{\bF(\sE,\Delta,v)}^\pm}
  \ar@<0.4ex>@{-}[ll]
  \\ \sF(\sE,\Delta,v)\ar@{=}[rr] && \sZ^0(\bF(\sE,\Delta,v))
 }
\end{gathered}
\end{equation}
 Here the upper horizontal double line
in~\eqref{factorizations-Phi-diagram} and the lower horizontal double
line in~\eqref{factorizations-Psi-diagram} is the definition
of the additive category of factorizations $\sF(\sE,\Delta,v)$.
 The leftmost diagonal double arrow
in~\eqref{factorizations-tilde-Phi-diagram} is the obvious fully faithful
inclusion of additive categories mentioned in
Section~\ref{DG-category-of-factorizations-defined-subsecn}.
 The leftmost vertical arrow in~\eqref{factorizations-Phi-diagram}
is the forgetful functor $X^\cu\longmapsto X^\cu{}^\#=X^\circ$.
 There are, actually, two commutative diagrams depicted
in~\eqref{factorizations-Psi-diagram}: one has to choose either
$G^+$ for the leftmost vertical arrow and $\Psi^+$ for the rightmost
one, or $G^-$ for the leftmost vertical arrow and $\Psi^-$ for
the rightmost one.
 All the double lines and double arrows are category equivalences;
the arrows with tails are fully faithful functors; the ordinary arrows
are faithful functors.

\Section{Abelian DG-Categories}  \label{abelian-DG-secn}

 This section is mostly an extraction from~\cite[Sections~4.1, 4.3,
4.5, 4.6, and~9.3]{Pedg}.

\subsection{Exactness properties of the functors $\Xi$ and~$\Psi$}
\label{exactness-properties-of-Xi-and-Psi-subsecn}
 Let $\bE$ be a DG\+category and $f\:A\rarrow B$ be a closed morphism
of degree~$0$ in~$\bE$ such that the shift $A[1]$ and the cone
$\cone(f)$ exist in~$\bE$.
 Consider the pair of natural closed morphisms
$B\rarrow\cone(f)\rarrow A[1]$ of degree~$0$ in~$\bE$.

\begin{lem} \label{cone-kernel-cokernel}
 In the pair of natural morphisms $B\rarrow\cone(f)\rarrow A[1]$
in the preadditive category\/ $\sZ^0(\bE)$, the former morphism is
a kernel of the latter one and the latter morphism is a cokernel of
the former one.
\end{lem}

\begin{proof}
 See~\cite[proof of Lemma~4.2]{Pedg}.
\end{proof}

 Let $\bE$ be an additive DG\+category with shifts and cones.
 As in Section~\ref{Phi-and-Psi-subsecn}, we denote by $\Xi=\Xi_\bE\:
\sZ^0(\bE)\rarrow\sZ^0(\bE)$ the additive functor
$\Xi(A)=\cone(\id_A[-1])$.
 The endofunctor $\Xi\:\sZ^0(\bE)\rarrow\sZ^0(\bE)$ is left and right
adjoint to its own shift $\Xi[1]\:A\longmapsto\cone(\id_A)$; hence
the functor $\Xi$ preserves all limits and colimits (in particular,
all kernels and cokernels) that exist in the additive category
$\sZ^0(\bE)$ \,\cite[Lemma~4.1]{Pedg}.
 The endofunctor $\Xi\:\sZ^0(\bE)\rarrow\sZ^0(\bE)$ is faithful and
conservative~\cite[Lemma~4.3]{Pedg}.
 The endofunctor $\Xi$ can be naturally extended to a faithful
additive functor $\widetilde\Xi=\widetilde\Xi_\bE\:\bE^0\rarrow
\sZ^0(\bE)$ \,\cite[Lemma~4.5]{Pedg}.

\begin{cor} \label{Xi-extension-cokernel}
 Let\/ $\bE$ be an additive DG\+category with shifts and cones.
 Then \par
\textup{(a)} for any object $A\in\bE$, there is a pair of natural
morphisms $A[-1]\rightarrowtail\Xi(A)\twoheadrightarrow A$ in\/
$\sZ^0(\bE)$ in which the former morphism is a kernel of the latter one
and the latter morphism is a cokernel of the former one; \par
\textup{(b)} any object $A\in\bE$ is the cokernel of a natural
morphism\/ $\Xi(A)[-1]\rarrow\Xi(A)$ in\/~$\sZ^0(\bE)$.
\end{cor}

\begin{proof}
 Part~(a) is a particular case of Lemma~\ref{cone-kernel-cokernel}.
 To deduce part~(b) from part~(a), consider the composition
$\Xi(A)[-1]\twoheadrightarrow A[-1]\rightarrowtail\Xi(A)$ and
observe that its cokernel agrees with the cokernel of the morphism
$A[-1]\rightarrowtail\Xi(A)$, since the morphism
$\Xi(A)[-1]\twoheadrightarrow A[-1]$ is an epimorphism.
\end{proof}

 The following lemma is trivial but useful.

\begin{lem} \label{contractibles-described-lemma}
 Let\/ $\bE$ be a DG\+category with shifts and cones.
 Then any contractible object in\/ $\bE$ is a direct summand of
an object\/ $\Xi(A)$, for a suitable object $A\in\bE$.
\end{lem}

\begin{proof}
 For any two homotopic closed morphisms $f'$, $f''\:A\rarrow B$ of
degree~$0$ in $\bE$, the cones of $f'$ and $f''$ are isomorphic
in~$\sZ^0(\bE)$.
 In particular, if an object $A\in\bE$ is contractible, then
the object $\Xi(A)=\cone(\id_A[-1])$ is isomorphic to the cone of
zero morphism $\cone(0\:A[-1]\to A[-1])\simeq A\oplus A[-1]$ as
objects of~$\sZ^0(\bE)$.
 Consequently, $A$ is a direct summand of~$\Xi(A)$.
\end{proof}

\begin{lem}
 Let\/ $\bE$ be an additive DG\+category with shifts and cones, and
let $A\overset f\rarrow B\overset g\rarrow C$ be a composable pair
of morphisms in the additive category\/ $\sZ^0(\bE)$.
 In this context: \par
\textup{(a)} if the morphism\/ $\Xi(f)$ is a kernel of the morphism\/
$\Xi(g)$, then the morphism~$f$ is a kernel of the morphism~$g$
in\/~$\sZ^0(\bE)$; \par
\textup{(b)} if the morphism\/ $\Xi(g)$ is a cokernel of the morphism\/
$\Xi(f)$, then the morphism~$g$ is a cokernel of the morphism~$f$
in\/~$\sZ^0(\bE)$.
\end{lem}

\begin{proof}
 This is~\cite[Lemma~4.4]{Pedg}.
\end{proof}

 The definition of an \emph{idempotent-complete} additive DG\+category
was given in Section~\ref{DG-functor-becbec-subsecn}.

\begin{lem} \label{Xi-reflects-existence-of-co-kernels-given-twists}
 Let\/ $\bE$ be an idempotent-complete additive DG\+category with
shifts and twists.  In this setting: \par
\textup{(a)} if $g$~is a morphism in\/ $\sZ^0(\bE)$ and the morphism\/
$\Xi(g)$ has a kernel in\/ $\sZ^0(\bE)$, then the morphism~$g$ also
has a kernel in\/~$\sZ^0(\bE)$; \par
\textup{(b)} if $f$~is a morphism in $\sZ^0(\bE)$ and the morphism\/
$\Xi(f)$ has a cokernel in\/ $\sZ^0(\bE)$, then the morphism~$f$ also
has a cokernel in\/~$\sZ^0(\bE)$.
\end{lem}

\begin{proof}
 This is~\cite[Lemma~4.6]{Pedg}.
\end{proof}

 The DG\+category $\bE^\bec$ was constructed in
Section~\ref{bec-construction-subsecn} and the additive functor
$\Psi^+\:\sZ^0(\bE^\bec)\rarrow\sZ^0(\bE)$ was defined in
Section~\ref{Phi-and-Psi-subsecn}.
 The functor $\Psi^+$ has adjoints on both sides, so it preserves all
limits and colimits (in particular, all kernels and cokernels) that
exist in $\sZ^0(\bE^\bec)$.

\begin{lem} \label{Psi-reflects-existence-of-co-kernels}
 Let\/ $\bE$ be an additive DG\+category with shifts and cones.
 In this setting: \par
\textup{(a)} if $g$~is a morphism in\/ $\sZ^0(\bE^\bec)$ and
the morphism\/ $\Psi^+(g)$ has a kernel in\/ $\sZ^0(\bE)$, then
the morphism~$g$ has a kernel in\/ $\sZ^0(\bE^\bec)$; \par
\textup{(b)} if $f$~is a morphism in\/ $\sZ^0(\bE^\bec)$ and
the morphism $\Psi^+(f)$ has a cokernel in\/ $\sZ^0(\bE)$, then
the morphism~$f$ has a cokernel in\/ $\sZ^0(\bE^\bec)$.
\end{lem}

\begin{proof}
 This is~\cite[Lemma~4.16]{Pedg}.
\end{proof}

\subsection{Equivalent characterizations of abelian DG-categories}
\label{abelian-DG-categories-subsecn}
 An additive DG\+category $\bE$ with shifts and cones is called
\emph{abelian} if both the additive categories $\sZ^0(\bE)$ and
$\sZ^0(\bE^\bec)$ are abelian~\cite[Section~4.5]{Pedg}.
 The following lemma is helpful.

\begin{lem} \label{conservative-to-abelian}
 Let\/ $\sA$ be an additive category with kernels and cokernels,
let\/ $\sB$ be an abelian category, and let\/ $\Theta\:\sA
\rarrow\sB$ be a conservative additive functor preserving kernels
and cokernels.
 Then\/ $\sA$ is also an abelian category.
\end{lem}

\begin{proof}
 This is~\cite[Lemma~4.31(a)]{Pedg}.
 For any morphism~$f$ in $\sA$, consider the natural morphism~$g$
from the coimage to the image of~$f$.
 Then the morphism $\Theta(g)$ is an isomorphism, since the functor
$\Theta$ preserves kernels and cokernels (hence images and coimages),
and the category $\sB$ is abelian.
 Since the functor $\Theta$ is conservative by assumption, we can
conclude that $g$~is an isomorphism.
\end{proof}

 The following proposition, which is a simplified restatement
of~\cite[Corollary~4.37]{Pedg}, lists equivalent characterizations of
abelian DG\+categories.

\begin{prop}
 Let\/ $\bE$ be an additive DG\+category with shifts and cones.
 Then the following conditions are equivalent:
\begin{enumerate}
\item the DG\+category\/ $\bE$ is abelian;
\item the additive category\/ $\sZ^0(\bE)$ is abelian;
\item all kernels and cokernels exist in the additive category\/
$\sZ^0(\bE)$, and the additive category\/ $\sZ^0(\bE^\bec)$ is
abelian;
\item the DG\+category\/ $\bE$ is idempotent-complete and has all
twists, and the additive category\/ $\sZ^0(\bE^\bec)$ is abelian.
\end{enumerate}
\end{prop}

\begin{proof}
 The implications (1)\,$\Longrightarrow$\,(2) and
(1)\,$\Longrightarrow$\,(3) are obvious.

 (2)\,$\Longrightarrow$\,(1) To prove that the additive category
$\sZ^0(\bE^\bec)$ is abelian, apply Lemma~\ref{conservative-to-abelian}
to the additive functor $\Psi^+\:\sZ^0(\bE^\bec)\rarrow\sZ^0(\bE)$ and
use Lemma~\ref{Psi-reflects-existence-of-co-kernels}.

 (3)\,$\Longrightarrow$\,(1) To prove that the additive category
$\sZ^0(\bE)$ is abelian, apply Lemma~\ref{conservative-to-abelian}
to the additive functor $\Phi\:\sZ^0(\bE)\rarrow\sZ^0(\bE^\bec)$,
which preserves all limits and colimits (hence all kernels and
cokernels) because it has adjoints on both sides.
 The functor $\Phi$ is conservative by~\cite[Lemma~3.12]{Pedg}, as
mentioned in Section~\ref{Phi-and-Psi-subsecn}.

 (4)\,$\Longrightarrow$\,(3) We need to prove the existence of kernels
and cokernels in the additive category $\sZ^0(\bE)$.
 Let $f$~be a morphism in $\sZ^0(\bE)$.
 Then the morphism $\Phi(f)$ has a kernel in $\sZ^0(\bE^\bec)$, since
the category $\sZ^0(\bE^\bec)$ is abelian by assumption.
 Hence the morphism $\Psi^+\Phi(f)$ has a kernel in $\sZ^0(\bE)$,
since the functor $\Psi^+$ preserves kernels.
 By~\cite[Lemma~3.8]{Pedg}, the morphism $\Psi^+\Phi(f)$ is isomorphic
to the morphism $\Xi(f)$ (see the end of
Section~\ref{Phi-and-Psi-subsecn}).
 So the morphism $\Xi(f)$ has a kernel in $\sZ^0(\bE)$, and it remains
to apply Lemma~\ref{Xi-reflects-existence-of-co-kernels-given-twists}
in order to conclude that the morphism~$f$ has a kernel in $\sZ^0(\bE)$.
 The existence of cokernels is provable similarly.

 (1)\,$\Longrightarrow$\,(4) The additive category $\sZ^0(\bE)$ is
abelian, hence idempotent-complete, which means that the DG\+category
$\bE$ is idempotent-complete.
 The assertion about existence of twists is so important that
 we formulate it as a separate proposition.
\end{proof}

\begin{prop} \label{abelian-DG-categories-have-twists}
 All twists exist in any abelian DG\+category.
\end{prop}

\begin{proof}
 This is~\cite[Proposition~4.35]{Pedg}.
 The point is that, for any object $A$ in an abelian DG\+category $\bE$,
all twists of $A$ can be produced as the cokernels of suitable
morphisms $\Xi(A)[-1]\rarrow\Xi(A)$
(cf.\ Corollary~\ref{Xi-extension-cokernel}(b)).
 Let us empasize that the functor $\Xi$ takes twists to
isomorphisms~\cite[Lemma~3.11 or~4.5]{Pedg}.
 The discussion in~\cite[Section~7]{NST} sheds some additional light
on this argument.
\end{proof}

\begin{cor} \label{becbec-equivalence-for-abelian-DG}
 For any abelian DG\+category\/ $\bE$, the DG\+functor\/
$\bec\bec\:\bE\rarrow\bE^{\bec\bec}$ is an equivalence of
DG\+categories.
\end{cor}

\begin{proof}
 This is~\cite[Corollary~4.36 or proof of Proposition~4.35]{Pedg}.
 Indeed, any abelian DG\+category is obviously idempotent-complete,
so the assertion follows from
Proposition~\ref{abelian-DG-categories-have-twists} and
the discussion in Section~\ref{DG-functor-becbec-subsecn}.
\end{proof}

\subsection{Exactly embedded full abelian DG-subcategories}
\label{exactly-embedded-subsecn}
 Let $\sA$ be an abelian category and $\sB\subset\sA$ be a full
subcategory.
 We will say that $\sB$ is an \emph{exactly embedded full abelian
subcategory} in $\sA$ if $\sB$ is an abelian category and
the inclusion $\sB\rightarrowtail\sA$ is an exact functor of
abelian categories.
 Equivalently, a strictly full subcategory $\sB\subset\sA$ is
an exactly embedded full abelian subcategory if and only if it is
closed under finite direct sums, kernels, and cokernels in~$\sA$.

 Let $\bA$ be a DG\+category and $\bB\subset\bA$ be a (strictly) full
DG\+subcategory.
 Then the DG\+category $\bB^\bec$ can be naturally viewed as
a (strictly) full DG\+subcategory in~$\bA^\bec$.
 Assuming that the DG\+category $\bA$ has shifts and cones and
the full subcategory $\bB$ is closed under shifts and cones in $\bA$,
the functors $\Phi_\bB$ and $\Psi^\pm_\bB$ agree with the functors
$\Phi_\bA$ and $\Psi^\pm_\bA$, respectively.
 Similarly, the DG\+functor $\bec\bec_\bB\:\bB\rarrow\bB^{\bec\bec}$
agrees with the DG\+functor $\bec\bec_\bA\:\bA\rarrow\bA^{\bec\bec}$.

 Let $\bA$ be an abelian DG\+category and $\bB\subset\bA$ be a full
DG\+subcategory closed under finite direct sums, shifts, and cones.
 We will say that $\bB$ is an \emph{exactly embedded full abelian
DG\+subcategory} in $\bA$ if $\bB$ is an abelian DG\+category and
both the fully faithful inclusions $\sZ^0(\bB)\rightarrowtail
\sZ^0(\bA)$ and $\sZ^0(\bB^\bec)\rarrow\sZ^0(\bA^\bec)$ are exact
functors of abelian categories~\cite[Section~9.3]{Pedg}.

\begin{prop}
 Let\/ $\bA$ be an abelian DG\+category and\/ $\bB\subset\bA$ be
a (strictly) full DG\+subcategory closed under finite direct sums,
shifts, and cones.
 Then the following conditions are equivalent:
\begin{enumerate}
\item $\bB$ is an exactly embedded full abelian DG\+subcategory
in\/~$\bA$;
\item $\sZ^0(\bB)$ is an exactly embedded full abelian subcategory
in\/~$\sZ^0(\bA)$;
\item the full DG\+subcategory\/ $\bB\subset\bA$ is closed under
twists and direct summands, and\/ $\sZ^0(\bB^\bec)$ is an exactly
embedded full abelian subcategory in\/ $\sZ^0(\bA^\bec)$.
\end{enumerate}
\end{prop}

\begin{proof}
 The implication (1)\,$\Longrightarrow$\,(2) is obvious.

 (1)\,$\Longrightarrow$\,(3) For any abelian DG\+category $\bB$,
the additive category $\sZ^0(\bB)$ is abelian, hence
idempotent-complete, hence closed under direct summands in
any ambient additive category $\sZ^0(\bA)$.
 Similarly, any abelian DG\+category $\bB$ has all twists by
Proposition~\ref{abelian-DG-categories-have-twists}, hence $\bB$
is closed under twists in any ambient DG\+category~$\bA$.

 (2)\,$\Longrightarrow$\,(1) We need to prove that the full
subcategory $\sZ^0(\bB^\bec)$ is closed under kernels and cokernels
in $\sZ^0(\bA^\bec)$.
 For this purpose we observe that, given an object
$K\in\sZ^0(\bA^\bec)$, one has $K\in\sZ^0(\bB^\bec)$ whenever
$\Psi^+(K)\in\sZ^0(\bB)$.
 Let $f\:X\rarrow Y$ be a morphism in $\sZ^0(\bB^\bec)$ and $K$ be
the kernel of~$f$ computed in $\sZ^0(\bA^\bec)$.
 Then $\Psi^+(K)$ is the kernel of $\Psi^+(f)$ in $\sZ^0(\bA)$.
 Since the full subcategory $\sZ^0(\bB)$ is closed under kernels in
$\sZ^0(\bA)$, it follows that $\Psi^+(K)\in\sZ^0(\bB)$, hence
$K\in\sZ^0(\bB^\bec)$.
 The proof of the closure under cokernels is similar.

 (3)\,$\Longrightarrow$\,(1) We need to prove that the full
subcategory $\sZ^0(\bB)$ is closed under kernels and cokernels
in $\sZ^0(\bA)$.
 Let $f\:M\rarrow N$ be a morphism in $\sZ^0(\bB)$.
 Then $\Phi(f)\:\Phi(M)\rarrow\Phi(N)$ is a morphism
in $\sZ^0(\bB^\bec)$.
 Since the full subcategory $\sZ^0(\bB^\bec)$ is closed under
kernels and cokernels in the abelian category $\sZ^0(\bA^\bec)$,
the morphism $\Phi(f)$ has the same kernel $K$ in both
$\sZ^0(\bA^\bec)$ and $\sZ^0(\bB^\bec)$.
 Hence the object $\Psi^+(K)\in\sZ^0(\bB)$ is the kernel of
the morphism $\Psi^+\Phi(f)$ in both $\sZ^0(\bA)$ and~$\sZ^0(\bB)$.

 Let $L$ be a kernel of the morphism~$f$ in the abelian category
$\sZ^0(\bA)$.
 Then the object $\Xi(L)$ is a kernel of the morphism $\Xi(f)$
in $\sZ^0(\bA)$.
 As the morphism $\Xi(f)$ is isomorphic to $\Psi^+\Phi(f)$,
we see that $\Xi(L)\simeq\Psi^+(K)\in\sZ^0(\bB)$.
 It remains to observe that the object $L$ is a direct summand of
$L\oplus L[-1]$, and the latter is a twist of $\Xi(L)$ in $\bA$, in
order to conclude that $L\in\bB$ (since the full DG\+subcategory
$\bB$ is closed under twists and direct summands in~$\bA$).
 The argument for cokernels is similar.
\end{proof}

\begin{ex}
 The following counterexample~\cite[Example~4.30]{Pedg} shows that
a full DG\+subcategory $\bB$ closed under finite direct sums,
shifts, and cones in an abelian DG\+category $\bA$ need \emph{not} be
an abelian DG\+category even when the additive category
$\sZ^0(\bB^\bec)$ is abelian.
 In particular, an additive DG\+category $\bE$ with shifts and cones
need \emph{not} be abelian when the additive category $\sZ^0(\bE^\bec)$
is abelian.
 
 Let $\bA$ be the DG\+category of complexes of vector spaces over
a field~$k$ and $\bB\subset\bA$ be the full DG\+subcategory of
acyclic complexes.
 Then the full DG\+subcategory $\bB^\bec$ coincides with the whole
ambient DG\+category $\bA^\bec$, and $\sZ^0(\bB^\bec)=\sZ^0(\bA^\bec)$
is an abelian category equivalent to the category of graded
$k$\+vector spaces.
 But the additive category of acyclic complexes of $k$\+vector
spaces $\sZ^0(\bB)$ is not abelian.
\end{ex}

\subsection{Examples}
 The following three classes of abelian DG\+categories are the main
intended examples in this paper.

\begin{ex} \label{abelian-DG-category-of-complexes-example}
 Let $\sE$ be an abelian category.
 Then the DG\+category $\bC(\sE)$ of complexes in $\sE$ has finite
direct sums, shifts, twists, and cones
(see Section~\ref{DG-category-of-complexes-defined-subsecn}).
 The additive category $\sZ^0(\bC(\sE))$ of complexes in $\sE$
and closed morphisms of degree~$0$ between them is well-known to be
abelian.
 According to Section~\ref{almost-involution-complexes-subsecn},
the additive category $\sZ^0(\bC(\sE)^\bec)$ is equivalent to
the category $\sG(\sE)$ of graded objects in $\sE$, which is also
abelian.
 Thus $\bC(\sE)$ is an abelian DG\+category.
\end{ex}

\begin{ex} \label{abelian-DG-category-of-CDG-modules-example}
 Let $\biR^\cu=(R^*,d,h)$ be a CDG\+ring.
 Then the DG\+category $\biR^\cu\bModl$ of left CDG\+modules over
$\biR^\cu$ has finite direct sums (as well as infinite products and
coproducts), shifts, twists, and cones (see
Section~\ref{DG-category-of-CDG-modules-defined-subsecn}).
 According to Section~\ref{almost-involution-CDG-modules-subsecn},
the additive category $\sZ^0(\biR^\cu\bModl)$ of left CDG\+modules over
$\biR^\cu$ and closed morphisms of degree~$0$ between them is abelian
(in fact, it is the category of graded modules over a graded
ring~$\widehat\biR^*$).
 Furthermore, the additive category $\sZ^0((\biR^\cu\bModl)^\bec)$ is
equivalent to the category $R^*\bModl$ of graded modules over
the graded ring $R^*$, so it is also abelian.
 Therefore, $\biR^\cu\bModl$ is an abelian DG\+category.
\end{ex}

\begin{ex} \label{abelian-DG-category-of-factorizations-example}
 Let $\sE$ be an abelian category, $\Delta\:\sE\rarrow\sE$ be
an autoequivalence, and $v\:\Id_\sE\rarrow\Delta$ be a potential
(see Section~\ref{DG-category-of-factorizations-defined-subsecn}).
 Then the DG\+category of factorizations $\bF(\sE,\Delta,v)$ has
finite direct sums, shifts, twists, and cones.
 One can easily check that the additive category
$\sZ^0(\bF(\sE,\Delta,v))$ is abelian.
 According to Section~\ref{almost-involution-factorizations-subsecn},
the additive category $\sZ^0(\bF(\sE,\Delta,v)^\bec)$ is equivalent
to the category $\sP(\sE,\Delta)$ of $2$\+$\Delta$-periodic objects
in $\sE$, which is also abelian (being, in turn, equivalent to
$\sE\times\sE$).
 Thus $\bF(\sE,\Delta,v)$ is an abelian DG\+category.
\end{ex}

 For further examples of abelian DG\+categories we refer the reader
to~\cite[Example~4.41]{Pedg}.
 Notice that Example~\ref{abelian-DG-category-of-factorizations-example}
is a particular case of~\cite[Example~4.42]{Pedg},
as per~\cite[Remark~2.7]{Pedg}.

\Section{Cotorsion Pairs} \label{cotorsion-pairs-secn}

 The concept of a complete cotorsion pair goes back to Salce~\cite{Sal}.
 The material of this section is an accumulated result of the work
of many people over the decades.
 The paper by Eklof and Trlifaj~\cite{ET} was the main development.
 The exposition below is an extraction from~\cite[Sections~1
and~3]{PS4}, which in turn is based on~\cite[Sections~3--4]{PR}.
 All proofs are omitted and replaced with references.

\subsection{Complete cotorsion pairs}
 Let\/ $\sE$ be an abelian category and $\sL$, $\sR\subset\sE$ be two
classes of objects in~$\sE$.
 We will denote by $\sL^{\perp_1}\subset\sE$ the class of all objects
$X\in\sE$ such that $\Ext_\sE^1(L,X)=0$ for all $L\in\sL$.
 Similarly, ${}^{\perp_1}\sR\subset\sE$ denotes the class of all objects
$Y\in\sE$ such that $\Ext_\sE^1(Y,R)=0$ for all $R\in\sR$.

 A pair of classes of objects $(\sL,\sR)$ in $\sE$ is called
a \emph{cotorsion pair} if $\sR=\sL^{\perp_1}$ and
$\sL={}^{\perp_1}\sR\subset\sE$.
 A cotorsion pair $(\sL,\sR)$ in $\sE$ is said to be \emph{complete}
if for every object $E\in\sE$ there exist short exact sequences
\begin{align}
 &0\lrarrow R'\lrarrow L\lrarrow E \lrarrow0,
 \label{special-precover-sequence} \\
 &0\lrarrow E\lrarrow  R\lrarrow L'\lrarrow0
 \label{special-preenvelope-sequence}
\end{align}
in $\sE$ with $L$, $L'\in\sL$ and $R$, $R'\in\sR$.

 Given a cotorsion pair $(\sL,\sR)$ in $\sE$, the short exact
sequence~\eqref{special-precover-sequence} is called
a \emph{special precover sequence}, and an epimorphism $L\rarrow E$
in $\sE$ with $L\in\sL$ and a kernel in $\sR$ is called
a \emph{special\/ $\sL$\+precover} of an object $E\in\sE$.
 The short exact sequence~\eqref{special-preenvelope-sequence} is called
a \emph{special preenvelope sequence}, and a monomorphism $E\rarrow R$
in $\sE$ with $R\in\sR$ and a cokernel in $\sL$ is called
a \emph{special\/ $\sR$\+preenvelope} of an object $E\in\sE$.
 The exact sequences~(\ref{special-precover-sequence}--%
\ref{special-preenvelope-sequence}) are collectively referred to
as the \emph{approximation sequences}.
 
 A class of objects $\sL\subset\sE$ is said to be \emph{generating}
if for every object $E\in\sE$ there exists an epimorphism
$L\twoheadrightarrow E$ in $\sE$ with $L\in\sL$.
 A class of objects $\sR\subset\sE$ is said to be \emph{cogenerating}
if for every object $E\in\sE$ there exists a monomorphism
$E\rightarrowtail R$ in $\sE$ with $R\in\sR$.
 Clearly, in a complete cotorsion pair $(\sL,\sR)$, the class $\sL$
must be generating and the class $\sR$ must be cogenerating.
 But in a noncomplete cotorsion pair this need not be the case,
unless it is assumed additionally that the abelian category $\sE$ has
enough projective and/or injective objects~\cite[Remark~3.5 and
Example~3.6]{PS4}.

\begin{lem} \label{salce-lemma}
 Let $(\sL,\sR)$ be a cotorsion pair in an abelian category\/~$\sE$.
\par
\textup{(a)} If the class\/ $\sL$ is generating in\/ $\sE$ and every
object of\/ $\sE$ admits a special preenvelope
sequence~\eqref{special-preenvelope-sequence}, then every object of\/
$\sE$ also admits a special precover
sequence~\eqref{special-precover-sequence}.
 So the cotorsion pair $(\sL,\sR)$ is complete in this case. \par
\textup{(b)} If the class\/ $\sR$ is cogenerating in\/ $\sE$ and every
object of\/ $\sE$ admits a special precover
sequence~\eqref{special-precover-sequence}, then every object of\/
$\sE$ also admits a special preenvelope
sequence~\eqref{special-preenvelope-sequence}.
 So the cotorsion pair $(\sL,\sR)$ is complete in this case. 
\end{lem}

\begin{proof}
 This a generalization of the classical Salce lemmas~\cite{Sal}.
 A proof in the stated generality can be found in~\cite[Lemma~1.1]{PS4}.
\end{proof}

\begin{lem} \label{hereditary-cotorsion}
 Let $(\sL,\sR)$ be a cotorsion pair in an abelian category\/~$\sE$.
 Assume that the class\/ $\sL$ is generating and the class\/ $\sR$ is
cogenerating in\/~$\sE$.
 Then the following conditions are equivalent:
\begin{enumerate}
\item the class\/ $\sL$ is closed under kernels of epimorphisms
in\/~$\sE$;
\item the class\/ $\sR$ is closed under cokernels of monomorphisms
in\/~$\sE$;
\item $\Ext^2_\sE(L,R)=0$ for all $L\in\sL$ and $R\in\sR$;
\item $\Ext^n_\sE(L,R)=0$ for all $L\in\sL$, \,$R\in\sR$,
and $n\ge1$.
\end{enumerate}
\end{lem}

\begin{proof}
 See~\cite[Lemma~1.4]{PS4} for a brief discussion with references.
\end{proof}

 A cotorsion pair satisfying the equivalent conditions of
Lemma~\ref{hereditary-cotorsion} is said to be \emph{hereditary}.

 Given any class of objects $\sS$ in an abelian category $\sE$, one
can construct a cotorsion pair $(\sL,\sR)$ in $\sE$ by setting
$\sR=\sS^{\perp_1}$ and $\sL={}^{\perp_1}\sR$.
 The resulting cotorsion pair $(\sL,\sR)$ is said to be \emph{generated}
by the class~$\sS$.
 Obviously, one has $\sS\subset\sL$.

\subsection{Transfinitely iterated extensions}
\label{transfinite-extensions-subsecn}
 Let $\sE$ be an abelian category and $\alpha$~be an ordinal.
 Let $(F_i\to F_j)_{i<j\le\alpha}$ be a direct system of objects
in $\sE$ indexed by the ordinal~$\alpha+1$.
 One says that the direct system $(F_i)_{i\le\alpha}$ is a \emph{smooth}
(or \emph{continuous}) chain if $F_j=\varinjlim_{i<j}F_i$ for all
limit ordinals $j\le\alpha$.

 Let $(F_i)_{i\le\alpha}$ be a direct system in $\sE$ which is a smooth
chain.
 Assume that $F_0=0$ and the morphism $F_i\rarrow F_{i+1}$ is
a monomorphism for every ordinal $i<\alpha$, and let $S_i$ be
a cokernel of $F_i\rarrow F_{i+1}$ in~$\sE$.
 Then the object $F=F_\alpha$ is said to be \emph{filtered by}
the objects $S_i$, \,$0\le i<\alpha$.
 In this case, the object $F\in\sE$ is also said to be
a \emph{transfinitely iterated extension} (\emph{in the sense of
the directed colimit}) of the objects~$S_i$.

 Let $\sS$ be a class of objects in~$\sE$.
 Then the class of all objects in $\sE$ filtered by (objects isomorphic
to) objects from $\sS$ is denoted by $\Fil(\sS)$.
 Given a class of objects $\sF\subset\sE$, we will denote by
$\sF^\oplus\subset\sE$ the class of all direct summands of objects
from~$\sF$.

\begin{lem} \label{eklof-lemma}
 Let\/ $\sE$ be an abelian category.
 Then, for any class of objects\/ $\sR\subset\sE$, the class of
objects\/ ${}^{\perp_1}\sR$ is closed under transfinitely iterated
extensions (in the sense of the directed colimit) and direct summands
in\/~$\sE$.
 In other words, ${}^{\perp_1}\sR=\Fil({}^{\perp_1}\sR)^\oplus$.
\end{lem}

\begin{proof}
 This result is known as the \emph{Eklof lemma}~\cite[Lemma~1]{ET}.
 In the stated generality, a proof can be found in~\cite[Lemma~4.5]{PR}.
 We refer to~\cite[Proposition~1.3 and the preceding paragraph]{PS4}
for a further discussion and historical references.
\end{proof}

\subsection{Small object argument} \label{small-object-subsecn}
 We start with a brief discussion of presentable objects and locally
presentable categories.

 Let $\lambda$~be an infinite regular cardinal.
 A poset $I$ is said to be \emph{$\lambda$\+directed} if any subset
of $I$ of the cardinality~$<\lambda$ has an upper bound in~$I$.
 A \emph{$\lambda$\+directed colimit} is a colimit of a diagram
indexed by a $\lambda$\+directed poset.

 Let $\sK$ be a category where $\lambda$\+directed colimits exist.
 An object $K\in\sK$ is said to be \emph{$\lambda$\+presentable} if
the functor $\Hom_\sK(K,{-})\:\sK\rarrow\Sets$ (where $\Sets$ denotes 
the category of sets) preserves $\lambda$\+directed colimits.
 The category $\sK$ is said to be \emph{$\lambda$\+accessible} if it has
a set of $\lambda$\+presentable objects such that every object of $\sK$
is a $\lambda$\+directed colimit of objects from this set.
 The category $\sK$ is said to be \emph{locally $\lambda$\+presentable}
if it is cocomplete and
$\lambda$\+accessible~\cite[Definition~1.17]{AR}.

 Equivalently, a category is locally $\lambda$\+presentable if and only
if it is cocomplete and has a strong generating set consisting of
$\lambda$\+presentable objects~\cite[Theorem~1.20]{AR}.
 In this paper, we are interested in abelian categories, and any
generating set of objects in an abelian category is a strong generating
set; so we do not recall the details of the definition of a strong
generating set here, referring the reader to~\cite[Section~0.6]{AR}
instead.
 To summarize it briefly, one can say that a set $(K_i)_{i\in I}$ of
objects in a category $\sK$ is generating if and only if the related
functor $(\Hom_\sK(K_i,{-}))_{i\in I}$ from $\sK$ to the category of
$I$\+sorted sets $\Sets^I$ is faithful; a generating set is strong if
and only if this functor is also conservative.

 A category is called \emph{locally presentable} if it is locally
$\lambda$\+presentable for some infinite regular cardinal~$\lambda$.
 In particular, locally $\omega$\+presentable categories (where
$\omega$~is the minimal infinite cardinal) are referred to as
\emph{locally finitely presentable}, and $\omega$\+presentable objects
are called \emph{finitely presentable}~\cite[Definitions~1.1 and~1.9,
Theorem~1.11]{AR}.

 The following theorem summarizes the application of Quillen's
classical ``small object argument'' to cotorsion pairs.
 It forms a generalization of the Eklof--Trlifaj
theorem~\cite[Theorems~2 and~10]{ET} from module categories to locally
presentable abelian categories.

\begin{thm} \label{eklof-trlifaj}
 Let $\sE$ be a locally presentable abelian category, $\sS\subset\sE$
be a set of objects, and $(\sL,\sR)$ be the cotorsion pair generated
by\/ $\sS$ in\/~$\sE$. \par
\textup{(a)} If the class\/ $\sL$ is generating and
the class\/ $\sR$ is cogenerating in\/ $\sE$, then the cotorsion pair
$(\sL,\sR)$ is complete. \par
\textup{(b)} If the class of objects\/ $\Fil(\sS)$ is generating
in\/ $\sE$, then one has\/ $\sL=\Fil(\sS)^\oplus$.
\end{thm}

\begin{proof}
 A proof in the stated generality can be found in~\cite[Corollary~3.6
and Theorem~4.8]{PR} or~\cite[Theorems~3.3 and~3.4]{PS4}.
 We refer to~\cite[Section~3]{PS4} for a historical discussion
with references.
 For counterexamples showing that the additional assumptions in part~(a)
cannot be dropped, see~\cite[Remark~3.5 and Examples~3.6]{PS4}.

 The argument in~\cite{PR} and~\cite{PS4} is based on the notion of
a \emph{weak factorization system}, which we will recall below in
Sections~\ref{weak-factorization-subsecn}--%
\ref{abelian-weak-factorization-subsecn}.
 See Proposition~\ref{small-object-argument} for a formulation of
the small object argument in the relevant generality and
Proposition~\ref{abelian-wfs-and-cotorsion-pairs} for the connection
between cotorsion pairs and weak factorization systems.
\end{proof}

\Section{Abelian Model Structures} \label{abelian-model-structures-secn}

 The notion of an abelian model structure and the connection with
complete cotorsion pairs is due to Hovey~\cite{Hov}.
 The exposition in this section is an extraction
from~\cite[Sections~2 and~4]{PS4}, which in turn is based
on the papers~\cite{Hov} and~\cite{Bec}.
 All the proofs are omitted and replaced with references.

\subsection{Weak factorization systems}
\label{weak-factorization-subsecn}
 A weak factorization system is ``a half of a model structure''.
 We refer to~\cite[Section~2]{PS4} for a terminological and historical
discussion with references.

 Let $\sE$ be a category, and let $l\:A\rarrow B$ and $r\:C\rarrow D$
be two morphisms.
 One says that $r$~has the \emph{right lifting property} with respect
to~$l$, or equivalently, $l$~has the \emph{left lifting property}
with respect to~$r$ if for any commutative square as on the diagram
$$
 \xymatrix{
  A \ar[r] \ar[d]_l & C \ar[d]^r \\
  B \ar[r] \ar@{..>}[ru] & D
 }
$$
a diagonal arrow exists making both the triangles commutative.
 (The word ``weak'' in the expression ``weak factorization system'',
which is defined below, refers to the fact that the diagonal filling
need not be unique.)

 Let $\cL$ and $\cR$ be two classes of morphisms in~$\sE$.
 We will denote by $\cL^\square$ the class of all morphisms~$x$ in
$\sE$ having the right lifting property with respect to all
morphisms $l\in\cL$.
 Similarly, ${}^\square\cR$ is the class of all morphisms~$y$ in
$\sE$ having the left lifting property with respect to all
morphisms $r\in\cR$.

 A pair of classes of morphisms $(\cL,\cR)$ in a category $\sE$ is
said to be a \emph{weak factorization system} if $\cL={}^\square\cR$,
\ $\cR=\cL^\square$, and every morphism~$f$ in $\sE$ can be
factorized as $f=rl$ with $r\in\cR$ and $l\in\cL$.

\subsection{Transfinite compositions}
 Let $\alpha$~be an ordinal, and let $(E_i\to E_j)_{i<j\le\alpha}$ be
a direct system of objects in $\sE$ indexed by the ordinal~$\alpha+1$.
 Assume that the direct system $(E_i)_{i\le\alpha}$ is a smooth
chain (as defined in Section~\ref{transfinite-extensions-subsecn}).
 Then we will say that the morphism $E_0\rarrow E_\alpha$ in
the direct system is a \emph{transfinite composition} of the morphisms
$E_i\rarrow E_{i+1}$, \ $0\le i<\alpha$ (\emph{in the sense of
the directed colimit}).

 Let $A\rarrow B$ be a morphism in~$\sE$.
 Then a morphism $X\rarrow Y$ in $\sE$ is said to be a \emph{pushout}
of the morphism $A\rarrow B$ if there exists a pair of morphisms
$A\rarrow X$ and $B\rarrow Y$ such that $A\rarrow B\rarrow Y$, \
$A\rarrow X\rarrow Y$ is a cocartesian square (otherwise known as
a pushout square) in~$\sE$.

 One says that an object $X$ is a \emph{retract} of an object $A\in\sE$
if there exist morphisms $i\:X\rarrow A$ and $p\:A\rarrow X$ such that
$pi=\id_X$ in~$\sE$.
 A morphism~$z$ is said to be a retract of a morphism~$f$ in $\sE$ if
$z$ is a retract of~$f$ in the category $\sE^\to$ of morphisms in $\sE$
(with commutative squares in $\sE$ being morphisms in~$\sE^\to$).

 Given a class of morphisms $\cS$ in $\sE$, one denotes by $\Cell(\cS)$
the closure of $\cS$ under pushouts and transfinite compositions, and
by $\Cof(\cS)$ the closure of $\cS$ under pushouts, transfinite
compositions, and retracts.
 (See~\cite[Section~2]{PS4} for a discussion of the basic properties of
these constructions.)

 The following very basic lemma is a ``morphism version'' of
Lemma~\ref{eklof-lemma}.

\begin{lem}
 For any class of morphisms $\cR$ in\/ $\sE$, the class of morphisms\/
${}^\square\cR$ is closed under transfinite compositions (in the sense
of the directed colimit), pushouts, and retracts.
 In other words, ${}^\square\cR=\Cof({}^\square\cR)$. \qed
\end{lem}

 The following theorem is a formulation of Quillen's small object
argument suitable for deducing Theorem~\ref{eklof-trlifaj}.

\begin{prop} \label{small-object-argument}
 Let\/ $\sE$ be a locally presentable category and\/ $\cS$ be a set
of morphisms in\/~$\sE$.
 Then any morphism~$f$ in\/ $\sE$ can be factorized as $f=rl$ with
$r\in\cS^\square$ and $l\in\Cell(\cS)$.
 Moreover, this factorization can be chosen for all morphisms~$f$
in\/ $\sE$ in such a way that it depends functorially on~$f$.
 The pair of classes of morphisms $\cL=\Cof(\cS)$ and $\cR=\cS^\square$
is a weak factorization system in\/~$\sE$.
\end{prop}

\begin{proof}
 A proof can be found in~\cite[Proposition~1.3]{Bek}.
 We refer to~\cite[Theorem~3.1]{PS4} for a discussion with references
(see also~\cite[Lemma~2.1]{PS4}).
\end{proof}

\subsection{Abelian weak factorization systems}
\label{abelian-weak-factorization-subsecn}
 Let $\sE$ be an abelian category, and let $\sL$ and $\sR$ be two
classes of objects~$\sE$.
 By an \emph{$\sL$\+monomorphism} we mean a monomorphism in $\sE$
with a cokernel belonging to~$\sL$.
 Dually, an \emph{$\sR$\+epimorphism} is an epimorphism in $\sE$ with
a kernel belonging to~$\sR$.
 The class of all $\sL$\+monomorphisms in $\sE$ is denoted by
$\sL\Mono$, and the class of all $\sR$\+epimorphisms in $\sE$ is
denoted by $\sR\Epi$.

 A weak factorization system $(\cL,\cR)$ in $\sE$ is said to be
\emph{abelian} if there exists a pair of classes of objects $\sL$
and $\sR$ in $\sE$ such that $\cL=\sL\Mono$ and $\cR=\sR\Epi$.

\begin{prop} \label{abelian-wfs-and-cotorsion-pairs}
 Let\/ $\sL$ and\/ $\sR$ be two classes of objects in an abelian
category\/~$\sE$.
 Then the pair of classes of morphisms $\cL=\sL\Mono$ and $\cR=\sR\Epi$
forms a weak factorization system $(\cL,\cR)$ in\/ $\sE$ if and only if
$(\sL,\sR)$ is a complete cotorsion pair in\/~$\sE$.
 So there is a bijection between abelian weak factorization systems
and complete cotorsion pairs in any abelian category\/~$\sE$.
\end{prop}

\begin{proof}
 This result is essentially due to Hovey~\cite{Hov}; see
also~\cite[Theorem~2.4]{PS4}.
\end{proof}

 A weak factorization system $(\cL,\cR)$ in a category $\sE$ is said
to be \emph{cofibrantly generated} if there exists a set of morphisms
$\cS$ in $\sE$ such that $\cR=\cS^\square$.

\begin{lem} \label{cofibrantly-generated-abelian-wfs}
 Let\/ $\sE$ be a locally presentable abelian category and $(\cL,\cR)$
be an abelian weak factorization system in\/ $\sE$ corresponding to
a complete cotorsion pair $(\sL,\sR)$, as in
Proposition~\ref{abelian-wfs-and-cotorsion-pairs}.
 Then the weak factorization system $(\cL,\cR)$ is cofibrantly generated
if and only if the cotorsion pair $(\sL,\sR)$ is generated by some
set of objects.
\end{lem}

\begin{proof}
 This is~\cite[Lemma~3.7]{PS4}.
\end{proof}

\subsection{Model categories} \label{model-categories-subsecn}
 The concept of a model category is due to Quillen~\cite{Quil}.
 There are several modern expositions available;
see, e.~g.,~\cite{Hov0}.
 A discussion penned by the present authors can be found
in~\cite[Section~4]{PS4}.
 In this section, we restrict ourselves to a brief sketch.

 A \emph{model structure} on a category $\sE$ is the datum of
three classes of morphisms $\cL$, $\cR$, and $\cW$ satisfying
the following conditions:
\begin{itemize}
\item the pair of classes of morphisms $\cL$ and $\cR\cap\cW$ is
a weak factorization system in~$\sE$;
\item the pair of classes of morphisms $\cL\cap\cW$ and $\cR$ is
a weak factorization system in~$\sE$;
\item the class $\cW$ is closed under retracts and satisfies
the 2\+out-of\+3 property: given a composable pair of morphisms
$f$ and~$g$ in $\sE$, if two of the three morphisms $f$, $g$,
and~$fg$ belong to $\cW$, then the third one also does.
\end{itemize}
 The morphisms from the classes $\cL$, $\cR$, and $\cW$ are called
\emph{cofibrations}, \emph{fibrations}, and \emph{weak equivalences},
respectively.
 The morphisms from the class $\cL\cap\cW$ are called \emph{trivial
cofibrations}.
 The morphisms from the class $\cR\cap\cW$ are called \emph{trivial
fibrations}.

 A \emph{model category} is a complete, cocomplete category with
a model structure.
 The \emph{homotopy category} $\sE[\cW^{-1}]$ of a category $\sE$ with
a model structure is obtained by adjoining the formal inverse morphisms
for all weak equivalences.

 Let $\sE$ be a category with a model structure.
 Assume that $\sE$ has an initial object~$\varnothing$ and a terminal
object~$*$.
 Then an object $L\in\sE$ is said to be \emph{cofibrant} if
the morphism $\varnothing\rarrow L$ is a cofibration.
 Dually, an object $R\in\sE$ is said to be \emph{fibrant} if
the morphism $R\rarrow*$ is a fibration.

 Assume further that $\sE$ is a pointed category, i.~e., it has
an initial and a terminal object and they coincide, $\varnothing=*$.
 The initial-terminal object of a pointed category is called
the \emph{zero object} and denoted by~$0$.
 Given a pointed category $\sE$ with a model structure, an object
$W\in\sE$ is called \emph{weakly trivial} if the morphism $0\rarrow W$
is a weak equivalence, or equivalently, the morphism $W\rarrow0$ is
a weak equivalence.
 Weakly trivial cofibrant objects are called \emph{trivially cofibrant},
and weakly trivial fibrant objects are called \emph{trivially fibrant}.

\subsection{Abelian model structures} \label{abelian-model-structures-subsecn}
 The following definition of an \emph{abelian model structure} is
due to Hovey~\cite{Hov} (a generalization to exact categories in
the sense of Quillen can be found in~\cite{Gil0}).
 Let $\sE$ be an abelian category.
 A model structure $(\cL,\cW,\cR)$ on $\sE$ is said to be \emph{abelian}
if the class of all cofibrations $\cL$ is the class of all monomorphisms
with cofibrant cokernels and the class of all fibrations $\cR$ is
the class of all epimorphisms with fibrant kernels.
 A model category is called \emph{abelian} if its underlying category
is abelian and the model structure is abelian.

 Let $\sE$ be an abelian category.
 The following lemma tells that a model structure $(\cL,\cW,\cR)$ on
$\sE$ is abelian if and only if both the weak factorization systems
$(\cL,\>\cR\cap\cW)$ and $(\cL\cap\cW,\>\cR)$ are abelian.

\begin{lem}
 In an abelian model structure, the class of all trivial cofibrations\/
$\cL\cap\cW$ is the class of all monomorphisms with trivially cofibrant
cokernels and the class of all trivial fibrations\/ $\cR\cap\cW$ is
the class of all epimorphisms with trivially fibrant kernels.
\end{lem}

\begin{proof}
 This is a part of~\cite[Proposition~4.2]{Hov}; see
also~\cite[Lemma~4.1]{PS4}.
\end{proof}

 A class of objects $\sW$ in an abelian category $\sE$ is said to be
\emph{thick} if it is closed under direct summands and satisfies
the following version of 2\+out-of\+3 property: for any short exact
sequence $0\rarrow A\rarrow B\rarrow C\rarrow0$ in $\sE$, if two of
the three objects $A$, $B$, $C$ belong to $\sW$, then the third one
also does.

\begin{thm} \label{abelian-model-structures}
 Let\/ $\sE$ be an abelian category.
 Then there is a bijection between abelian model structures on\/
$\sE$ and triples of classes of objects $(\sL,\sW,\sR)$ in\/ $\sE$
such that
\begin{itemize}
\item $(\sL,\>\sR\cap\sW)$ is a complete cotorsion pair in\/~$\sE$;
\item $(\sL\cap\sW,\>\sR)$ is a complete cotorsion pair in\/~$\sE$;
\item $\sW$ is a thick class of objects in\/~$\sE$.
\end{itemize}
 To an abelian model structure $(\cL,\cW,\cR)$, the triple of
classes of objects $(\sL,\sW,\sR)$ is assigned, where\/ $\sL$ is
the class of all cofibrant objects, $\sR$ is the class of all
fibrant objects, and $\sW$ is the class of all weakly trivial objects.
 To a triple of classes of objects $(\sL,\sW,\sR)$ satisfying
the conditions above, the abelian model structure $(\cL,\cW,\cR)$ is
assigned, where\/ $\cL=\sL\Mono$ and $\cR=\sR\Epi$.
 The classes of all trivial cofibrations and trivial fibrations are
$\cL\cap\cW=(\sL\cap\sW)\Mono$ and $\cR\cap\cW=(\sR\cap\sW)\Epi$.
 The class of all weak equivalences\/ $\cW$ consists of all
morphisms~$w$ in\/ $\sE$ that can be factorized as $w=rl$ with
$r\in\cR\cap\cW$ and $l\in\cL\cap\cW$.
\end{thm}

\begin{proof}
 This is a part of~\cite[Theorem~2.2]{Hov}; see
also~\cite[Theorem~4.2]{PS4}.
 A generalization to Quillen exact categories can be found
in~\cite[Theorem~3.3]{Gil0}.
\end{proof}

 In the rest of this paper, we will identify abelian model structures
with the triples of classes of objects $(\sL,\sW,\sR)$ satisfying
the conditions of Theorem~\ref{abelian-model-structures}, and write
simply ``an abelian model structure $(\sL,\sW,\sR)$ on an abelian
category~$\sE$\,''.

 A model structure $(\cL,\cW,\cR)$ is called \emph{cofibrantly
generated} if both the weak factorization systems $(\cL,\>\cR\cap\cW)$
and $(\cL\cap\cW,\>\cR)$ are cofibrantly generated.

\begin{lem} \label{cofibrantly-generated-abelian-model}
 An abelian model structure $(\sL,\sW,\sR)$ on a locally presentable
abelian category\/ $\sE$ is cofibrantly generated if and only if both
the cotorsion pairs $(\sL,\>\sR\cap\sW)$ and $(\sL\cap\sW,\>\sR)$
are generated by some sets of objects in\/~$\sE$.
\end{lem}

\begin{proof}
 This is~\cite[Corollary~4.3]{PS4}.
 The assertion follows from
Lemma~\ref{cofibrantly-generated-abelian-wfs}.
\end{proof}

\subsection{Projective and injective abelian model structures}
 Given an abelian category $\sE$, we denote by $\sE_\proj\subset\sE$
the class (or the full subcategory) of all projective objects in $\sE$
and by $\sE_\inj\subset\sE$ the class/full subcategory of all
injective objects in~$\sE$.

 A model structure $(\cL,\cW,\cR)$ on a category $\sB$ is called
\emph{projective} if all the objects of $\sB$ are fibrant.
 For an abelian model structure $(\sL,\sW,\sR)$, this means that
$\sR=\sB$, or equivalently, $\sL\cap\sW$ is the class of all
projective objects in $\sB$, that is $\sL\cap\sW=\sB_\proj$.
 So an abelian model structure is projective if and only if
the trivial cofibrations are the monomorphisms with projective
cokernel.

 Dually, a model structure $(\cL,\cW,\cR)$ on a category $\sA$ is
called \emph{injective} if all the objects of $\sA$ are cofibrant.
 For an abelian model structure $(\sL,\sW,\sR)$, this means that
$\sL=\sA$, or equivalently, $\sR\cap\sW$ is the class of all
injective objects in $\sA$, that is $\sR\cap\sW=\sA_\inj$.
 So an abelian model structure is injective if and only if
the trivial fibrations are the epimorphisms with injective kernel.

\begin{lem} \label{hereditary-abelian-model}
 Let $(\sL,\sW,\sR)$ be an abelian model structure on an abelian
category\/~$\sE$.
 Then the cotorsion pair $(\sL,\>\sR\cap\sW)$ is hereditary if and only
if the cotorsion pair $(\sL\cap\sW,\>\sR)$ is hereditary in\/~$\sE$.
\end{lem}

\begin{proof}
 This standard observation can be found in~\cite[Lemma~4.4]{PS4}.
 The assertion follows easily from Lemma~\ref{hereditary-cotorsion}
above.
\end{proof}

 An abelian model structure satisfying the equivalent conditions
of Lemma~\ref{hereditary-abelian-model} is called \emph{hereditary}.
 In particular, all projective and all injective abelian model
structures are hereditary~\cite[Corollary~1.1.12]{Bec} (see
also~\cite[paragraph before Lemma~4.5]{PS4}).

\begin{lem} \label{projective-injective-abelian-model-structures}
\textup{(a)} A pair of classes of objects $(\sL,\sW)$ in an abelian
category\/ $\sB$ defines a projective abelian model structure
$(\sL,\sW,\sB)$ on\/ $\sB$ if and only if\/ $\sB$ has enough projective
objects, $(\sL,\sW)$ is a complete cotorsion pair in\/ $\sB$ with\/
$\sL\cap\sW=\sB_\proj$, and the class\/ $\sW\subset\sB$ is thick. \par
\textup{(b)} A pair of classes of objects $(\sW,\sR)$ in an abelian
category\/ $\sA$ defines an injective abelian model structure
$(\sA,\sW,\sR)$ on\/ $\sA$ if and only if\/ $\sA$ has enough injective
objects, $(\sW,\sR)$ is a complete cotorsion pair in\/ $\sA$ with\/
$\sR\cap\sW=\sA_\inj$, and the class\/ $\sW\subset\sA$ is thick.
\end{lem}

\begin{proof}
 This is~\cite[Corollary~1.1.9]{Bec}; see also~\cite[Lemma~4.5]{PS4}.
\end{proof}

\begin{lem} \label{W-is-thick-lemma}
\textup{(a)} Let $(\sL,\sW)$ be a hereditary complete cotorsion pair
in an abelian category\/ $\sB$ such that\/ $\sL\cap\sW=\sB_\proj$.
 Then the class\/ $\sW$ is thick in\/~$\sB$. \par
\textup{(b)} Let $(\sW,\sR)$ be a hereditary complete cotorsion pair
in an abelian category\/ $\sA$ such that\/ $\sR\cap\sW=\sA_\inj$.
 Then the class\/ $\sW$ is thick in\/~$\sA$.
\end{lem}

\begin{proof}
 This is~\cite[Lemma~1.1.10]{Bec}; see also~\cite[Lemma~4.6]{PS4}.
\end{proof}

\subsection{Stable combinatorial abelian model structures}
 Let $\sE$ be a finitely complete, finitely cocomplete category
(i.~e., all finite limits and finite colimits exist in~$\sE$).
 Assume further that $\sE$ is a pointed category (as defined in
Section~\ref{model-categories-subsecn}).
 In particular, all abelian categories satisfy these conditions.

 Let $(\cL,\cW,\cR)$ be a model structure on~$\sE$.
 In this context, Quillen~\cite[Theorem~I.2.2]{Quil} constructed
a pair of adjoint endofunctors $\Sigma$ and $\Omega\:\sE[\cW^{-1}]
\rarrow\sE[\cW^{-1}]$ on the homotopy category~$\sE[\cW^{-1}]$.
 The functor $\Sigma$ is called the \emph{suspension functor}
and the functor $\Omega$ is the \emph{loop functor}.

 Under stricter assumptions that $\sE$ is complete and cocomplete
and that factorizations of morphisms in the weak factorization
systems $(\cL,\>\cR\cap\cW)$ and $(\cL\cap\cW,\>\cR)$ can be
chosen to be functorial, Hovey~\cite[Chapter~5]{Hov0} constructs
a natural closed action of the homotopy category of pointed
simplicial sets on $\sE[\cW^{-1}]$ and uses this action to define
the suspension and loop functors in~\cite[Definition~6.1.1]{Hov0}.
 Then~\cite[Lemma~6.1.2]{Hov0} tells that the functor $\Sigma$
is left adjoint to~$\Omega$.

 The model structure $(\cL,\cW,\cR)$ is called \emph{stable} if
the functor $\Sigma$ is an auto-equiv\-a\-lence of the category
$\sE[\cW^{-1}]$, or equivalently, the functor $\Omega$ is
an auto-equivalence of $\sE[\cW^{-1}]$.
 It is well-known that, for any stable model structure $(\cL,\cW,\cR)$
on a finitely complete, finitely cocomplete, pointed category $\sE$,
the homotopy category $\sE[\cW^{-1}]$ is
triangulated (see, e.~g.,~\cite[Proposition~7.1.6]{Hov0}).
 The suspension functor $\Sigma$ plays the role of the shift functor
$X\longmapsto X[1]$ on $\sE[\cW^{-1}]$. 

\begin{lem} \label{hereditary-is-stable}
 Any hereditary abelian model structure is stable.
\end{lem}

\begin{proof}
 See~\cite[Corollary~1.1.15 and the preceding discussion]{Bec}.
\end{proof}

 We refer to~\cite{Neem1} or~\cite{Kra1,Kra2} for the definition of
a \emph{well-generated triangulated category} (see~\cite[Section~4]{PS4}
for a brief discussion).
 A \emph{combinatorial model category} $\sE$ is a locally presentable
category endowed with a cofibrantly generated model structure.
 (Notice that any locally presentable category is complete and
cocomplete, so our terminology here is consistent with the one
in Section~\ref{model-categories-subsecn}.)

\begin{prop} \label{stable-combinatorial-well-generated}
 For any stable combinatorial model category\/ $\sE$, the homotopy
category\/ $\sE[\cW^{-1}]$ is a well-generated triangulated category.
\end{prop}

\begin{proof}
 This is~\cite[Proposition~6.10]{Ros} or~\cite[Theorem~3.1,
Definition~3.2, and Theorem~3.9]{CR}.
\end{proof}

\Section{Contraderived Model Structure}  \label{contraderived-secn}

 The aim of this is section is to work out a common generalization
of the contraderived model structure on the abelian category of
CDG\+modules over a CDG\+ring~\cite[Proposition~1.3.6(1)]{Bec}
and the contraderived model structure on the category of complexes
over a locally presentable abelian category with enough projective
objects~\cite[Section~7]{PS4}.
 The locally presentable abelian DG\+categories with enough projective
objects (to be defined in this section) form a suitable context.

\subsection{Projective objects in abelian DG-categories}
\label{projectives-in-abelian-DG-categories-subsecn}
 We start with a general lemma connecting the groups $\Ext^1$ in
the abelian category $\sZ^0(\bE)$ with the groups $\Hom$ in
the triangulated category $\sH^0(\bE)$ for an abelian
DG\+category~$\bE$.

\begin{lem} \label{Ext-in-cocycles-and-Hom-in-homotopy}
 Let\/ $\bE$ be an abelian DG\+category and $X$, $Y\in\bE$ be
two objects.
 Then the kernel of the abelian group homomorphism
$$
 \Phi\:\Ext^1_{\sZ^0(\bE)}(X,Y)\lrarrow
 \Ext^1_{\sZ^0(\bE^\bec)}(\Phi(X),\Phi(Y))
$$
induced by the exact functor\/ $\Phi_\bE\:\sZ^0(\bE)\rarrow
\sZ^0(\bE^\bec)$ is naturally isomorphic to the group\/
$\Hom_{\sH^0(\bE)}(X,Y[1])$.
 In particular, if either\/ $\Phi(X)$ is a projective object in\/
$\sZ^0(\bE^\bec)$, or\/ $\Phi(Y)$ is an injective object in\/
$\sZ^0(\bE^\bec)$, then
$$
 \Ext^1_{\sZ^0(\bE)}(X,Y)\simeq\Hom_{\sH^0(\bE)}(X,Y[1]).
$$
\end{lem}

\begin{proof}
 This is a generalization of a well-known~\cite[Lemma~5.1]{PS4}
and a particular case of~\cite[Lemma~9.41]{Pedg}.
 Following Section~\ref{tilde-Phi-and-Psi-subsecn}, the additive
functor $\Phi\:\sZ^0(\bE)\rarrow\sZ^0(\bE^\bec)$ factorizes as
$\sZ^0(\bE)\rarrow\bE^0\rarrow\sZ^0(\bE^\bec)$, where
$\sZ^0(\bE)\rightarrowtail\bE^0$ is the natural inclusion
(bijective on the objects) and $\widetilde\Phi\:\bE^0\rarrow
\sZ^0(\bE^\bec)$ is a fully faithful functor.
 Consequently, a short exact sequence in $\sZ^0(\bE)$ splits
after applying the functor $\Phi$ if and only if it splits in~$\bE^0$.

 It remains to refer to the (well-known and easy) description of
short sequences $0\rarrow Y\rarrow W\rarrow X\rarrow0$ in $\sZ^0(\bE)$
that are split exact in $\bE^0$ in terms of the cones of
closed morphisms $f\:X\rarrow Y[1]$ of degree~$0$ in~$\bE$ (see
Section~\ref{twists-and-cones-subsecn}).
 Notice that, for any such closed morphism~$f$, the related cone
sequence $0\rarrow Y\rarrow\cone(f)[-1]\rarrow X\rarrow0$ is exact
in $\sZ^0(\bE)$ by Lemma~\ref{cone-kernel-cokernel}.
 One also needs to observe that two closed morphisms~$f'$
and $f''\:X\rightrightarrows Y[1]$ of degree~$0$ in $\bE$ lead to
equivalent extensions if and only if they are homotopic.
\end{proof}

\begin{lem} \label{abelian-DG-enough-projectives}
 Let\/ $\bB$ be an abelian DG\+category.
 Then the abelian category\/ $\sZ^0(\bB)$ has enough projective objects
if and only if the abelian category\/ $\sZ^0(\bB^\bec)$ has enough
projective objects.
\end{lem}

\begin{proof}
 The functors $\Phi_\bB\:\sZ^0(\bB)\rarrow\sZ^0(\bB^\bec)$ and
$\Psi^+_\bB\:\sZ^0(\bB^\bec)\rarrow\sZ^0(\bB)$ take projectives
to projectives, because they are left adjoint to exact functors.
 Furthermore, the roles of the functors $\Phi$ and $\Psi^\pm$ are
completely symmetric in view of
Corollary~\ref{becbec-equivalence-for-abelian-DG} and
the last paragraph of Section~\ref{DG-functor-becbec-subsecn}.
 Assume that the abelian category $\sZ^0(\bB^\bec)$ has enough
projective objects, and let $B\in\sZ^0(\bB)$ be an object.
 Choose a projective object $P\in\sZ^0(\bB^\bec)$ together with
an epimorphism $p\:P\rarrow\Phi(B)$ in $\sZ^0(\bB^\bec)$.
 By adjunction, we have the corresponding morphism $q\:\Psi^+(P)\rarrow
B$ in $\sZ^0(\bB)$.
 Let us show that~$q$ is an epimorphism.

 Indeed, consider the morphism $\Psi^+(p)\:\Psi^+(P)\rarrow
\Psi^+\Phi(B)$ in $\sZ^0(\bB)$.
 The functor $\Psi^+$ is exact, since it has adjoints on both sides;
so in particular $\Psi^+(p)$ is an epimorphism.
 The morphism~$q$ is equal to the composition $\Psi^+(P)\rarrow
\Psi^+\Phi(B)\rarrow B$, where $\Psi^+\Phi(B)\rarrow B$ is
the adjunction morphism.
 According to the last paragraph of Section~\ref{Phi-and-Psi-subsecn},
we have a natural isomorphism $\Psi^+\Phi(B)\simeq\Xi(B)$
in $\sZ^0(\bB)$.
 It remains to recall that the natural morphism $\Xi(B)\rarrow B$
is an epimorphism in $\sZ^0(\bB)$ by
Corollary~\ref{Xi-extension-cokernel}(a).
 (Cf.~\cite[Lemma~5.7(b)]{Pedg}.)
\end{proof}

 An abelian DG\+category satisfying the equivalent conditions of
Lemma~\ref{abelian-DG-enough-projectives} is said to 
\emph{have enough projective objects}.

\begin{lem} \label{projectives-are-contractible-graded-projectives}
 Let\/ $\bB$ be an abelian DG\+category with enough projective objects.
 Then \par
\textup{(a)} an object $Q\in\sZ^0(\bB)$ is projective if and only
if $Q$ is contractible in\/ $\bB$ \emph{and} the object\/
$\Phi(Q)\in\sZ^0(\bB^\bec)$ is projective; \par
\textup{(b)} an object $P\in\sZ^0(\bB^\bec)$ is projective if and
only if $P$ is contractible in\/ $\bB^\bec$ \emph{and} the object\/
$\Psi^+(P)\in\sZ^0(\bB)$ is projective.
\end{lem}

\begin{proof}
 This is a common generalization of~\cite[Lemma~1.3.3]{Bec}
(which is the CDG\+mod\-ule version) and of the long
well known~\cite[Lemma~5.2(b)]{PS4} (which is the version for
complexes in abelian categories).
 Let us prove part~(a).
 
 ``Only if'': let $Q$ be a projective object in $\sZ^0(\bB)$.
 As explained in the proof of Lemma~\ref{abelian-DG-enough-projectives},
the functor $\Phi$ preserves projectivity; so $\Phi(Q)\in
\sZ^0(\bB^\bec)$ is a projective object.
 Furthermore, following the same proof, there are enough projective
objects in $\sZ^0(\bB)$ of the form $\Psi^+(P)$, where $P$ ranges over
the projective objects of $\sZ^0(\bB^\bec)$.
 Therefore, $Q$ is a direct summand of $\Psi^+(P)$ in $\sZ^0(\bB)$
for some projective $P\in\sZ^0(\bB^\bec)$.
 It remains to observe that the essential image of the functor $\Psi^+$
consists of contractible objects (by construction) and direct  summands
of contractible objects are contractible in order to conclude that
the object $Q\in\bB$ is contractible.
 (Similarly, in view of a natural isomorphism in the last paragraph 
of Section~\ref{DG-functor-becbec-subsecn}, the essential image of
the functor $\Phi$ consists of contractible objects.)

 ``If'': let $Q\in\bB$ be a contractible object such that the object
$\Phi(Q)\in\sZ^0(\bB^\bec)$ is projective.
 For any object $Y\in\sZ^0(\bB)$, we have
$\Ext^1_{\sZ^0(\bB)}(Q,Y)\simeq\Hom_{\sH^0(\bB)}(Q,Y[1])$ by
Lemma~\ref{Ext-in-cocycles-and-Hom-in-homotopy}.
 Since the object $Q$ vanishes in $\sH^0(\bB)$, it follows that
$\Ext^1_{\sZ^0(\bB)}(Q,Y)=0$ for all $Y\in\sZ^0(\bB)$, hence
$Q\in\sZ^0(\bB)$ is a projective object.
\end{proof}

 Let $\bB$ be an abelian DG\+category.
 We will say that an object $Q\in\bB$ is \emph{graded-projective}
if the object $\Phi(Q)$ is projective in the abelian category
$\sZ^0(\bB^\bec)$.
 The full DG\+subcategory formed by the graded-projective objects
in $\bB$ is denoted by $\bB_\bproj\subset\bB$.
 So the notation $\sZ^0(\bB_\bproj)\subset\sZ^0(\bB)$ stands for
the full subcategory of all graded-projective objects in $\sZ^0(\bB)$,
while $\sZ^0(\bB)_\proj\subset\sZ^0(\bB)$ is the (smaller) full
subcategory of all projective objects in $\sZ^0(\bB)$.
 Since the functor $\Phi$ is additive and transforms shifts into
inverse shifts and twists into isomorphisms, the full DG\+subcategory
$\bB_\bproj$ is closed under finite direct sums, shifts, twists,
and cones in~$\bB$.

\begin{lem} \label{cotorsion-pairs-with-graded-projectives}
 Let\/ $\bB$ be an abelian DG\+category with enough projective objects,
and let $(\sL,\sW)$ be a cotorsion pair in the abelian category\/
$\sZ^0(\bB)$ such that\/ $\sL\subset\sZ^0(\bB_\bproj)$.
 Assume that the cotorsion pair $(\sL,\sW)$ is preserved by the shift:\/
$\sW=\sW[1]$, or equivalently, $\sL=\sL[1]$.
 Then\/ $\sL\cap\sW=\sZ^0(\bB)_\proj$.
 If, moreover, the cotorsion pair $(\sL,\sW)$ is complete, then it is
hereditary and the class\/ $\sW$ is thick in\/ $\sZ^0(\bB)$.
\end{lem}

\begin{proof}
 This is a common generalization of~\cite[Lemma~1.3.4(1)]{Bec}
(the CDG\+module version) and~\cite[Lemma~5.3(b)]{PS4} (the version
for complexes in abelian categories).

 The inclusion $\sZ^0(\bB)_\proj\subset{}^{\perp_1}\sW=\sL$ is obvious.
 The inclusion $\sZ^0(\bB)_\proj\subset\sZ^0(\bB_\bproj)^{\perp_1}$
holds because $\sZ^0(\bB)_\proj$ is the full subcategory of
projective-injective objects in the Frobenius exact category
$\sZ^0(\bB_\bproj)$ (with the exact structure inherited from
the abelian exact structure of $\sZ^0(\bB)$).
 Specifically, for any $P\in\sZ^0(\bB_\bproj)$ and
$Q\in\sZ^0(\bB)_\proj$ we have, by
Lemmas~\ref{Ext-in-cocycles-and-Hom-in-homotopy}
and~\ref{projectives-are-contractible-graded-projectives}(a),
$\Ext^1_{\sZ^0(\bB)}(P,Q)\simeq\Hom_{\sH^0(\bB)}(P,Q[1])=0$
since the object $Q\in\bB$ is contractible.
 Hence $\sZ^0(\bB)_\proj\subset\sL^{\perp_1}=\sW$.

 To prove the inclusion $\sL\cap\sW\subset\sZ^0(\bB)_\proj$, it
suffices to show that every object $Q\in\sL\cap\sW$ is contractible
in~$\bB$ (as we already know that $\sL\cap\sW\subset\sL\subset
\sZ^0(\bB_\bproj)$, and all contractible objects
in $\sZ^0(\bB_\bproj)$ belong to $\sZ^0(\bB)_\proj$ by
Lemma~\ref{projectives-are-contractible-graded-projectives}(a)).
 Indeed, by Lemma~\ref{Ext-in-cocycles-and-Hom-in-homotopy},
$\Hom_{\sH^0(\bB)}(Q,Q)=\Ext^1_{\sZ^0(\bB)}(Q,Q[-1])=0$.

 To prove the remaining assertions of the lemma, we observe that
the class $\sL$ is closed under syzygies in $\sZ^0(\bB)$.
 Indeed, for each $P\in\sL$ we have a short exact sequence
$0\rarrow P[-1]\rarrow\Xi(P)\rarrow P\rarrow0$ in $\sZ^0(\bB)$
by Corollary~\ref{Xi-extension-cokernel}(a) and an isomorphism
$\Xi(P)\simeq\Psi^+\Phi(P)$ by the last paragraph of
Section~\ref{Phi-and-Psi-subsecn}.
 Now $\Phi(P)$ is a projective object in $\sZ^0(\bB^\bec)$, hence
$\Psi^+\Phi(P)\in\sZ^0(\bB)_\proj$ as explained in the proof of
Lemma~\ref{abelian-DG-enough-projectives}.
 It follows by a standard dimension shifting argument that
the cotorsion pair $(\sL,\sW)$ is hereditary.
 In order to conclude that the class $\sW$ is thick, it remains
to apply Lemma~\ref{W-is-thick-lemma}(a).
\end{proof}

\subsection{Locally presentable DG-categories}
\label{locally-presentable-DG-categories-subsecn}
 This section is mostly an extraction
from~\cite[Sections~9.1 and~9.2]{Pedg}.

 Let $\lambda$~be an infinite regular cardinal.
 The definitions of $\lambda$\+presentable objects and locally
$\lambda$\+presentable categories (based on the exposition in~\cite{AR})
have already appeared in Section~\ref{small-object-subsecn}.

\begin{lem} \label{coproducts-in-DG-categories}
 For any additive DG\+category\/ $\bE$ with shifts and cones,
coproducts in the additive category\/ $\sZ^0(\bE)$ are the same things
as coproducts in the DG\+category\/~$\bE$.
 (Infinite) coproducts exist in the DG\+category\/ $\bE$ if and only if
they exist in the additive category\/ $\sZ^0(\bE)$.
 If this is the case, then coproducts also exist in the DG\+category\/
$\bE^\bec$ and in the additive category\/ $\sZ^0(\bE^\bec)$, and
the functors\/ $\Phi_\bE$ and\/ $\Psi^\pm_\bE$ preserve them.
\end{lem}

\begin{proof}
 The functors $\Phi$ and $\Psi^\pm$ always preserve limits and
colimits, since they have adjoint functors on both sides.
 The assertion that coproducts in $\bE$ (as defined above in
Section~\ref{co-products-in-DG-categories}) are the same things as
coproducts in $\sZ^0(\bE)$ when $\bE$ has shifts and cones is not
immediately obvious; it is~\cite[Lemma~9.2]{Pedg}.
 The key observation is that for each $n\in\boZ$ and $X$, $Y\in\bE$,
one has a natural isomorphism $\Hom^n_\bE(X,Y)\simeq
\sZ^0\Hom^\bu_\bE(X,\cone(\id_Y)[n])$.
 The coproducts in $\bE^\bec$ are easily constructed in terms of
those in $\bE$ when the latter exist.
\end{proof}

\begin{lem} \label{Xi-reflects-presentability}
 Let\/ $\bE$ be an additive DG\+category with shifts and cones such
that all colimits exist in the additive category\/~$\sZ^0(\bE)$.
 Then the additive functor\/ $\Xi_\bE\:\sZ^0(\bE)\rarrow\sZ^0(\bE)$
\emph{reflects} $\lambda$\+presentability of objects.
 In other words, if $A\in\bE$ is an object such that the object\/
$\Xi(A)$ is\/ $\lambda$\+presentable in\/ $\sZ^0(\bE)$, then
the object $A$ is\/ $\lambda$\+presentable in\/~$\sZ^0(\bE)$.
\end{lem}

\begin{proof}
 The class of all $\lambda$\+presentable objects is closed under
$\lambda$\+small colimits by~\cite[Proposition~1.16]{AR}; in
particular, it is closed under cokernels.
 Thus the assertion follows from
Corollary~\ref{Xi-extension-cokernel}(b).
\end{proof}

\begin{lem} \label{Phi-Psi-lambda-presentability}
 Let\/ $\bE$ be an additive DG\+category with shifts and cones such
that all colimits exist in the additive category\/~$\sZ^0(\bE)$.
 Then the additive functors\/ $\Phi_\bE\:\sZ^0(\bE)\rarrow
\sZ^0(\bE^\bec)$ and\/ $\Psi^+_\bE\:\sZ^0(\bE^\bec)\rarrow\sZ^0(\bE)$
preserve \emph{and} reflect\/ $\lambda$\+presentability of objects.
\end{lem}

\begin{proof}
 This is explained in~\cite[Lemmas~9.6 and~9.7]{Pedg}.
 Under the assumptions of the lemma, the additive category
$\sZ^0(\bE^\bec)$ has all coproducts by
Lemma~\ref{coproducts-in-DG-categories} and
all cokernels by Lemma~\ref{Psi-reflects-existence-of-co-kernels}.
 Therefore, all colimits exist in $\sZ^0(\bE^\bec)$.
 The functors $\Phi_\bE$ and $\Psi^+_\bE$ preserve
$\lambda$\+presentability of objects, since they are left adjoint
to functors preserving $\lambda$\+directed colimits.
 Since the compositions $\Xi_\bE\simeq\Psi^+_\bE\circ\Phi_\bE$ and
$\Xi_{\bE^\bec}\simeq\Phi_\bE\circ\Psi^-_\bE$ (see the last paragraph
of Section~\ref{Phi-and-Psi-subsecn}) reflect $\lambda$\+presentability
by Lemma~\ref{Xi-reflects-presentability}, it follows that
the functors $\Phi_\bE$ and $\Psi^-_\bE$ also reflect
$\lambda$\+presentability.
\end{proof}

\begin{lem} \label{Phi-Psi-strong-generators}
 Let\/ $\bE$ be an additive DG\+category with shifts and cones.
 Then the additive functors\/ $\Phi_\bE\:\sZ^0(\bE)\rarrow
\sZ^0(\bE^\bec)$ and\/ $\Psi^+_\bE\:\sZ^0(\bE^\bec)\rarrow\sZ^0(\bE)$
take (strongly) generating sets of objects to (strongly) generating
sets of objects.
\end{lem}

\begin{proof}
 This is~\cite[Lemma~9.1]{Pedg}.
 The assertion about generating sets of objects holds because
the functors $\Phi$ and $\Psi^+$ have faithful right adjoints.
 The assertion about strongly generating sets of objects is valid
because the right adjoint functors to $\Phi$ and $\Psi^+$
are also conservative
(but recall also the discussion of generators and strong generators
in Section~\ref{small-object-subsecn}).
\end{proof}

\begin{prop} \label{locally-presentable-DG-categories}
 Let\/ $\bE$ be an additive DG\+category with shifts and cones such
that the additive category\/ $\sZ^0(\bE)$ has all colimits.
 Then the additive category\/ $\sZ^0(\bE)$ is locally\/
$\lambda$\+presentable if and only if the additive category\/
$\sZ^0(\bE^\bec)$ is locally\/ $\lambda$\+presentable.
\end{prop}

\begin{proof}
 Under the assumptions of the proposition, the additive category
$\sZ^0(\bE^\bec)$ also has all colimits, as it was explained in
the proof of Lemma~\ref{Phi-Psi-lambda-presentability}.
 Now applying the functor $\Phi$ to any strong generating set of
$\lambda$\+presentable objects in $\sZ^0(\bE)$ produces a strong
generating set of $\lambda$\+presentable objects in $\sZ^0(\bE^\bec)$,
and conversely with the functor~$\Psi^+$ (by
Lemmas~\ref{Phi-Psi-lambda-presentability}
and~\ref{Phi-Psi-strong-generators}).
\end{proof}

 We will say that an additive DG\+category $\bE$ with shifts and cones
is \emph{locally $\lambda$\+pre\-sent\-able} if both the additive
categories $\sZ^0(\bE)$ and $\sZ^0(\bE^\bec)$ are locally
$\lambda$\+pre\-sent\-able.
 In other words, this means that $\bE$ satisfies the assumptions and
the equivalent conditions of
Proposition~\ref{locally-presentable-DG-categories}.
 A DG\+category is called \emph{locally presentable} if it is locally
$\lambda$\+presentable for some infinite regular cardinal~$\lambda$.

 In particular, locally $\omega$\+presentable DG\+categories (where
$\omega$~is the minimal infinite cardinal) are called \emph{locally
finitely presentable}.

\begin{exs} \label{locally-presentable-abelian-DG-of-complexes}
 (1)~Let $\sE$ be a locally $\lambda$\+presentable additive category.
 Then it is clear that the additive category of graded objects $\sG(\sE)
=\sE^\boZ$ is locally $\lambda$\+presentable.
 For an uncountable regular cardinal~$\lambda$, an object of $\sG(\sE)$
is $\lambda$\+presentable if and only if all its grading components
are $\lambda$\+presentable (see
Example~\ref{locally-coherent-DG-category-of-complexes}(1) below for
a discussion of the case $\lambda=\omega$).
 The additive category of complexes $\sC(\sE)=\sZ^0(\bC(\sE))$ is
locally $\lambda$\+presentable by~\cite[Lemma~6.3]{PS4}.
 We recall that the additive category $\sZ^0(\bC(\sE)^\bec)$ is
equivalent to $\sG(\sE)$ according to
Section~\ref{almost-involution-complexes-subsecn}.
 Hence the DG\+category $\bC(\sE)$ of complexes in $\sE$ is locally
$\lambda$\+presentable.

 (2)~Let $\sB$ be a locally presentable abelian category with enough
projective objects.
 Then it is obvious that the abelian category of graded objects
$\sG(\sB)=\sB^\boZ$ also has enough projectives.
 Following Example~\ref{abelian-DG-category-of-complexes-example},
the DG\+category of complexes $\bC(\sB)$ is an abelian DG\+category.
 Thus $\bC(\sB)$ is a locally presentable abelian DG\+category with
enough projective objects.
\end{exs}

\begin{ex} \label{locally-presentable-abelian-DG-of-CDG-modules}
 Let $\biR^\cu=(R^*,d,h)$ be a CDG\+ring.
 Then the DG\+category $\biR^\cu\bModl$ of left CDG\+modules over
$\biR^\cu$ is an abelian DG\+category
(see Example~\ref{abelian-DG-category-of-CDG-modules-example}).
 Moreover, following
Section~\ref{almost-involution-CDG-modules-subsecn}, both the abelian
categories $\sZ^0(\biR^\cu\bModl)$ and $\sZ^0((\biR^\cu\bModl)^\bec)$ are
graded module categories, $\sZ^0(\biR^\cu\bModl)=\widehat\biR^*\bModl$
and $\sZ^0((\biR^\cu\bModl)^\bec)\simeq R^*\bModl$.
 The abelian category of graded modules over any graded ring is
locally finitely presentable with enough projective objects; hence
$\biR^\cu\bModl$ is a locally finitely presentable abelian DG\+category
with enough projective objects.
\end{ex}

\begin{exs} \label{locally-presentable-abelian-DG-of-factorizations}
 (1)~Let $\sE$ be a locally $\lambda$\+presentable additive category
and $\Delta\:\sE\rarrow\sE$ be an autoequivalence.
 Then it is clear that the additive category of $2$\+$\Delta$-periodic
objects $\sP(\sE,\Delta)\simeq\sE\times\sE$ is locally
$\lambda$\+presentable.
 For any potential $v\:\Id_\sE\rarrow\Delta$, the additive category of
factorizations $\sF(\sE,\Delta,v)$ has all colimits (since
the category $\sE$ does).
 Recall that the additive category $\sZ^0(\bF(\sE,\Delta,v)^\bec)$ is
equivalent to $\sP(\sE,\Delta)$ according to
Section~\ref{almost-involution-factorizations-subsecn}.
 Thus the DG\+category $\bF(\sE,\Delta,v)$ of factorizations of~$v$
in $\sE$ is locally $\lambda$\+presentable by
Proposition~\ref{locally-presentable-DG-categories}.

 (2)~Let $\sB$ be a locally presentable abelian category with enough
projective objects, $\Delta\:\sB\rarrow\sB$ be an autoequivalence, and
$v\:\Id_\sB\rarrow\Delta$ be a potential.
 Then it is obvious that the abelian category of $2$\+$\Delta$-periodic
objects $\sP(\sB,\Delta)\simeq\sB\times\sB$ also has enough projectives.
 Following Example~\ref{abelian-DG-category-of-factorizations-example},
the DG\+category of factorizations $\bF(\sB,\Delta,v)$ is an abelian
DG\+category.
 Thus $\bF(\sB,\Delta,v)$ is a locally presentable abelian
DG\+category with enough projective objects.
\end{exs}

\subsection{Becker's contraderived category}
\label{becker-contraderived-subsecn}
 In the rest of Section~\ref{contraderived-secn}, we will be mostly
working with a locally presentable abelian DG\+category $\bB$ with
enough projective objects.
 Here the definition of an abelian DG\+category was given
in Section~\ref{abelian-DG-categories-subsecn},
abelian DG\+categories with enough projective objects were defined
in Section~\ref{projectives-in-abelian-DG-categories-subsecn}, and
the definition of a locally presentable DG\+category can be found in
Section~\ref{locally-presentable-DG-categories-subsecn}.

 In this paper we are interested in the contraderived category
\emph{in the sense of Becker}~\cite{Jor,Bec,PS4}, which we denote
by $\sD^\bctr(\bB)$.
 In well-behaved cases, it is equivalent to the homotopy category of
graded-projective objects, $\sD^\bctr(\bB)\simeq\sH^0(\bB_\bproj)$.

 The contraderived category in the sense of Becker has to be
distinguished from the contraderived category in the sense of
books and papers~\cite{Psemi,Pkoszul,Pps,Prel}.
 The two definitions of a contraderived category are known to be
equivalent under certain assumptions~\cite[Section~3.8]{Pkoszul},
\cite[Corollary~4.9]{Pctrl}, \cite[Theorem~5.10(b)]{Pedg}, but it is
an open question whether they are equivalent for the category of
modules over an arbitrary ring (see~\cite[Example~2.6(3)]{Pps},
\cite[Remark~9.2]{PS4}, and~\cite[Section~7.9]{Pksurv} for a discussion).

 Notice that, in any locally presentable category, all limits exists,
and in particular all products exist, by~\cite[Corollary~2.47]{AR}.
 So the dual version of Lemma~\ref{coproducts-in-DG-categories}
tells that infinite products exist in any locally presentable
DG\+category $\bB$, and consequently they also exist in the homotopy
category~$\sH^0(\bB)$.
 Furthermore, infinite products are exact in any abelian category with
enough projective objects (see~\cite[Remark~5.2]{Pedg} for a more
general assertion concerning exact categories.)

 Let $\bB$ be an abelian DG\+category.
 An object $X\in\bB$ is said to be \emph{contraacyclic} (\emph{in
the sense of Becker}) if $\Hom_{\sH^0(\bB)}(Q,X)=0$ for all
graded-projective objects $Q\in\bB_\bproj$.
 We will denote the full subcategory of Becker-contraacyclic objects
by $\sH^0(\bB)_\ac^\bctr\subset\sH^0(\bB)$.
 Clearly, $\sH^0(\bB)_\ac^\bctr$ is a triangulated (and even thick)
subcategory in the homotopy category $\sH^0(\bB)$.
 The triangulated Verdier quotient category $\sD^\bctr(\bB)=
\sH^0(\bB)/\sH^0(\bB)_\ac^\bctr$ is called the \emph{contraderived
category of\/~$\bB$} (in the sense of Becker).

\begin{lem} \label{becker-and-lp-contraacyclic}
\textup{(a)} For any short exact sequence\/ $0\rarrow K\rarrow L\rarrow
M\rarrow0$ in the abelian category\/ $\sZ^0(\bB)$, the total object\/
$\Tot(K\to L\to M)\in\bB$ (as defined in
Section~\ref{totalizations-subsecn}) belongs to\/
$\sH^0(\bB)_\ac^\bctr$. \par
\textup{(b)} The full subcategory of contraacyclic objects\/
$\sH^0(\bB)_\ac^\bctr$ is closed under infinite products in\/
$\sH^0(\bB)$.
\end{lem}

\begin{proof}
 This is a common generalization of~\cite[Lemma~7.1]{PS4}
(which is the version for complexes in abelian categories)
and~\cite[Proposition~4.2]{Pctrl} (the CDG\+module version).
 A proof of an even more general assertion can be found
in~\cite[Theorem~5.5(b)]{Pedg}.
\end{proof}

 The following result, generalizing~\cite[Corollary~7.4]{PS4} and
a similar unstated corollary of~\cite[Proposition~1.3.6(1)]{Bec},
will be deduced below in
Section~\ref{contraderived-model-structure-subsecn}.

\begin{cor} \label{becker-contraderived-corollary}
 Let\/ $\bB$ be a locally presentable abelian DG\+category with
enough projective objects.
 Then the composition of the fully faithful inclusion of triangulated
categories\/ $\sH^0(\bB_\bproj)\rarrow\sH^0(\bB)$ with the Verdier
quotient functor\/ $\sH^0(\bB)\rarrow\sD^\bctr(\bB)$ is a triangulated
equivalence\/ $\sH^0(\bB_\bproj)\simeq\sD^\bctr(\bB)$.
\end{cor}

\subsection{Deconstructibility of graded-projectives}
\label{deconstructibility-of-graded-projectives-subsecn}
 Let $\sE$ be a cocomplete abelian category.
 A class of objects $\sL\subset\sE$ is said to be \emph{deconstructible}
if there exists a set of objects $\sS\subset\sL$ such that
$\sL=\Fil(\sS)$ (see Section~\ref{transfinite-extensions-subsecn}
for the notation).
 Another useful property, which does not seem to have a convenient
name in the literature, is existence of a set of objects $\sS\subset
\sL$ such that $\sL=\Fil(\sS)^\oplus$.

\begin{prop} \label{graded-projective-deconstructible}
 Let\/ $\bB$ be a locally presentable abelian DG\+category with
enough projective objects.
 Then there exists a set of graded-projective objects\/
$\sS\subset\sZ^0(\bB_\bproj)$ such that\/ $\sZ^0(\bB_\bproj)=
\Fil(\sS)^\oplus$ in the abelian category\/ $\sZ^0(\bB)$.
\end{prop}

 This important technical assertion is a generalization
of~\cite[Proposition~7.2]{PS4}.
 Similarly to the exposition in~\cite{PS4}, we suggest two proofs.

\begin{proof}[First proof]
 We restrict ourselves to a sketch of this abstract
category-theoretic argument based on~\cite[Proposition~A.1.5.12]{Lur},
\cite[Corollary~3.6]{MR}, and~\cite[Proposition~2.5]{PS4}.

 Let $\bE$ be an additive DG\+category with shifts and cones.
 Then we have a pair of adjoint functors $\Psi^-_\bE\:\sZ^0(\bE^\bec)
\rarrow\sZ^0(\bE)$ and $\Phi_\bE\:\sZ^0(\bE)\rarrow\sZ^0(\bE^\bec)$.
 Consequently, the composition $\Xi_\bE[1]=\Psi^-_\bE\circ\Phi_\bE$,
\ $A\longmapsto\cone(\id_A)$ for all $A\in\bE$, is a monad on
the additive category $\sZ^0(\bE)$.
 The claim is that the functor $\Psi^-_\bE$ is monadic, i.~e.,
the comparison functor from $\sZ^0(\bE^\bec)$ to the category of
algebras over the monad $\Xi_\bE[1]$ on the category $\sZ^0(\bE)$
is a category equivalence.
 This can be established by examining the construction of
the DG\+category~$\bE^\bec$.
 (Dually, the functor $\Xi_\bE\:A\longmapsto\cone(\id_A[-1])$ is
a comonad on $\sZ^0(\bE)$ and the functor $\Psi^+_\bE\:\sZ^0(\bE^\bec)
\rarrow\sZ^0(\bE)$ is comonadic, so this functor is in fact both
monadic and comonadic; cf.~\cite[Section~2]{NST} for a discussion
of a particular case.)

 Notice that the monad $\Xi_\bE[1]\:\sZ^0(\bE)\rarrow\sZ^0(\bE)$ preserves all limits and colimits, as does its left and right
adjoint comonad~$\Xi_\bE$
(see Section~\ref{exactness-properties-of-Xi-and-Psi-subsecn}).

 Now let $\bE$ be an abelian DG\+category.
 Then, according to Corollary~\ref{becbec-equivalence-for-abelian-DG},
we have $\bE\simeq\bE^{\bec\bec}$.
 Taking into account the isomorphisms of functors from the last
paragraph of Section~\ref{DG-functor-becbec-subsecn}, it follows
that the functor $\Phi_\bE\:\sZ^0(\bE)\rarrow\sZ^0(\bE^\bec)$ is
monadic (as well as comonadic), i.~e., the additive category
$\sZ^0(\bE)$ is equivalent to the category of algebras over the monad
$\Xi_{\bE^\bec}[1]=\Phi_\bE\circ\Psi_\bE^+$ on the additive category
$\sZ^0(\bE^\bec)$.
 This sets the stage for our intended application
of~\cite[Corollary~3.6]{MR}.

 The rest is essentially explained in~\cite[first proof of
Proposition~7.2]{PS4}.
 Following~\cite{PS4}, we say that a class of objects $\sL$ in
an abelian category $\sE$ is \emph{transmonic} if any transfinite
composition of $\sL$\+monomorphisms is a monomorphism.
 Then it follows from~\cite[Proposition~A.1.5.12]{Lur} in view
of~\cite[Proposition~2.5]{PS4} that, for any transmonic set of
objects $\sS$ in a locally presentable abelian category $\sE$,
the class of objects $\Fil(\sS)^\oplus$ is deconstructible, i.~e.,
there exists a set of objects $\sS'\subset\sE$ such that
$\Fil(\sS)^\oplus=\Fil(\sS')$.

 Furthermore, for any locally presentable abelian category $\sE$
with a colimit-preserving monad $M\:\sE\rarrow\sE$ and a transmonic
deconstructible class of objects $\sL\subset\sE$, consider
the category $\sF$ of $M$\+algebras in $\sE$ and the forgetful
functor $U\:\sF\rarrow\sE$.
 Then it follows from~\cite[Corollary~3.6]{MR} in view
of~\cite[Proposition~2.5]{PS4} that the class $U^{-1}(\sL)\subset\sF$
is (transmonic and) deconstructible.
 Notice that the category $\sF$ is abelian (with an exact functor~$U$)
by~\cite[Lemma~5.9]{Prel} and $\sF$ is locally presentable
by~\cite[Theorem and Remark~2.78]{AR}.
 It is also easy to see that the functor $U$ preserves all colimits.

 In particular, for any locally presentable abelian category $\sB$
with a projective generator $P$, one has $\sB_\proj=\Fil(\{P\})^\oplus$
(see~\cite[Lemmas~6.1 and~6.2]{PS4}).
 As the class $\Fil(\{P\})$ is transmonic and deconstructible in $\sB$,
it follows that the class of all projective objects $\sB_\proj$ is
deconstructible.
 
 Returning to the situation at hand with a locally presentable abelian
DG\+category $\bB$ with enough projective objects, it remains to apply
the observations above to the abelian categories
$\sB=\sE=\sZ^0(\bB^\bec)$ and $\sF=\sZ^0(\bB)$ and the class of objects
$\sL=\sB_\proj$ in order to conclude that the class $\sZ^0(\bB_\bproj)$
is deconstructible in~$\sZ^0(\bB)$.
 So there exists a set of objects $\sS''\subset\sZ^0(\bB_\bproj)$ such
that $\sZ^0(\bB_\bproj)=\Fil(\sS'')$ in~$\sZ^0(\bB)$.
 As one obviously has $\sZ^0(\bB_\bproj)=\sZ^0(\bB_\bproj)^\oplus$,
we have obtained an even stronger result than claimed in
the proposition.
\end{proof}

\begin{proof}[Second proof]
 Here is a direct proof of the assertion stated in
Proposition~\ref{graded-projective-deconstructible}.
 Notice that the functor $\Phi\:\sZ^0(\bB)\rarrow\sZ^0(\bB^\bec)$
preserves extensions and directed colimits; hence it also preserves
transfinitely iterated extensions (in the sense of
the directed colimit).
 Since the class of all projective objects in $\sZ^0(\bB^\bec)$ is
closed under transfinitely iterated extensions and direct
summands~\cite[Lemma~6.2]{PS4}, it follows that the class
$\sZ^0(\bB_\bproj)$ is also closed under these operations in
$\sZ^0(\bB)$, i.~e., $\Fil(\sZ^0(\bB_\bproj))^\oplus
\subset\sZ^0(\bB_\bproj)$.
 So it suffices to find a set of graded-projective objects
$\sS\subset\sZ^0(\bB_\bproj)$ such that
$\sZ^0(\bB_\bproj)\subset\Fil(\sS)^\oplus$.

 Fix a single projective generator $P$ of the abelian category
$\sZ^0(\bB^\bec)$ (cf.~\cite[Lemma~6.1]{PS4}).
 Replacing if needed $P$ by $\coprod_{n\in\boZ}P[n]$, we can assume
without loss of generality that $P$ is invariant under the shift,
$P\simeq P[1]$.
 According to Section~\ref{Phi-and-Psi-subsecn}, we have
$\Phi\Psi^+(P)\simeq\Xi_{\bB^\bec}(P)[1]$, and by
Corollary~\ref{Xi-extension-cokernel}(a) there is a short exact
sequence $0\rarrow P\rarrow\Xi_{\bB^\bec}(P)\rarrow P[1]\rarrow0$
in $\sZ^0(\bB^\bec)$, which is split because the object $P[1]$
is projective.
 So we have $\Phi\Psi^+(P)\simeq P\oplus P[1]\simeq P\oplus P$.

 Let us show that any graded-projective object $Q'\in\sZ^0(\bB_\bproj)$
is a direct summand of an object $Q\in\sZ^0(\bB_\bproj)$ such that
$\Phi(Q)$ is a coproduct of copies of~$P$.
 Indeed, let $\mu$~be a cardinal such that $\Phi(Q')$ is a direct
summand of~$P^{(\mu)}$, i.~e.\ $\Phi(Q')\oplus P'\simeq P^{(\mu)}$
for some $P'\in\sZ^0(\bB^\bec)$.
 Put $Q=Q'\oplus\Psi^+(P^{(\mu\times\omega)})$.
 Then $\Phi(Q)=\Phi(Q')\oplus \Phi\Psi^+(P^{(\mu\times\omega)})
\simeq\Phi(Q')\oplus P^{(\mu\times2\omega)}\simeq
\Phi(Q')\oplus P^{(\mu\times\omega)}$.
Now $\Phi(Q')\oplus P^{(\mu\times\omega)}\simeq\Phi(Q')\oplus(\Phi(Q')\oplus P')^{(\omega)}\simeq(\Phi(Q')\oplus P')^{(\omega)}\simeq P^{(\mu\times\omega)}$ by the cancellation trick.

 Let $\kappa$~be an uncountable regular cardinal such that the object
$P\in\sZ^0(\bB^\bec)$ is $\kappa$\+presentable (then the category
$\sZ^0(\bB^\bec)$ is locally $\kappa$\+presentable).
 Let $\sS$ be the set of (representatives of isomorphism classes) of
objects $S\in\sZ^0(\bB)$ such that $\Phi(S)$ is a coproduct of less
than~$\kappa$ copies of~$P$.
 Notice that the object $\Phi(S)\in\sZ^0(\bB^\bec)$ is
$\kappa$\+presentable in this case, and by
Proposition~\ref{Phi-Psi-lambda-presentability} it follows that
the object $S\in\sZ^0(\bB)$ is also $\kappa$\+presentable.
 Since there is only a set of $\kappa$\+presentable objects in
a locally presentable category (up to isomorphism)
\cite[Remarks~1.19 and~1.20]{AR}, we can see that $\sS$ is indeed a set.
 Clearly, $\sS\subset\sZ^0(\bB_\bproj)$.
 We claim that any object $Q\in\sZ^0(\bB)$ such that $\Phi(Q)$ is
a coproduct of copies of $P$ belongs to $\Fil(\sS)$.

 Indeed, assume that $\Phi(Q)=P^{(X)}$ for some set~$X$.
 Let $\alpha$~be the successor cardinal of the cardinality of~$X$.
 Proceeding by transfinite induction on ordinals $0\le\beta\le\alpha$,
we will construct a smooth chain of subobjects $R_\beta\subset Q$
and a smooth chain of subsets $Y_\beta\subset X$ such that
$\Phi(R_\beta)=P^{(Y_\beta)}$ as a subobject in $\Phi(Q)=P^{(X)}$
for all $0\le\beta\le\alpha$,\ $Y_0=\varnothing$ (hence $R_0=0$) and
$Y_\alpha=X$ (hence $R_\alpha=Q$), and the cardinality of
$Y_{\beta+1}\setminus Y_\beta$ is smaller than~$\kappa$ for all
$0\le\beta<\alpha$.
 Then it will follow that $\Phi(R_{\beta+1}/R_\beta)\simeq
P^{(Y_{\beta+1})}/P^{(Y_\beta)}\simeq
P^{(Y_{\beta+1}\setminus Y_\beta)}$, since the functor $\Phi$ preserves cokernels; hence $R_{\beta+1}/R_\beta\in\sS$, as desired.

 Suppose that the subsets $Y_\gamma\subset X$ and the subobjects
$R_\gamma\subset Q$ have been already constructed for all
$\gamma<\beta$.
 For a limit ordinal~$\beta$, we put $Y_\beta=
\bigcup_{\gamma<\beta}Y_\gamma$ and $R_\beta=\varinjlim_{\gamma<\beta}
R_\gamma$; then $Y_\beta$ is a subset in $X$ and there is
the induced morphism $R_\beta\rarrow Q$ in~$\sZ^0(\bB)$.
 We have $\Phi(R_\beta)=\varinjlim_{\gamma<\beta}\Phi(R_\gamma)=
\varinjlim_{\gamma<\beta}P^{(Y_\gamma)}=P^{(Y_\beta)}$, since
the functor $\Phi$ preserves colimits.
 So, applying the functor $\Phi$ to the morphism $R_\beta\rarrow Q$,
we obtain the (split) monomorphism $P^{(Y_\beta)}\rarrow P^{(X)}$.
 Since the functor $\Phi$ is faithful and preserves kernels, it
follows that the morphism $R_\beta\rarrow Q$ is a monomorphism.
 This finishes the construction in the case of a limit ordinal~$\beta$.
 To deal with the case of a successor ordinal, we need some
preparatory work.

 For every element $x\in X$, let $\iota_x\:P\rarrow P^{(X)}=\Phi(Q)$
be the direct summand inclusion corresponding to the element~$x$.
 By adjunction, we have the corresponding morphism $\Psi^+(P)\rarrow Q$
in the category~$\sZ^0(\bB)$.
 Applying the functor $\Phi$, we obtain a morphism $\tilde\iota_x\:
\Phi\Psi^+(P)\rarrow\Phi(Q)=P^{(X)}$ in $\sZ^0(\bB^\bec)$.
 We know that $\Phi\Psi^+(P)\simeq P\oplus P$; so the object
$\Phi\Psi^+(P)$ is $\kappa$\+presentable in $\sZ^0(\bB^\bec)$.
 Therefore, there exists a subset $V_x\subset X$ of the cardinality
smaller than~$\kappa$ such that the morphism $\tilde\iota_x$ factorizes
as $\Phi\Psi^+(P)\rarrow P^{(V_x)}\rightarrowtail P^{(X)}$.
 Proceeding by induction in $n\in\omega$, define subsets
$W_x^n\subset X$ by the rules $W_x^0=\{x\}$ and $W_x^{n+1}=
\bigcup_{w\in W_x^n} V_w$.
 Put $W_x=\bigcup_{n\in\omega}W_x^n$.
 Then $W_x$ is a subset in $X$ of the cardinality smaller than~$\kappa$.

 Similarly to the notation in the previous paragraph, given a subset
$W\subset X$, let us denote by $\iota_W\:P^{(W)}\rarrow P^{(X)}=\Phi(Q)$
the related split monomorphism.
 By adjunction, we have the corresponding morphism $\Psi^+(P^{(W)})
\rarrow Q$ in $\sZ^0(\bB)$; applying the functor $\Phi$, we obtain
a morphism $\tilde\iota_W\:\Phi\Psi^+(P^{(W)})\rarrow\Phi(Q)=P^{(X)}$.
 There is also the natural adjunction morphism
$P^{(W)}\rarrow\Phi\Psi^+(P^{(W)})$, whose composition with
the morphism~$\tilde\iota_W$ is the morphism~$\iota_W$.
 Notice that $\Phi\Psi^+(P^{(W)})=(\Phi\Psi^+(P))^{(W)}$.
 For any element $x\in X$, the subset $W_x\subset X$ was constructed
in such a way that the morphism~$\tilde\iota_{W_x}$ factorizes through
the split monomorphism~$\iota_{W_x}$:
\begin{equation} \label{tilde-iota-W-factorizes-through-iota-W}
 P^{(W_x)}\rarrow\Phi\Psi^+(P^{(W_x)})\rarrow P^{(W_x)}
 \rightarrowtail P^{(X)}=\Phi(Q).
\end{equation}

 Denote by $L_x$ the image of the morphism $\Psi^+(P^{(W_x)})
\rarrow Q$ in the abelian category~$\sZ^0(\bB)$.
 Applying the exact functor $\Phi$, we see that the object $\Phi(L_x)$
is the image of the morphism~$\tilde\iota_W$ in the abelian
category~$\sZ^0(\bB^\bec)$:
\begin{equation} \label{image-of-tilde-iota-W}
 \Phi\Psi^+(P^{(W_x)})\twoheadrightarrow\Phi(L_x)
 \rightarrowtail\Phi(Q).
\end{equation}
 Comparing~\eqref{tilde-iota-W-factorizes-through-iota-W}
with~\eqref{image-of-tilde-iota-W}, we conclude that
$\Phi(L_x)$ and $P^{(W_x)}$ is one and the same subobject in
$\Phi(Q)=P^{(X)}$.

 Now we can return to our successor ordinal $\beta=\gamma+1\le\alpha$.
 If $Y_\gamma=X$, then $\Phi(Q/R_\gamma)=0$ and it follows that
$R_\gamma=Q$ (since the functor $\Phi$ is faithful); so we put
$Y_\beta=X$ and $R_\beta=Q$ as well.
 Otherwise, choose an element $x\in X\setminus Y_\gamma$.
 Put $Y_\beta=Y_\gamma\cup W_x\subset X$ and
$R_\beta=R_\gamma+L_x\subset Q$.
\end{proof}

\subsection{Contraderived abelian model structure}
\label{contraderived-model-structure-subsecn}
 The results of this section form a common generalization
of~\cite[Proposition~1.3.6(1)]{Bec} (the CDG\+module case)
and~\cite[Section~7]{PS4} (the case of complexes in abelian categories).
 The references to preceding results in the literature can be
found in~\cite{PS4}.

 Let $\bB$ be an abelian DG\+category.
 Introduce the notation $\sZ^0(\bB)_\ac^\bctr$ for the full
subcategory of Becker-contraacyclic objects in~$\sZ^0(\bB)$.
 So $\sZ^0(\bB)_\ac^\bctr\subset\sZ^0(\bB)$ is the full preimage of
$\sH^0(\bB)_\ac^\bctr\subset\sH^0(\bB)$ under the obvious functor
$\sZ^0(\bB)\rarrow\sH^0(\bB)$.

\begin{thm} \label{contraderived-cotorsion-pair}
 Let\/ $\bB$ be a locally presentable abelian DG\+category with
enough projective objects.
 Then the pair of classes of objects\/ $\sZ^0(\bB_\bproj)$ and\/
$\sZ^0(\bB)_\ac^\bctr$ is a hereditary complete cotorsion pair
in the abelian category\/~$\sZ^0(\bB)$.
\end{thm}

\begin{proof}
 This is a generalization of~\cite[Theorem~7.3]{PS4}.

 Let $\sS\subset\sZ^0(\bB_\bproj)$ be a set of graded-projective objects
in $\bB$ such that $\sZ^0(\bB_\bproj)=\Fil(\sS)^\oplus$ in $\sZ^0(\bB)$,
as in Proposition~\ref{graded-projective-deconstructible}.
 The claim is that the set $\sS\subset\sZ^0(\bB)$ generates
the desired cotorsion pair.

 Indeed, for any objects $Q\in\sZ^0(\bB_\bproj)$ and $B\in\sZ^0(\bB)$
we have $\Ext_{\sZ^0(\bB)}^1(Q,B)\simeq\Hom_{\sH^0(\bB)}(Q,B[1])$
by Lemma~\ref{Ext-in-cocycles-and-Hom-in-homotopy}.
 Hence $\sZ^0(\bB_\bproj)^{\perp_1}=\sZ^0(\bB)_\ac^\bctr
\subset\sZ^0(\bB)$.
 By Lemma~\ref{eklof-lemma}, it follows that
$\sS^{\perp_1}=\sZ^0(\bB)_\ac^\bctr$.

 Furthermore, the class $\sZ^0(\bB_\bproj)=\Fil(\sS)^\oplus$ is
generating in $\sZ^0(\bB)$ (e.~g., because $\sZ^0(\bB)_\proj\subset
\sZ^0(\bB_\bproj)$ by
Lemma~\ref{projectives-are-contractible-graded-projectives}(a)
and there are enough projective objects in the abelian category
$\sZ^0(\bB)$ by assumption).
 Applying Theorem~\ref{eklof-trlifaj}(b), we can conclude that
$\sZ^0(\bB_\bproj)={}^{\perp_1}(\sZ^0(\bB)_\ac^\bctr)$.

 Here is an alternative elementary argument proving the latter equality.
 Notice that in the pair of adjoint functors between abelian categories
$\Phi\:\sZ^0(\bB)\rarrow\sZ^0(\bB^\bec)$ and $\Psi^-\:\sZ^0(\bB^\bec)
\rarrow\sZ^0(\bB)$, both the functors are exact.
 Therefore, for any objects $B\in\sZ^0(\bB)$ and $E\in\sZ^0(\bB^\bec)$
one has $\Ext^n_{\sZ^0(\bB)}(B,\Psi^-(E))\simeq
\Ext^n_{\sZ^0(\bB^\bec)}(\Phi(B),E)$ for all $n\ge0$
(see, e.~g., \cite[Lemma~9.34]{Pedg}).
 In particular, this isomorphism holds for $n=1$.
 Since the object $\Psi^-(E)$ is contractible (by the construction of
the functor~$\Psi^-$), and consequently belongs to
$\sZ^0(\bB)_\ac^\bctr$, it follows that
${}^{\perp_1}(\sZ^0(\bB)_\ac^\bctr)\subset\sZ^0(\bB_\bproj)$.
 The converse inclusion holds since we already know that
$\sZ^0(\bB_\bproj)^{\perp_1}=\sZ^0(\bB)_\ac^\bctr$.

 Any object $B\in\sZ^0(\bB)$ is a subobject of the contractible object
$\Xi(B)[1]=\cone(\id_B)$ by Corollary~\ref{Xi-extension-cokernel}(a),
so the class $\sZ^0(\bB)_\ac^\bctr$ is cogenerating in~$\sZ^0(\bB)$.
 Hence Theorem~\ref{eklof-trlifaj}(a) is applicable, and we can
conclude that our cotorsion pair is complete.
 The cotorsion pair is hereditary, since the class
$\sZ^0(\bB)_\ac^\bctr$ is closed under the cokernels of monomorphisms
in $\sZ^0(\bB)$, as one can see from
Lemma~\ref{becker-and-lp-contraacyclic}(a).
 It is also clear that the class $\sZ^0(\bB_\bproj)\subset\sZ^0(\bB)$
is closed under the kernels of epimorphisms, since the functor $\Phi$
preserves kernels.
\end{proof}

 Now we are ready to prove the corollary formulated in
Section~\ref{becker-contraderived-subsecn}.

\begin{proof}[Proof of Corollary~\ref{becker-contraderived-corollary}]
 It is clear from the definitions that the triangulated functor
$\sZ^0(\bB_\bproj)\rarrow\sD^\bctr(\bB)$ is fully faithful
(cf.\ \cite[Proposition 6.5]{PS4}).
 In order to prove the corollary, it remains to find, for any
object $B\in\sZ^0(\bB)$, a graded-projective object
$Q\in\sZ^0(\bB_\bproj)$ together with a morphism $Q\rarrow B$ in
$\sZ^0(\bB)$ whose cone belongs to $\sZ^0(\bB)_\ac^\bctr$.

 For this purpose, consider a special precover short exact sequence
$0\rarrow X\rarrow Q\rarrow B\rarrow0$ in the complete cotorsion
pair of Theorem~\ref{contraderived-cotorsion-pair}.
 So we have $X\in\sZ^0(\bB)_\ac^\bctr$ and $Q\in\sZ^0(\bB_\bproj)$.
 By Lemma~\ref{becker-and-lp-contraacyclic}(a), the totalization
$\Tot(X\to Q\to C)$ of the short exact sequence $0\rarrow X\rarrow Q
\rarrow C\rarrow0$ in $\sZ^0(\bB)$ is a Becker-contraacyclic object.
 Since the object $X$ is Becker-contraacyclic, it follows that the cone
of the closed morphism $Q\rarrow C$ is Becker-contraacyclic, too.
\end{proof}

\begin{thm} \label{contraderived-model-structure}
 Let\/ $\bB$ be a locally presentable abelian DG\+category with
enough projective objects.
 Then the triple of classes of objects\/ $\sL=\sZ^0(\bB_\bproj)$, \
$\sW=\sZ^0(\bB)_\ac^\bctr$, and\/ $\sR=\sZ^0(\bB)$ is a cofibrantly
generated hereditary abelian model structure on the abelian
category\/~$\sZ^0(\bB)$.
\end{thm}

\begin{proof}
 This is a common generalization of~\cite[Proposition~1.3.6(1)]{Bec}
(which is the CDG\+module case) and~\cite[Theorem~7.5]{PS4} (which
is the case of complexes in abelian categories).
 The pair of classes $(\sL,\sW)$ is a hereditary complete cotorsion
pair in $\sZ^0(\bB)$ by Theorem~\ref{contraderived-cotorsion-pair}.
 According to Lemma~\ref{cotorsion-pairs-with-graded-projectives},
it follows that $\sL\cap\sW=\sZ^0(\bB)_\proj$ is the class of all
projective objects in $\sZ^0(\bB)$ and the class of all
Becker-contraacyclic objects $\sW$ is thick in $\sZ^0(\bB)$
(see also Lemma~\ref{becker-and-lp-contraacyclic}(a)).
 By Lemma~\ref{projective-injective-abelian-model-structures}(a),
the triple $(\sL,\sW,\sR)$ is a projective abelian model structure
on the abelian category~$\sZ^0(\bB)$.

 Finally, it was shown in the proof of
Theorem~\ref{contraderived-cotorsion-pair} that the cotorsion pair
$(\sL,\sW)$ in $\sZ^0(\bB)$ is generated by a set of objects.
 The cotorsion pair $(\sZ^0(\bB)_\proj,\sZ^0(\bB))$ is generated by
the empty set of objects (or by the single zero object, or by any
chosen single projective generator) in~$\sZ^0(\bB)$.
 According to Lemma~\ref{cofibrantly-generated-abelian-model}, this
means that our abelian model structure is cofibrantly generated.
\end{proof}

 The abelian model structure $(\sL,\sW,\sR)$ defined in
Theorem~\ref{contraderived-model-structure} is called
the \emph{contraderived model structure} on the category~$\sZ^0(\bB)$.

\begin{lem} \label{contraderived-weak-equivalences}
 For any locally presentable abelian DG\+category\/ $\bB$ with enough
projective objects, the class\/ $\cW$ of all weak equivalences in
the contraderived model structure on\/ $\sZ^0(\bB)$ coincides with
the class of all closed morphisms of degree\/~$0$ in\/ $\bB$ with
the cones belonging to\/ $\sZ^0(\bB)_\ac^\bctr$.
\end{lem}

\begin{proof}
 \cite[Lemma~5.8]{Hov} tells that a monomorphism in an abelian model
category is a weak equivalence if and only if its cokernel is weakly
trivial.
 Dually, an epimorphism is a weak equivalence if and only if its
kernel is weakly trivial.

 In the contraderived model structure on $\sZ^0(\bB)$, the class of
weakly trivial objects $\sW$ is the class of all Becker-contraacyclic
objects.
 Any morphism~$f$ in $\sZ^0(\bB)$ can be factorized as $f=rl$,
where $l$~is, say, a trivial cofibration, and $r$~is a fibration.
 Now $l$~is a monomorphism with a Becker-contraacyclic cokernel,
hence in view of Lemma~\ref{becker-and-lp-contraacyclic}(a) also
with a Becker-contraacyclic cone; at the same time~$l$ is a weak
equivalence.
 On the other hand, $r$~is an epimorphism.
 
 Notice that both the class of weak equivalences and the class of
morphisms with a Becker-contraacyclic cone satisfy the 2\+out-of\+3
property.
 If $f$~is a weak equivalence, then so is~$r$;
then the kernel of~$r$ is Becker-contraacyclic, hence by
Lemma~\ref{becker-and-lp-contraacyclic}(a) the cone of~$r$ is
Becker-contraacyclic as well; it follows that the cone of~$f$
is Becker-contraacyclic.
 If $f$~has a Becker-contraacyclic cone, then so does~$r$;
hence by Lemma~\ref{becker-and-lp-contraacyclic}(a) the kernel of~$r$
is Becker-contraacyclic; thus $r$~is a weak equivalence, and
it follows that~$f$ is a weak equivalence.
 Alternatively, one could use the fact that any morphism~$f$ is
the composition of a cofibration~$l$ and a trivial fibration~$r$
and analogous reasoning in that case.
\end{proof}

 The following corollary presumes existence of infinite coproducts
in Becker's contraderived category $\sD^\bctr(\bB)$.
 Such coproducts can be simply constructed as the coproducts in
the homotopy category of graded-projective objects $\sH^0(\bB_\bproj)$,
which is equivalent to $\sB^\bctr(\bB)$ by
Corollary~\ref{becker-contraderived-corollary}.
 Notice that coproducts of graded-projective objects in $\bB$ are
graded-projective, since the coproducts of projectives in
$\sZ^0(\bB^\bec)$ are projective and the functor $\Phi$ preserves
coproducts.

\begin{cor} \label{contraderived-well-generated}
 For any locally presentable abelian DG\+category\/ $\bB$ with
enough projective objects, Becker's contraderived category\/
$\sD^\bctr(\bB)$ is a well-generated triangulated category.
\end{cor}

\begin{proof}
 This is a generalization of~\cite[Corollary~7.7]{PS4}.
 Let us show that the contraderived category $\sD^\bctr(\bB)$ can be
equivalently defined as the homotopy category $\sZ^0(\bB)[\cW^{-1}]$
of the contraderived model structure on $\sZ^0(\bB)$.

 The key observation, generalizing the proof in~\cite{PS4}, is that,
for any DG\+category $\bE$ with shifts and cones, inverting all
the homotopy equivalences in $\sZ^0(\bE)$ produces the homotopy
category $\sH^0(\bE)$.
 In other words, homotopic closed morphisms of degree~$0$ become equal
after inverting the homotopy equivalences.
 Indeed, for any pair of homotopic closed morphisms of degree zero
$f'$, $f''\:A\rarrow B$ in $\bE$, there exist a morphism $h\:A\oplus
\cone(\id_A)\rarrow B$ and two morphisms $\iota'$, $\iota''\:A
\rarrow A\oplus\cone(\id_A)$ in $\sZ^0(\bE)$ such that $f'=h\iota'$
and $f''=h\iota''$, while $\pi\iota'=\id_A=\pi\iota''$, where
$\pi\:A\oplus\cone(\id_A)\rarrow A$ is the direct summand projection.
 The morphism~$\pi$ is a homotopy equivalence in $\bE$, since
the object $\cone(\id_A)$ is contractible.
 Inverting~$\pi$ identifies $\iota'$ with~$\iota''$ and consequently
$f'$ with~$f''$.

 In the situation at hand, it follows that inverting all the morphisms
with Becker-contraacyclic cones in $\sZ^0(\bB)$ produces Becker's
contraderived category $\sD^\bctr(\bB)$.
 It remains to use Lemma~\ref{contraderived-weak-equivalences} in
order to conclude that $\sD^\bctr(\bB)=\sZ^0(\bB)[\cW^{-1}]$.

 By Theorem~\ref{contraderived-model-structure}, the contraderived
model structure on $\sZ^0(\bB)$ is hereditary abelian and cofibrantly
generated.
 By Lemma~\ref{hereditary-is-stable}, this model structure is also
stable.
 So Proposition~\ref{stable-combinatorial-well-generated} can be
applied, finishing the proof. 
\end{proof}

\Section{Coderived Model Structure} \label{coderived-secn}

 In this section we work out a common generalization of the coderived
model structure on the abelian category of CDG\+modules over
a CDG\+ring~\cite[Proposition~1.3.6(2)]{Bec} and the coderived model
structure on the category of complexes over a Grothendieck abelian
category~\cite[Section~4.1]{Gil}, \cite[Section~9]{PS4}.
 The Grothendieck abelian DG\+categories~$\bA$ (as defined
in~\cite[Section~9.1]{Pedg}) form a suitable context.

 The main results of the section, however, are
Corollary~\ref{coacyclics-as-closure-extensions-directed-colimits},
claiming that the class of all Becker-coacyclic objects is closed
under directed colimits, and
Theorem~\ref{generators-of-coderived-category}, providing a sufficient
condition for a set of objects to generate Becker's coderived category
$\sD^\bco(\bA)$ as a triangulated category with coproducts.

\subsection{Injective objects in abelian DG-categories}
 In this section we briefly list the dual assertions to the ones
in Section~\ref{projectives-in-abelian-DG-categories-subsecn} and
fix the related terminology/notation.

\begin{lem} \label{abelian-DG-enough-injectives}
 Let\/ $\bA$ be an abelian DG\+category.
 Then the abelian category\/ $\sZ^0(\bA)$ has enough injective objects
if and only if the abelian category\/ $\sZ^0(\bA^\bec)$ has enough
injective objects.
\end{lem}

\begin{proof}
 This is the dual version of
Lemma~\ref{abelian-DG-enough-projectives}.
\end{proof}

 An abelian DG\+category is said to \emph{have enough injective objects}
if it satisfies the equivalent conditions of
Lemma~\ref{abelian-DG-enough-injectives}.

\begin{lem} \label{injectives-are-contractible-graded-injectives}
 Let\/ $\bA$ be an abelian DG\+category with enough injective objects.
 Then \par
\textup{(a)} an object $J\in\sZ^0(\bA)$ is injective if and only
if $J$ is contractible in\/ $\bA$ \emph{and} the object\/
$\Phi(J)\in\sZ^0(\bA^\bec)$ is injective; \par
\textup{(b)} an object $I\in\sZ^0(\bA^\bec)$ is injective if and
only if $I$ is contractible in\/ $\bA^\bec$ \emph{and} the object\/
$\Psi^-(I)\in\sZ^0(\bA)$ is injective.
\end{lem}

\begin{proof}
 This is the dual version of
Lemma~\ref{projectives-are-contractible-graded-projectives}.
\end{proof}

 Let $\bA$ be an abelian DG\+category.
 We will say that an object $J\in\bA$ is \emph{graded-injective}
if the object $\Phi(J)$ is injective in the abelian category
$\sZ^0(\bA^\bec)$.
 The full DG\+subcategory formed by the graded-injective objects
in $\bA$ is denoted by $\bA_\binj\subset\bA$.
 So the notation $\sZ^0(\bA_\binj)\subset\sZ^0(\bA)$ stands for
the full subcategory of all graded-injective objects in $\sZ^0(\bA)$,
while $\sZ^0(\bA)_\inj\subset\sZ^0(\bA)$ is the (smaller) full
subcategory of all injective objects in $\sZ^0(\bA)$.
 Similarly to the case of graded-projectives discussed in
Section~\ref{projectives-in-abelian-DG-categories-subsecn},
the full DG\+subcategory of graded-injective objects $\bA_\binj$ is
closed under finite direct sums, shifts, twists, and cones in~$\bA$.

\begin{lem} \label{cotorsion-pairs-with-graded-injectives}
 Let\/ $\bA$ be an abelian DG\+category with enough injective objects,
and let $(\sW,\sR)$ be a cotorsion pair in the abelian category\/
$\sZ^0(\bA)$ such that\/ $\sR\subset\sZ^0(\bA_\binj)$.
 Assume that the cotorsion pair $(\sW,\sR)$ is preserved by the shift:\/
$\sW=\sW[1]$, or equivalently, $\sR=\sR[1]$.
 Then\/ $\sR\cap\sW=\sZ^0(\bA)_\inj$.
 If, moreover, the cotorsion pair $(\sW,\sR)$ is complete, then it is
hereditary and the class\/ $\sW$ is thick in\/ $\sZ^0(\bA)$.
\end{lem}

\begin{proof}
 This is the dual version of
Lemma~\ref{cotorsion-pairs-with-graded-projectives}.
\end{proof}

\subsection{Grothendieck abelian DG-categories}
\label{Grothendieck-DG-categories-subsecn}
 This section is mostly an extraction from~\cite[Section~9.1]{Pedg}.

\begin{lem} \label{abelian-DG-exact-directed-colimits}
 Let\/ $\bA$ be an abelian DG\+category with infinite coproducts.
 Then the coproduct functors are exact in the abelian category\/
$\sZ^0(\bA)$ if and only if they are exact in the abelian category\/
$\sZ^0(\bA^\bec)$.
 Moreover, the directed colimits are exact in\/ $\sZ^0(\bA)$ if and
only if they are exact in\/ $\sZ^0(\bA^\bec)$.
\end{lem}

\begin{proof}
 This is~\cite[Lemma~9.3]{Pedg}.
 The point is that both the functors $\Phi\:\sZ^0(\bA)\rarrow
\sZ^0(\bA^\bec)$ and $\Psi^+\:\sZ^0(\bA^\bec)\rarrow\sZ^0(\bA)$
are exact and faithful, and preserve directed colimits.
 So any instance of nonexactness of a directed coproduct/colimit
in one of the two abelian categories would be taken by
the respective functor to an instance of nonexactness of a similar
coproduct/colimit in the other category.
\end{proof}

 A \emph{Grothendieck category} is an abelian category with
a (single) generator, infinite coproducts, and exact directed colimits.

\begin{prop} \label{Grothendieck-DG-categories}
 Let $\bA$ be an abelian DG\+category with infinite coproducts.
 Then the abelian category\/ $\sZ^0(\bA)$ is Grothendieck if and
only if the abelian category\/ $\sZ^0(\bA^\bec)$ is Grothendieck.
\end{prop}

\begin{proof}
 This is~\cite[Proposition~9.4]{Pedg}.
 The assertion follows from Lemmas~\ref{Phi-Psi-strong-generators}
and~\ref{abelian-DG-exact-directed-colimits} above.
\end{proof}

 An abelian DG\+category $\bA$ with infinite coproducts is said to be
\emph{Grothendieck} if both the abelian categories $\sZ^0(\bA)$ and
$\sZ^0(\bA^\bec)$ are Grothendieck.
 In other words, this means that $\bA$ satisfies assumptions and
the equivalent conditions of
Proposition~\ref{Grothendieck-DG-categories}.

 Notice that any locally finitely presentable abelian category is
Grothendieck~\cite[Proposition~1.59]{AR}.
 Consequently, any locally finitely presentable abelian DG\+category
is Grothendieck.
 On the other hand, all Grothendieck abelian categories are locally
presentable~\cite[Corollary~5.2]{Kra3}.
 Therefore, any Grothendieck abelian DG\+category is locally
presentable.
 Furthermore, it is well-known that all Grothendieck abelian categories 
have enough injective objects.
 Consequently, all Grothendieck abelian DG\+categories have enough
injective objects.

\begin{ex} \label{Grothendieck-abelian-DG-of-complexes}
 Let $\sA$ be a Grothendieck abelian category.
 Then the DG\+category $\bC(\sA)$ of complexes in $\sA$ is abelian,
as explained in Example~\ref{abelian-DG-category-of-complexes-example}.
 It is clear that the abelian category of graded objects
$\sG(\sA)=\sA^\boZ$ is Grothendieck, and it is also easy to see
that the abelian category of complexes $\sC(\sA)=\sZ^0(\bC(\sA))$ is
Grothendieck.
 According to Section~\ref{almost-involution-complexes-subsecn},
the abelian category $\sZ^0(\bC(\sA)^\bec)$ is equivalent to $\sG(\sA)$.
 Hence $\bC(\sA)$ is a Grothendieck abelian DG\+category.
\end{ex}

\begin{ex} \label{Grothendieck-abelian-DG-of-CDG-modules}
 Let $\biR^\cu=(R^*,d,h)$ be a CDG\+ring.
 Then, according to
Example~\ref{locally-presentable-abelian-DG-of-CDG-modules},
the DG\+category $\biR^\cu\bModl$ of left CDG\+modules over $\biR^\cu$
is a locally finitely presentable abelian DG\+category.
 Consequently, $\biR^\cu\bModl$ is also a Grothendieck DG\+category.
\end{ex}

\begin{ex} \label{Grothendieck-abelian-DG-of-factorizations}
 Let $\sA$ be a Grothendieck abelian category, $\Delta\:\sA\rarrow
\sA$ be an autoequivalence, and $v\:\Id_\sA\rarrow\sA$ be a potential.
 Then the DG\+category $\bF(\sA,\Delta,v)$ of factorizations of~$v$
in $\sA$ is abelian, as per
Example~\ref{abelian-DG-category-of-factorizations-example}.
 It is clear that the abelian category of $2$\+$\Delta$-periodic
objects $\sP(\sA,\Delta)\simeq\sA\times\sA$ is Grothendieck and
the abelian category of factorizations $\sF(\sA,\Delta,v)$ has
infinite coproducts.
 According to Section~\ref{almost-involution-factorizations-subsecn},
the abelian category $\sZ^0(\bF(\sA,\Delta,v)^\bec)$ is equivalent to
$\sP(\sA,\Delta)$.
 Hence $\bF(\sA,\Delta,v)$ is a Grothendieck abelian DG\+category.
\end{ex}

\subsection{Becker's coderived category}
\label{becker-coderived-subsecn}
 In the rest of Section~\ref{coderived-secn}, we will be mostly working
with a Grothendieck abelian DG\+category~$\bA$.

 In this paper we are interested in the coderived category \emph{in
the sense of Becker}~\cite{Kra,Bec,Sto2,PS4}, which we denote by
$\sD^\bco(\bA)$.
 In well-behaved cases, it is equivalent to the homotopy category of
graded-injective objects, $\sD^\bco(\bA)\simeq\sH^0(\bA_\binj)$.

 The coderived category in the sense of Becker has to be distinguished
from the coderived category in the sense of books and
papers~\cite{Psemi,Pkoszul,EP,Pps,Prel}.
 The two definitions of a coderived category are known to be equivalent
under certain assumptions~\cite[Section~3.7]{Pkoszul},
\cite[Corollary~4.18]{Pctrl}, \cite[Theorem~5.10(a)]{Pedg}, but it is
an open question whether they are equivalent for the category of modules
over an arbitrary ring (see~\cite[Example~2.5(3)]{Pps},
\cite[Remark~9.2]{PS4}, and~\cite[Section~7.9]{Pksurv} for a discussion).

 Let $\bA$ be an abelian DG\+category.
 An object $Y\in\bA$ is said to be \emph{coacyclic} (\emph{in the sense
of Becker}) if $\Hom_{\sH^0(\bA)}(Y,J)=0$ for all graded-injective
objects $J\in\bA_\binj$.
 We will denote the full subcategory of Becker-coacyclic objects by
$\sH^0(\bA)_\ac^\bco\subset\sH^0(\bA)$ and its preimage under
the obvious functor $\sZ^0(\bA)\rarrow\sH^0(\bA)$ by
$\sZ^0(\bA)_\ac^\bco\subset\sZ^0(\bA)$.
 Clearly, $\sH^0(\bA)_\ac^\bco$ is a triangulated (and even thick)
subcategory in the homotopy category $\sH^0(\bA)$.
 The triangulated Verdier quotient category $\sD^\bco(\bA)=
\sH^0(\bA)/\sH^0(\bA)_\ac^\bco$ is called the \emph{coderived category
of\/~$\bA$} (in the sense of Becker).

\begin{lem} \label{becker-and-lp-coacyclic}
 The class of all Becker-coacyclic objects\/ $\sZ^0(\bA)_\ac^\bco$
is closed under transfinitely iterated extensions (in the sense of
the directed colimit) in the abelian category\/~$\sZ^0(\bA)$.
 In particular, \par
\textup{(a)} for any short exact sequence\/ $0\rarrow K\rarrow L\rarrow
M\rarrow0$ in the abelian category\/ $\sZ^0(\bA)$, the total object\/
$\Tot(K\to L\to M)\in\bA$ (as defined in
Section~\ref{totalizations-subsecn}) belongs to\/ $\sH^0(\bA)_\ac^\bco$.
\par
\textup{(b)} The full subcategory of coacyclic objects\/
$\sH^0(\bA)_\ac^\bco$ is closed under infinite coproducts in\/
$\sH^0(\bA)$.
\end{lem}

\begin{proof}
 For any object $Y\in\bA$ and any graded-injective object
$J\in\bA_\binj$, there is an isomorphism of abelian groups
$\Hom_{\sH^0(\bA)}(Y,J)\simeq\Ext^1_{\sZ^0(\bA)}(Y,J[-1])$ provided
by Lemma~\ref{Ext-in-cocycles-and-Hom-in-homotopy}.
 Therefore, the full subcategory $\sZ^0(\bA)_\ac^\bco\subset
\sZ^0(\bA)$ consists precisely of all the objects $Y\in\bA$ such
that $\Ext^1_{\sZ^0(\bA)}(Y,J)=0$ for all $J\in\bA_\binj$.
 Now Lemma~\ref{eklof-lemma} implies that the class
$\sZ^0(\bA)_\ac^\bco$ is closed under transfinitely iterated extensions
in~$\sZ^0(\bA)$.

 Both parts~(a) and~(b) are particular cases of the assertion about
transfinitely iterated extensions (assuming that coproducts exist
in~$\bA$).
 Indeed, the object $\Tot(K\to L\to M)$ fits into a short exact
sequence
\[ 0 \rarrow \cone(\id_K) \rarrow \Tot(K\to L\to M) \rarrow \cone(\id_M)[-1] \rarrow 0 \]
in $\sZ^0(\bA)$ and (the suitable shifts of) the cones of identity
endomorphisms of the objects $K$ and $M$ are contractible and
therefore coacyclic.
 If coproducts exist in $\bA$, then the coproducts in $\sH^0(\bA)$
agree with those in $\sZ^0(\bA)$, and the latter can be interpreted
as transfinitely iterated extensions.
 Alternatively, both parts~(a) and~(b) are the dual versions of
the respective parts of Lemma~\ref{becker-and-lp-contraacyclic}.
\end{proof}

 The following result, generalizing~\cite[Corollary~9.5]{PS4} and
a similar unstated corollary of~\cite[Proposition~1.3.6(2)]{Bec},
will be deduced below in
Section~\ref{coderived-model-structure-subsecn}.
 Other relevant references include~\cite[Theorem~2.13]{Neem2},
\cite[Corollary~5.13]{Kra3}, and~\cite[Theorem~4.2]{Gil}.

\begin{cor} \label{becker-coderived-corollary}
 Let\/ $\bA$ be a Grothendieck abelian DG\+category.
 Then the composition of the fully faithful inclusion of triangulated
categories\/ $\sH^0(\bA_\binj)\rarrow\sH^0(\bA)$ with the Verdier
quotient functor\/ $\sH^0(\bA)\rarrow\sD^\bco(\bA)$ is a triangulated
equivalence\/ $\sH^0(\bA_\binj)\simeq\sD^\bco(\bA)$.
\end{cor}

\subsection{Coderived abelian model structure}
\label{coderived-model-structure-subsecn}
 The results of this section form a common generalization
of~\cite[Proposition~1.3.6(2)]{Bec} (the CDG\+module case)
and~\cite[Section~9]{PS4} (the case of complexes in abelian categories).
 The references to preceding results in the literature can be
found two paragraphs above or in~\cite{PS4}.

\begin{thm} \label{coderived-cotorsion-pair}
 Let\/ $\bA$ be a Grothendieck abelian DG\+category.
 Then the pair of classes of objects\/ $\sZ^0(\bA)_\ac^\bco$ and\/
$\sZ^0(\bA_\binj)$ is a hereditary complete cotorsion pair in
the abelian category\/~$\sZ^0(\bA)$.
\end{thm}

\begin{proof}
 This is a generalization of~\cite[Theorem~9.3]{PS4}.

 In any Grothendieck abelian category $\sA$, there exists a set
of objects $\sS_0$ such that $\sA=\Fil(\sS_0)$
\,\cite[Example~1.2.6(1)]{Bec}, \cite[Lemma~8.2]{PS4}.
 Let $\sS_0$ be such a set of objects in the Grothendieck abelian
category $\sA=\sZ^0(\bA^\bec)$.
 The claim is that the desired cotorsion pair is generated by the set
of objects $\sS=\{\Psi^+(S)\mid S\in\sS_0\}\subset\sZ^0(\bA)$.

 Indeed, for any objects $E\in\sZ^0(\bA^\bec)$ and $A\in\sZ^0(\bA)$
we have $\Ext^1_{\sZ^0(\bA)}(\Psi^+(E),A)\simeq
\Ext^1_{\sZ^0(\bA^\bec)}(E,\Phi(A))$ by~\cite[Lemma~9.34]{Pedg},
because both functors in the adjoint pair $\Psi^+\:\sZ^0(\bA^\bec)
\rarrow\sZ^0(\bA)$ and $\Phi\:\sZ^0(\bA)\rarrow\sZ^0(\bA^\bec)$ are
exact (cf.\ the proof of Theorem~\ref{contraderived-cotorsion-pair}).
 The functor $\Psi^+$ is exact and preserves colimits, so it also
preserves transfinitely iterated extensions.
 In view of Lemma~\ref{eklof-lemma}, it follows that
$\sS^{\perp_1}=\Psi^+(\sZ^0(\bA^\bec))^{\perp_1}=
\sZ^0(\bA_\binj)\subset\sZ^0(\bA)$.

 Furthermore, $\sZ^0(\bA)_\ac^\bco={}^{\perp_1}\sZ^0(\bA_\binj)\subset
\sZ^0(\bA)$, as we have already seen in the proof of
Lemma~\ref{becker-and-lp-coacyclic}.
 Hence our pair of classes of objects is indeed the cotorsion pair
generated by the set $\sS$ in~$\sZ^0(\bA)$.

 Any object $A\in\sZ^0(\bA)$ is a quotient object of the contractible
object $\Xi(A)=\cone(\id_A[-1])$ by
Corollary~\ref{Xi-extension-cokernel}.
 All contractible objects are Becker-coacyclic; so the class
$\sZ^0(\bA)_\ac^\bco$ is generating in $\sZ^0(\bA)$.
 Any object in $\sZ^0(\bA)$ is also a subobject of an injective
object, and all injective objects are graded-injective by
Lemma~\ref{injectives-are-contractible-graded-injectives}(a);
hence the class $\sZ^0(\bA_\binj)$ is cogenerating in~$\sZ^0(\bA)$.
 Applying Theorem~\ref{eklof-trlifaj}(a), we conclude that our
cotorsion pair is complete.

 The cotorsion pair is hereditary, because the class
$\sZ^0(\bA)_\ac^\bco$ is closed under the kernels of epimorphisms
in $\sZ^0(\bA)$, as one can see from
Lemma~\ref{becker-and-lp-coacyclic}(a).
 It is also clear that the class $\sZ^0(\bA_\binj)\subset\sZ^0(\bA)$ is
closed under the cokernels of monomorphisms, since the functor $\Phi$
preserves cokernels.
\end{proof}

 The proof of Theorem~\ref{coderived-cotorsion-pair} implies
the following description of the class of all Becker-coacyclic objects.

\begin{cor} \label{coacyclics-as-filtered-by-contractibles}
 Let\/ $\bA$ be a Grothendieck abelian DG\+category.
 Then the Becker-coacyclic objects in\/ $\sZ^0(\bA)$ are precisely
the direct summands of objects filtered by contractible ones.
 In fact, if\/ $\sS_0\subset\sZ^0(\bA^\bec)$ is a set of objects
such that\/ $\sZ^0(\bA^\bec)=\Fil(\sS_0)$ and\/
$\sS=\{\Psi^+(S)\mid S\in\sS_0\}$ as above, then\/
$\sZ^0(\bA)_\ac^\bco=\Fil(\sS)^\oplus$.
\end{cor}

\begin{proof}
 This is a generalization of~\cite[Corollary~9.4]{PS4}.

 Let $E\in\sZ^0(\bA^\bec)$ be an arbitrary object.
 Since $E\in\Fil(\sS_0)$ and the functor $\Psi^+$ preserves
transfinitely iterated extensions, we have $\Psi^+(E)\in\Fil(\sS)$.
 As any object $A\in\sA$ is a quotient object of the object
$\Xi(A)\simeq\Psi^+\Phi(A)$, we see that the class $\Fil(\sS)$
is generating in~$\sZ^0(\bA)$.
 Applying Theorem~\ref{eklof-trlifaj}(b), we can conclude that
$\sZ^0(\bA)_\ac^\bco=\Fil(\sS)^\oplus$.
 Since every object in $\sS$ is contractible, any contractible
object is Becker-coacyclic, and the class of all Becker-coacyclic
objects is closed under transfinitely iterated extensions, it also
follows that $\sZ^0(\bA)_\ac^\bco$ is the class of all direct summands
of transfinitely iterated extensions of contractible objects.
\end{proof}

\begin{proof}[Proof of Corollary~\ref{becker-coderived-corollary}]
 It is clear from the definitions that the functor $\sH^0(\bA_\binj)
\rarrow\sD^\bco(\bA)$ is fully faithful
(cf.\ \cite[Corollary 9.5]{PS4}).
 In order to prove the corollary, it remains to find, for any object
$A\in\sZ^0(\bA)$, a graded-injective object $J\in\sZ^0(\bA_\binj)$
together with a morphism $A\rarrow J$ in $\sZ^0(\bA)$ whose cone
belongs to $\sZ^0(\bA)_\ac^\bco$.

 For this purpose, consider a special preenvelope short exact sequence
$0\rarrow A\rarrow J\rarrow Y\rarrow0$ in the complete cotorsion pair
of Theorem~\ref{coderived-cotorsion-pair}.
 So we have $J\in\sZ^0(\bA_\binj)$ and $Y\in\sZ^0(\bA)_\ac^\bco$.
 By Lemma~\ref{becker-and-lp-coacyclic}(a), the totalization
$\Tot(A\to J\to Y)$ of the short exact sequence $0\rarrow A\rarrow
J\rarrow Y\rarrow0$ is a Becker-coacyclic object.
 Since the object $Y$ is Becker-coacyclic, it follows that the cone of
the morphism $A\rarrow J$ is Becker-coacyclic, too.
\end{proof}

\begin{thm} \label{coderived-model-structure}
 Let\/ $\bA$ be a Grothendieck abelian DG\+category.
 Then the triple of classes of objects\/ $\sL=\sZ^0(\bA)$, \
$\sW=\sZ^0(\bA)_\ac^\bco$, and\/ $\sR=\sZ^0(\bA_\binj)$ is
a cofibrantly generated hereditary abelian model structure on
the abelian category\/~$\sZ^0(\bA)$.
\end{thm}

\begin{proof}
 This is a common generalization of~\cite[Proposition~1.3.6(2)]{Bec}
(which is the CDG\+module case) and~\cite[Theorem~4.2]{Gil}
or~\cite[Theorem~9.6]{PS4} (which is the case of complexes in
abelian categories).

 The pair of classes $(\sW,\sR)$ is a complete cotorsion pair in
$\sZ^0(\bA)$ by Theorem~\ref{coderived-cotorsion-pair}.
 According to Lemma~\ref{cotorsion-pairs-with-graded-injectives},
it follows that $\sR\cap\sW=\sZ^0(\bA)_\inj$.
 It follows from Lemma~\ref{becker-and-lp-coacyclic}(a) that the class
of Becker-coacyclic complexes $\sW$ is thick in $\sZ^0(\bA)$
(see also Lemma~\ref{W-is-thick-lemma}(b)
or~\ref{cotorsion-pairs-with-graded-injectives}).
 By Lemma~\ref{projective-injective-abelian-model-structures}(b),
the triple $(\sL,\sW,\sR)$ is an injective abelian model structure on
the category~$\sZ^0(\bA)$.

 Finally, it was shown in the proof of
Theorem~\ref{coderived-cotorsion-pair} that the cotorsion pair
$(\sW,\sR)$ in $\sZ^0(\bA)$ is generated by a set of objects.
 The cotorsion pair $(\sZ^0(\bA),\sZ^0(\bA)_\inj)$ is generated
by any set of objects $\sS_0\subset\sZ^0(\bA)$ such that
$\sZ^0(\bA)=\Fil(\sS_0)$, as in~\cite[Example~1.2.6(1)]{Bec} or
\cite[Lemma~8.2]{PS4} applied to the category~$\sZ^0(\bA)$.
 By Lemma~\ref{cofibrantly-generated-abelian-model}, this means that
our abelian model structure is cofibrantly generated.
\end{proof}

 The abelian model structure $(\sL,\sW,\sR)$ defined in
Theorem~\ref{coderived-model-structure} is called the \emph{coderived
model structure} on the category~$\sZ^0(\bA)$.

\begin{lem} \label{coderived-weak-equivalences}
 For any Grothendieck abelian DG\+category\/ $\bA$, the class\/ $\cW$ of
all weak equivalences in the coderived model structure on\/ $\sZ^0(\bA)$
coincides with the class of all closed morphisms of degree\/~$0$
in\/ $\bA$ with the cones belonging to\/ $\sZ^0(\bA)_\ac^\bco$.
\end{lem}

\begin{proof}
 Similar to Lemma~\ref{contraderived-weak-equivalences}.
\end{proof}

 The next corollary presumes existence of infinite coproducts in
Becker's coderived category $\sD^\bco(\bA)$.
 Here we point out that, since the thick subcategory of Becker-coacyclic
complexes $\sH^0(\bA)_\ac^\bco\subset\sH^0(\bA)$ is closed under
infinite coproducts by Lemma~\ref{becker-and-lp-coacyclic}(b),
the coproducts in Becker's coderived category $\sD^\bco(\bA)$ are
induced by those in the homotopy category $\sH^0(\bA)$.
 In other words, the Verdier quotient functor $\sH^0(\bA)\rarrow
\sD^\bco(\bA)$ preserves coproducts~\cite[Lemma~3.2.10]{Neem1}.

\begin{cor} \label{coderived-well-generated}
 For any Grothendieck abelian DG\+category\/ $\bA$, Becker's coderived
category\/ $\sD^\bco(\bA)$ is a well-generated triangulated category.
\end{cor}

\begin{proof}
 This is a generalization of the well-known result for the categories
of complexes in Grothendieck abelian categories (which can be found
in~\cite[Theorem~3.13]{Neem2}, \cite[Theorem~5.12]{Kra3},
\cite[Theorem~4.2]{Gil}, or~\cite[Corollary~9.8]{PS4}).

 The proof is similar to that of
Corollary~\ref{contraderived-well-generated}.
 One observes that Becker's coderived category $\sD^\bco(\bA)$ can be
equivalently defined as the homotopy category $\sZ^0(\bA)[\cW^{-1}]$
of the coderived model structure on $\sZ^0(\bA)$, and uses
Theorem~\ref{coderived-model-structure} together with
Lemma~\ref{hereditary-is-stable} and
Proposition~\ref{stable-combinatorial-well-generated}. 
\end{proof}

\subsection{Directed colimits of Becker-coacyclic objects}
\label{directed-colimits-of-becker-coacyclics}
 Varions descriptions of the class of all Becker-coacyclic objects
play a key role in the rest of this paper.
 One such description was already obtained in
Corollary~\ref{coacyclics-as-filtered-by-contractibles}.
 The aim of this section is to provide another one.
 Under more restrictive assumptions, yet another description of
the class $\sZ^0(\bA)_\ac^\bco$ will be given in
Section~\ref{absolute-derived-subsecn}.

\begin{prop} \label{directed-colimits-from-transfinite-extensions}
 Let\/ $\sA$ be an abelian category with coproducts and exact functors
of directed colimit, and let\/ $\sC\subset\sA$ be a class of objects.
 Then the following conditions are equivalent:
\begin{enumerate}
\item the class\/ $\sC$ is closed under the cokernels of monomorphisms and transfinitely iterated extensions (in the sense of the directed
colimit) in\/~$\sA$;
\item the class\/ $\sC$ is closed under the cokernels of monomorphisms,
extensions, and directed colimits of smooth well-ordered chains of
monomorphisms in\/~$\sA$;
\item the class\/ $\sC$ is closed under the cokernels of monomorphisms,
extensions, and directed colimits in\/~$\sA$.
\end{enumerate}
\end{prop}

\begin{proof}
 (3)\,$\Longrightarrow$\,(2) Obvious.
  
 (2)\,$\Longrightarrow$\,(1) It suffices to observe that
a transfinitely iterated extension in a category with exact functors
of directed colimit is by the definition built from
extensions and directed colimits of smooth chains of monomorphisms.

 (1)\,$\Longrightarrow$\,(2)
 Let $(C_i\to C_j)_{i<j\le\alpha}$ be a smooth well-ordered chain of
objects in $\sA$, indexed by a limit ordinal~$\alpha$, such that
$C_i\in\sC$ for all $i<\alpha$ and the morphism $C_i\rarrow C_{i+1}$ is
a monomorphism in $\sA$ for all $i<\alpha$.
 Without loss of generality we can assume that $C_0=0$.
 By assumption, the cokernel $S_i$ of the morphism $C_i\rarrow C_{i+1}$
belongs to $\sC$ for all $i<\alpha$.
 The object $C=C_\alpha$ is a transfinitely iterated extension of
the objects $S_i$, \,$i<\alpha$; hence it also belongs to $\sC$,
as desired.

 (2)\,$\Longrightarrow$\,(3)
 It is well-known that a class of objects in a cocomplete category
is closed under directed colimits if and only if it is closed under
directed colimits of smooth well-ordered chains~\cite[Lemma~1.6,
Corollary~1.7, and Remark~1.7]{AR}.

 Let $(f_{ji}\:D_i\to D_j)_{i<j\le\alpha}$ be a smooth well-ordered
chain of objects in $\sA$, indexed by a limit ordinal~$\alpha$,
such that $D_i\in\sC$ for all $i<\alpha$.
 Under the assumptions of~(2), we have to show that $D_\alpha\in\sC$.
 Following~\cite[Lemma~2.1]{GG} or~\cite[Proposition~4.1]{BPS},
we consider for every ordinal $\beta\le\alpha$ the short exact
sequence
\begin{equation} \label{intermediate-colimit}
 0\lrarrow K_\beta\lrarrow\coprod\nolimits_{i<\beta} D_i\lrarrow
 \varinjlim\nolimits_{i<\beta} D_i\lrarrow0
\end{equation}
of canonical presentation of the directed colimit
$\varinjlim_{i<\beta}D_i$.
 In particular, for any limit ordinal $\beta\le\alpha$,
the sequence~\eqref{intermediate-colimit} takes the form
\begin{equation} \label{smoothness-limit-ordinal-colimit}
 0\lrarrow K_\beta\lrarrow\coprod\nolimits_{i<\beta} D_i\lrarrow
 D_\beta\lrarrow0,
\end{equation}
since $D_\beta=\varinjlim_{i<\beta}D_i$ by assumption.

 For any successor ordinal $\beta=\gamma+1<\alpha$,
we have $\varinjlim_{i<\beta}D_i=\varinjlim_{i\le\gamma}D_i=D_\gamma$.
 In this case, the sequence~\eqref{intermediate-colimit} is split
and takes the form
\begin{equation} \label{split-colimit}
 0\lrarrow\coprod\nolimits_{i<\gamma}D_i\lrarrow
 \coprod\nolimits_{i\le\gamma}D_i\lrarrow D_\gamma\lrarrow0,
\end{equation}
 Here the components of the map
$\coprod_{i<\beta}D_i\rarrow D_\gamma$ are precisely
the morphisms $f_{\gamma i}\:D_i\rarrow D_\gamma$,
\,$i<\gamma$, and $\id_{D_\gamma}\:D_\gamma\rarrow D_\gamma$.
 In this case, we have a natural isomorphism
$K_\beta\simeq\coprod_{i<\gamma}D_i$.


 The short exact sequences~\eqref{intermediate-colimit}
form a smooth chain (in the category) of short exact sequences in~$\sA$.
 Proceeding by the transfinite induction, we can now prove that
$K_\beta\in\sC$ for all ordinals $\beta\le\alpha$.
 Indeed, the assertion holds for successor ordinals~$\beta$, since
the class $\sC$ is closed under coproducts (which can be easily
expressed in terms of extensions and directed colimits of smooth
well-ordered chains of monomorphisms).
 For a limit ordinal~$\beta\le\alpha$, the object $K_\beta$ is
the directed colimit of the smooth well-ordered chain of
monomorphisms $(K_i\to K_j)_{i<j<\beta}$.

 We have shown that $K_\alpha\in\sC$.
 Finally, consider the short exact
sequence~\eqref{smoothness-limit-ordinal-colimit} for $\beta=\alpha$,
\begin{equation} \label{final-colimit}
 0\lrarrow K_\alpha\lrarrow\coprod\nolimits_{i<\alpha} D_i\lrarrow
 D_\alpha\lrarrow0,
\end{equation}
and observe that $\coprod\nolimits_{i<\alpha}D_i\in\sC$ and
$K_\alpha\in\sC$.
 Since the class $\sC$ is closed under the cokernels of monomorphisms,
we can conclude that $D_\alpha\in\sC$.
\end{proof}

\begin{cor} \label{coacyclics-as-closure-extensions-directed-colimits}
 Let\/ $\bA$ be a Grothendieck abelian DG\+category.
 Then the class of all Becker-coacyclic objects in\/ $\sZ^0(\bA)$ is
closed under directed colimits.
 Moreover, the class\/ $\sZ^0(\bA)_\ac^\bco$ is precisely the closure
of the class of all contractible objects under extensions and directed
colimits.
 In fact, if\/ $\sS_0\subset\sZ^0(\bA^\bec)$ is a set of objects
such that\/ $\sZ^0(\bA^\bec)=\Fil(\sS_0)$ and\/ $\sS=\{\Psi^+(S)\mid
S\in\sS_0\}\subset\sZ^0(\bA)$, then\/ $\sZ^0(\bA)_\ac^\bco$ is
the closure of\/ $\sS$ under extensions and directed colimits
in\/~$\sZ^0(\bA)$.
\end{cor}

\begin{proof}
 The class of all Becker-coacyclic objects\/ $\sC=\sZ^0(\bA)_\ac^\bco$
in the Grothendieck abelian category $\sA=\sZ^0(\bA)$ is closed under
the cokernels of monomorphisms by
Lemma~\ref{becker-and-lp-coacyclic}(a), and under transfinitely
iterated extensions by the first assertion of the same lemma.
 So the class $\sC$ satisfies condition~(1) of
Proposition~\ref{directed-colimits-from-transfinite-extensions}.
 Therefore, the equivalent condition~(3) of the same proposition is
also satisfied, and we can conclude that the class $\sZ^0(\bA)_\ac^\bco$
is closed under directed colimits in~$\sZ^0(\bA)$.
 Furthermore, all contractible objects obviously belong to
$\sZ^0(\bA)_\ac^\bco$; so the closure of the class of contractible
objects under extensions and directed colimits is contained
in $\sZ^0(\bA)_\ac^\bco$.

 Conversely, by Corollary~\ref{coacyclics-as-filtered-by-contractibles},
all the Becker-coacyclic objects in $\bA$ can be obtained as direct
summands of transfinitely iterated extensions of objects from $\sS$
in the abelian category~$\sZ^0(\bA)$.
 A transfinitely iterated extension is built from extensions and
directed colimits; and a direct summand of an object can be expressed
as a countable directed colimit of copies of that object.
\end{proof}

\subsection{Generation theorem} \label{generation-theorem-subsecn}
 Let $\sT$ be a triangulated category.
 A class of objects $\sS\subset\sT$ is said to \emph{weakly generate}
$\sT$ if any object $X\in\sT$ such that $\Hom_\sT(S,X[n])=0$
for all $S\in\sS$ and $n\in\boZ$ vanishes in $\sT$, that is $X=0$.

 Let $\sT$ be a triangulated category with (infinite) coproducts.
 Then a class of objects $\sS\subset\sT$ is said to \emph{generate}
$\sT$ (as a triangulated category with coproducts) if the minimal full
triangulated subcategory of $\sT$ containing $\sS$ and closed under
coproducts coincides with~$\sT$.

 Notice that, for any object $X\in\sT$, the full subcategory of
all objects $Y\in\sT$ such that $\Hom_\sT(Y,X[n])=0$ for all $n\in\boZ$
is a full triangulated subcategory closed under coproducts.
 Therefore, the terminology above is consistent: if $\sS$ generates
$\sT$, then $\sS$ also weakly generates~$\sT$.

 The converse implication holds for well-generated triangulated
categories $\sT$ and \emph{sets} of objects $\sS\subset\sT$.
 If $\sT$ is a well-generated triangulated category and $\sS$ is
a weakly generating set of objects in $\sT$, then $\sS$ is also
a generating set of objects for~$\sT$
\,\cite[Theorems~7.2.1 and~5.1.1]{Kra2}.

\begin{thm} \label{generators-of-coderived-category}
 Let\/ $\bA$ be a Grothendieck abelian DG\+category and\/
$\sS\subset\sZ^0(\bA)$ be a set of objects such that the whole
abelian category\/ $\sZ^0(\bA)$ is the closure of\/ $\sS$ under
shifts, kernels of epimorphisms, cokernels of monomorphisms,
extensions, and directed colimits.
 Then the set of objects\/ $\sS$ generates the triangulated category\/
$\sD^\bco(\bA)$.
\end{thm}

\begin{proof}
 Corollary~\ref{coderived-well-generated} tells that Becker's coderived
category $\sD^\bco(\bA)$ is well-generated.
 According to the discussion above, it suffices to show that the set
$\sS$ weakly generates the triangulated category $\sD^\bco(\bA)$.

 According to Corollary~\ref{becker-coderived-corollary}, the coderived
category $\sD^\bco(\bA)$ is equivalent to the homotopy category of
graded-injective objects $\sH^0(\bA_\binj)$.
 It is clear from the definition of the class of Becker-coacyclic
objects $\sH^0(\bA)_\ac^\bco$ that the natural map of abelian groups
$\Hom_{\sH^0(\bA)}(Y,J)\rarrow\Hom_{\sD^\bco(\bA)}(Y,J)$ is
an isomorphism for all $Y\in\bA$ and $J\in\bA_\binj$.
 Let $X\in\sD^\bco(\bA)$ be an object such that
$\Hom_{\sD^\bco(\bA)}(S,X[n])=0$ for all $S\in\sS$ and $n\in\boZ$.
 Choose a graded-injective object $J\in\bA_\binj$ isomorphic to $X$
in $\sD^\bco(\bA)$.
 Then we have $\Hom_{\sH^0(\bA)}(S,J[n])=0$ for all $S\in\sS$ and
$n\in\boZ$.

 Denote by $\sY\subset\sZ^0(\bA)$ the class of all objects $Y$ such
that $\Hom_{\sH^0(\bA)}(Y,J[n])=0$ for all $n\in\boZ$.
 Clearly, the class of objects $\sY$ is closed under shifts and
cones.
 Furthermore, by the definition, all Becker-coacyclic objects in
$\bA$ belong to $\sY$, that is $\sZ^0(\bA)_\ac^\bco\subset\sY$.
 In particular, all the totalizations of short exact sequences in
$\sZ^0(\bA)$ belong to $\sY$ by Lemma~\ref{becker-and-lp-coacyclic}(a).
 It follows that the class $\sY$ is closed under the kernels of
epimorphisms, cokernels of monomorphisms, and extensions
in~$\sZ^0(\bA)$.

 Finally, by Lemma~\ref{Ext-in-cocycles-and-Hom-in-homotopy},
the class $\sY$ can be described as $\sY={}^{\perp_1}\{J[n]\mid n\in\boZ\}
\subset\sZ^0(\bA)$.
 By Lemma~\ref{eklof-lemma}, the class $\sY$ is closed under
transfinitely iterated extensions in~$\sZ^0(\bA)$.
 Now Proposition~\ref{directed-colimits-from-transfinite-extensions}%
\,(1)\,$\Rightarrow$\,(3) tells that the class $\sY$ is closed
under directed colimits in~$\sZ^0(\sY)$.
 Since $\sS\subset\sY$ by assumption, we can conclude that
the class $\sY$ coincides with the whole abelian category~$\sZ^0(\bA)$.
 Thus $J\in\sY$, and it follows that the object $J$ is contractible.
 Hence $X=0$ in $\sD^\bco(\bA)$.
\end{proof}

 Let us formulate explicitly the natural particular case of
Theorem~\ref{generators-of-coderived-category} for locally
finitely presentable abelian DG\+categories (as defined in
Section~\ref{locally-presentable-DG-categories-subsecn}).

\begin{cor} \label{generators-of-coderived-for-locally-presentable}
 Let\/ $\bA$ be a locally finitely presentable abelian DG\+category
and\/ $\sS$ be the set of all (representatives of isomorphism
classes of) finitely presentable objects in the locally finitely
presentable abelian category\/~$\sZ^0(\bA)$.
 Then the set of objects\/ $\sS$ generates Becker's coderived
category\/ $\sD^\bco(\bA)$.
\end{cor}

\begin{proof}[First proof]
 We recall that any locally finitely presentable abelian DG\+category
is Grothendieck (see Section~\ref{Grothendieck-DG-categories-subsecn}).
 All the objects of the locally finitely presentable category
$\sZ^0(\bA)$ can be obtained as directed colimits of finitely
presentable objects, so Theorem~\ref{generators-of-coderived-category}
is applicable.
\end{proof}

 Another proof of
Corollary~\ref{generators-of-coderived-for-locally-presentable},
working in the special case of a locally coherent abelian DG\+category
$\bA$,  will be given below in
Section~\ref{full-and-faithfulness-and-compactness-subsecn}.

\Section{Locally Coherent DG-Categories}  \label{locally-coherent-secn}

 The aim of this section is to prove that Becker's coderived category
$\sD^\bco(\bA)$ is compactly generated, and describe (up to direct
summands) its full subcategory of compact objects, for a locally
coherent abelian DG\+category~$\bA$.
 In particular, this result applies to the coderived category
$\sD^\bco(\biR^\cu\bModl)$ of CDG\+modules over a CDG\+ring $(R^*,d,h)$
whose underlying graded ring $R^*$ is graded left coherent.
 The main results are Theorem~\ref{becker-coderived-compactly-generated}
and Corollary~\ref{CDG-modules-becker-coderived-comp-gen}.

\subsection{Locally finitely presentable abelian categories}
\label{loc-fin-pres-abelian-subsecn}
 The definitions of a finitely presentable object and a locally finitely
presentable category were given in Section~\ref{small-object-subsecn}
(based on the book~\cite{AR}, which we use as the standard reference).
 Other important references include the papers~\cite{Len,CB,Kra0}.

 Notice that the categories called ``locally finitely presented''
in~\cite{CB,Kra0} are called ``finitely accessible'' in
the terminology of~\cite{AR}.
 These categories need not have finite colimits or finite limits,
but only directed colimits.
 So the general setting in~\cite{CB,Kra0} is more general than ours in
this section.

 Before stating the following lemma, we recall that an object
$E\in\sA$ is said to be \emph{finitely generated} if
the functor $\Hom_\sA(E,{-})\:\sA\rarrow\Sets$ preserves
the colimits of directed diagrams of
monomorphisms~\cite[Section~1.E]{AR}.
 An object $E\in\sA$ is finitely generated if and only if it cannot be
represented as the union of an infinite directed set of its proper
subobjects \cite[Proposition V.3.2]{Sten}.

\begin{lem} \label{finitely-presentable-closed-under-extensions}
 In any locally finitely presentable abelian category, the class of all
finitely presentable objects is closed under extensions.
\end{lem}

\begin{proof}
 Let $\sA$ be a locally finitely presentable (abelian) category.
The following properties are straightforward to prove:
\begin{itemize}
\item the class of all finitely generated objects is closed under
quotients and extensions in~$\sA$ \cite[Lemma V.3.1]{Sten};
\item an object in $\sA$ is finitely generated if and only if it is
a quotient of a finitely presentable object;
\item a finitely generated object $F\in\sA$ is finitely presentable if
and only if, for any epimorphism $f\:E\rarrow F$ onto $F$ from
a finitely generated object $E\in\sA$, the kernel of~$f$ is
finitely generated \cite[Proposition V.3.4]{Sten}.
\end{itemize}

Suppose now that $0\rarrow X\rarrow Y\rarrow Z\rarrow 0$ is a short exact sequence in $\sA$ with $X$ and $Z$ finitely presentable and that $f\:E\rarrow Y$ is an epimorphism from a finitely generated object $E$. Then we can form a commutative diagram with exact rows,
\[
\xymatrix{
  0 \ar[r] & D \ar[r] \ar@{->>}[d]_g & E \ar[r] \ar@{->>}[d]_f & Z \ar[r] \ar@{=}[d] & 0 \\
  0 \ar[r] & X \ar[r] & Y \ar[r] & Z \ar[r] & 0
}
\]
where $D$ is the pullback of $X\rarrow Y\longleftarrow E$. Since $D$ is the kernel of $E\rarrow Z$ and $Z$ is finitely presentable, it follows that $D$ is finitely generated. Consequently, the kernel of $g$, which coincides with the kernel of $f$, is finitely generated since $X$ is finitely presentable. Since this holds for any $f$ as above, it follows that $Y$ is finitely presentable.
\end{proof}

 Given a category $\sA$ with directed colimits and a class of objects
$\sC\subset\sA$, we denote by $\varinjlim\sC\subset\sA$ the class of
all colimits of directed diagrams of objects from~$\sC$ in~$\sA$.
 For a locally finitely presentable category $\sA$, we denote by
$\sA_\fp\subset\sA$ the full subcategory of all finitely presentable
objects.

\begin{prop} \label{varinjlim-of-class-of-finitely-presentables}
 Let\/ $\sA$ be a locally finitely presentable additive category
and\/ $\sC\subset\sA_\fp$ be a full subcategory closed under finite
direct sums.
 Then the class of objects\/ $\varinjlim\sC\subset\sA$ is closed under
coproducts and directed colimits in\/~$\sA$.
 An object $B\in\sA$ belongs to\/ $\varinjlim\sC$ if and if, for any
object $A\in\sA_\fp$, any morphism $A\rarrow B$ in\/ $\sA$ factorizes
through an object from\/~$\sC$.
\end{prop}

\begin{proof}
 This is~\cite[Proposition~2.1]{Len},
\cite[Section~4.1]{CB}, or~\cite[Proposition~5.11]{Kra0}.
\end{proof}

\subsection{Locally coherent abelian categories}
\label{locally-coherent-abelian-subsecn}
 Let $\sA$ be a locally finitely presentable abelian category.
 A finitely presentable object $F\in\sA$ is said to be \emph{coherent} 
if any finitely generated subobject of $F$ (as defined in
Section~\ref{loc-fin-pres-abelian-subsecn}) is finitely presentable.
 The category $\sA$ is said to be \emph{locally
coherent}~\cite[Section~2]{Roo} if it has a generating set consisting
of coherent objects.
 Equivalently, $\sA$ is locally coherent if and only if the kernel
of any (epi)morphism of finitely presentable objects in $\sA$ is
finitely presentable.

 In a locally coherent abelian category $\sA$, an object is coherent
if and only if it is finitely presentable, and the full subcategory
of all coherent objects $\sA_\fp$ is closed under kernels, cokernels,
and extensions in~$\sA$.
 So, $\sA_\fp$ is an abelian category; in the terminology of
Section~\ref{exactly-embedded-subsecn}, \,$\sA_\fp$ is an exactly
embedded full abelian subcategory of~$\sA$.
 We refer to~\cite[Section~9.5]{Pedg} for a further discussion.

 In the following lemma, we denote by $\Ab$ the category of abelian
groups.

\begin{lem} \label{ext-from-coherent-preserves-directed-colimits}
 Let\/ $\sA$ be a locally coherent abelian category and $E\in\sA_\fp$ be
a finitely presentable object.
 Then, for every $n\ge0$, the functor\/ $\Ext_\sA^n(E,{-})\:\sA\rarrow
\Ab$ preserves directed colimits.
\end{lem}

\begin{proof}
 In any abelian category $\sA$, the functor $\Ext_\sA^n(X,Y)$ can be
computed as the filtered colimit of the cohomology groups
$H^n\Hom_\sA(R_\bu,Y)$, taken over the (large) filtered category of
all exact complexes $\dotsb\rarrow R_2\rarrow R_1\rarrow R_0\rarrow
X\rarrow0$ in~$\sA$.
 Here the morphisms in the category of such arbitrary resolutions
$R_\bu\rarrow X$ are the usual closed morphisms of complexes acting
by the identity morphisms on the object $X$ and viewed up to
the cochain homotopy.

 In the situation at hand with $\sA$ locally coherent, the full
subcategory $\sA_\fp\subset\sA$ has the property
of~\cite[Section~12]{Kel2}, or in other words, it is ``self-resolving''
in the sense of~\cite[Section~7.1]{Pedg}.
 This means that for any epimorphism $A\rarrow F$ with $A\in\sA$ and
$F\in\sA_\fp$ there exists an epimorphism $G\rarrow F$ with
$G\in\sA_\fp$ and a morphism $G\rarrow A$ in $\sA$ such that
the triangle diagram $G\rarrow A\rarrow F$ is commutative in~$\sA$.
 Furthermore, the kernel of the morphism $G\rarrow F$ also belongs
to~$\sA_\fp$.

 Consequently, the full subcategory of resolutions $\dotsb\rarrow
F_2\rarrow F_1\rarrow F_0\rarrow E\rarrow0$ with $F_i\in\sA_\fp$ is
cofinal in the category of all resolutions $\dotsb\rarrow R_2\rarrow
R_1\rarrow R_0\rarrow E\rarrow0$ with $R_i\in\sA$ (when $E\in\sA_\fp$).
 So one can compute the group $\Ext_\sA^n(E,Y)$ as the filtered
colimit of $H^n\Hom_\sA(F_\bu,Y)$ taken over all the resolutions
$F_\bu\rarrow E$ with $F_i\in\sA_\fp$.
 Now the functor $\Hom_\sA(F,{-})\:\sA\rarrow\Ab$, by the definition,
preserves directed colimits when $F\in\sA_\fp$.
 The functors of cohomology of a complex of abelian groups also
preserve directed colimits, and the filtered colimit over the category
of all resolutions $F_\bu$ commutes with directed colimits.
\end{proof}

 A version of the following proposition for module categories can be
found in~\cite[Theorem~2.3]{AT}.

\begin{prop} \label{varinjlim-closed-under-extensions}
 Let\/ $\sA$ be a locally coherent abelian category and\/ $\sC\subset
\sA_\fp$ be a full subcategory closed under extensions.
 Then the full subcategory\/ $\varinjlim\sC\subset\sA$ is also
closed under extensions in\/~$\sA$.
\end{prop}

\begin{proof}
 Given an abelian category $\sA$ and two classes of objects
$\sX$, $\sY\subset\sA$, let us denote by $\sX*\sY$ the class of
of all objects $Z\in\sA$ for which there exists a short exact
sequence $0\rarrow X\rarrow Z\rarrow Y\rarrow0$ in~$\sA$.
 In the situation at hand, the proposition claims that
$\varinjlim\sC*\varinjlim\sC\subset\varinjlim\sC$.
 We will prove three inclusions
$$
 \varinjlim\sC*\varinjlim\sC\subset\varinjlim(\varinjlim\sC*\sC)
 \subset\varinjlim\varinjlim(\sC*\sC)\subset\varinjlim(\sC*\sC)
$$
for any class of objects $\sC\subset\sA_\fp$.

 Firstly, given a cocomplete abelian category $\sA$ with exact
directed colimits and two classes of objects $\sD$, $\sE\subset\sA$,
we claim that $\sD*\varinjlim\sE\subset\varinjlim(\sD*\sE)$.
 Let $(E_i)_{i\in I}$ be a directed system in $\sA$, indexed by
a directed poset $I$, with $E_i\in\sE$, and let
\begin{equation} \label{first-original-ses}
 0\lrarrow D\lrarrow G\lrarrow\varinjlim\nolimits_{i\in I} E_i\lrarrow0
\end{equation}
be a short exact sequence in $\sA$ with $D\in\sD$.
 Taking the pullback of~\eqref{first-original-ses}
with respect to the natural morphism $E_j\rarrow\varinjlim_{i\in I}E_i$,
for every $j\in I$, we obtain a short exact sequence
\begin{equation} \label{pullback-of-first-original-ses}
 0\lrarrow D\lrarrow G_j\lrarrow E_j\lrarrow0.
\end{equation}
 As $j\in I$ varies, the short exact
sequences~\eqref{pullback-of-first-original-ses} form
a directed system whose directed colimit is the original short
exact sequence~\eqref{first-original-ses}.
 Thus $G_j\in\sD*\sE$ and $G=\varinjlim_{j\in I}G_j$.

 Secondly, for a locally coherent abelian category $\sA$ and two
classes of objects $\sC\subset\sA$ and $\sD\subset\sA_\fp$, we
assert that $(\varinjlim\sC)*\sD\subset\varinjlim(\sC*\sD)$.
 Let $(C_i)_{i\in I}$ be a directed system in $\sA$, indexed by
a directed poset $I$, with $C_i\in\sC$, and let
\begin{equation} \label{second-original-ses}
 0\lrarrow\varinjlim\nolimits_{i\in I}C_i\lrarrow H\lrarrow D\lrarrow0
\end{equation}
be a short exact sequence in $\sA$ with $D\in\sD$.
 By Lemma~\ref{ext-from-coherent-preserves-directed-colimits}
(for $n=1$), the related class in $\Ext^1_\sA(D,\varinjlim_{i\in I}C_i)$
comes from an element of $\Ext^1_\sA(D,C_k)$ for some index $k\in I$.
 Consider the related short exact sequence
\begin{equation} \label{preimage-of-second-original-ses}
 0\lrarrow C_k\lrarrow H_k\lrarrow D\lrarrow0
\end{equation}
 For every $j\in I$, \,$j\ge k$, take the pushout of the short
exact sequence~\eqref{preimage-of-second-original-ses} with respect
to the morphism $C_k\rarrow C_j$:
\begin{equation} \label{pushout-of-preimage-of-second-original-ses}
 0\lrarrow C_j\lrarrow H_j\lrarrow D\lrarrow0.
\end{equation}
 As the index $j\in I$, \,$j\ge k$ varies, the short exact
sequences~\eqref{pushout-of-preimage-of-second-original-ses} form
a directed system whose directed colimit is the original short
exact sequence~\eqref{second-original-ses}.
 Therefore, $H_j\in\sC*\sD$ and $H=\varinjlim_{j\in I}^{j\ge k}H_j$.

 Finally, for any class of objects $\sC\subset\sA_\fp$ we have
$\sC*\sC\subset\sA_\fp$ by
Lemma~\ref{finitely-presentable-closed-under-extensions} and
$\varinjlim\varinjlim\sC\subset\varinjlim\sC$ by
Proposition~\ref{varinjlim-of-class-of-finitely-presentables}.
\end{proof}

\subsection{Locally coherent abelian DG-categories}
 In this section, which is largely an extraction
from~\cite[Section~9.5]{Pedg}, we prove the (locally) coherent
versions of the results of
Sections~\ref{locally-presentable-DG-categories-subsecn}
and~\ref{Grothendieck-DG-categories-subsecn}.

\begin{lem} \label{Xi-preserves-coherence}
 Let\/ $\bA$ be a locally finitely presentable abelian DG\+category.
 Then the additive functor\/ $\Xi_\bA\:\sZ^0(\bA)\rarrow\sZ^0(\bA)$
\emph{preserves} coherence of objects.
\end{lem}

\begin{proof}
 Using the properties listed in the proof of
Lemma~\ref{finitely-presentable-closed-under-extensions}, one can
easily show that the class of all coherent objects in a locally
finitely presentable abelian category is closed under extensions.
 Hence the assertion of the lemma follows from
Corollary~\ref{Xi-extension-cokernel}(a).
\end{proof}

\begin{lem} \label{Phi-Psi-preserve-reflect-coherence}
 Let\/ $\bA$ be a locally finitely presentable abelian DG\+category.
 Then the additive functors\/ $\Phi_\bA\:\sZ^0(\bA)\rarrow
\sZ^0(\bA^\bec)$ and\/ $\Psi^+_\bA\:\sZ^0(\bA^\bec)\rarrow
\sZ^0(\bA)$ preserve \emph{and} reflect coherence of objects.
\end{lem}

\begin{proof}
 This is our version of~\cite[Lemma~9.13]{Pedg}.
 Let us show that the functor $\Phi_\bA$ reflects coherence.
 Let $E\in\bA$ be an object for which the object $\Phi(E)$ is
coherent in~$\sZ^0(\bA^\bec)$.
 By Lemma~\ref{Phi-Psi-lambda-presentability}, the object $E$
is finitely presentable in~$\sZ^0(\bA)$.

 Let $F\subset E$ be a finitely generated subobject.
 It is easily provable similarly to
Lemma~\ref{Phi-Psi-lambda-presentability}
(cf.~\cite[Lemma~9.6]{Pedg}) that the functor $\Phi$ preserves
finite generatedness of objects; so the object $\Phi(F)$ is
finitely generated.
 Therefore, $\Phi(F)$ is a finitely generated subobject in $\Phi(E)$;
by assumption, it follows that the object $\Phi(F)$ is finitely
presentable in $\sZ^0(\bA^\bec)$.
 It remains to invoke Lemma~\ref{Phi-Psi-lambda-presentability}
again in order to conclude that the object $F\in\sZ^0(\bA)$ is
finitely presentable.

 Similarly one shows that the functor $\Psi^+_\bA$ reflects
coherence.
 Since the compositions $\Xi_\bA\simeq\Psi^+_\bA\circ\Phi_\bA$
and $\Xi_{\bA^\bec}\simeq\Phi_\bA\circ\Psi^-_\bA$ (see the last
paragraph of Section~\ref{Phi-and-Psi-subsecn}) preserve
coherence by Lemma~\ref{Xi-preserves-coherence}, and the functors
$\Psi^+_\bA$ and $\Phi_\bA$ reflect coherence, it follows that
the functors $\Phi_\bA$ and~$\Psi^-_\bA$ preserve coherence.
\end{proof}

\begin{prop} \label{locally-coherent-DG-category}
 Let\/ $\bA$ be a locally finitely presentable abelian DG\+category.
 Then the abelian category\/ $\sZ^0(\bA)$ is locally coherent if and
only if the abelian category\/ $\sZ^0(\bA^\bec)$ is locally coherent.
\end{prop}

\begin{proof}
 This is~\cite[Proposition~9.14]{Pedg}.
 The assertion follows from Lemmas~\ref{Phi-Psi-strong-generators}
and~\ref{Phi-Psi-preserve-reflect-coherence} above.
\end{proof}

 We will say that an abelian DG\+category $\bA$ is \emph{locally
coherent} if both the abelian categories $\sZ^0(\bA)$ and
$\sZ^0(\bA^\bec)$ are locally coherent.
 In other words, this means that $\bA$ satisfies the assumptions
and the equivalent conditions of
Proposition~\ref{locally-coherent-DG-category}.

 Given a locally finitely presentable abelian DG\+category $\bA$,
we denote by $\bA_\bfp\subset\bA$ the full DG\+subcategory whose
objects are all the objects of $\bA$ that are finitely presentable
as objects of the category~$\sZ^0(\bA)$.
 So, by the definition, we have $\sZ^0(\bA_\bfp)=\sZ^0(\bA)_\fp$.
 The full DG\+subcategory $\bA_\bfp\subset\bA$ is obviously closed
under shifts and finite direct sums; in view of
Lemmas~\ref{cone-kernel-cokernel}
and~\ref{finitely-presentable-closed-under-extensions}, it is also
closed under cones.
 Lemma~\ref{Phi-Psi-lambda-presentability} implies that the full
subcategory $\sZ^0((\bA_\bfp)^\bec)\subset\sZ^0(\bA^\bec)$ consists
of all the finitely presentable objects in $\sZ^0(\bA^\bec)$,
that is $\sZ^0((\bA_\bfp)^\bec)=\sZ^0(\bA^\bec)_\fp$.

 When $\bA$ is a locally coherent abelian DG\+category, the full
DG\+subcategory $\bA_\bfp\subset\bA$ is an exactly embedded full
abelian DG\+subcategory in $\bA$ in the sense of
Section~\ref{exactly-embedded-subsecn}.
 In particular, $\bA_\bfp$ is an abelian DG\+category.

\begin{exs} \label{locally-coherent-DG-category-of-complexes}
 (1)~Let $\sA$ be a locally finitely presentable abelian category.
 Consider the abelian DG\+category $\bC(\sA)$ of complexes in $\sA$,
as per Example~\ref{abelian-DG-category-of-complexes-example}.
 According to
Sections~\ref{DG-category-of-complexes-defined-subsecn}
and~\ref{almost-involution-complexes-subsecn}, we have
$\sZ^0(\bC(\sA)^\bec)\simeq\sG(\sA)=\sA^\boZ$.
 By Example~\ref{locally-presentable-abelian-DG-of-complexes}(1),
\,$\bC(\sA)$ is a locally finitely presentable abelian DG\+category.

 One can easily see that a graded object in $\sA$ is finitely
presentable if and only if it is \emph{bounded} with all but a finite
number of grading components vanishing \emph{and} all the grading
components finitely presentable.
 By Lemma~\ref{Phi-Psi-lambda-presentability},
the finitely presentable objects of the abelian category $\sC(\sA)=
\sZ^0(\bC(\sA))$ of complexes in $\sA$ are precisely all the bounded
complexes of finitely presentable objects in~$\sA$.

 (2)~Now if $\sA$ is a locally coherent abelian category, then so
is $\sG(\sA)=\sA^\boZ$.
 By Proposition~\ref{locally-coherent-DG-category}, it follows that
$\bC(\sA)$ is a locally coherent abelian DG\+category.
\end{exs}

 A graded ring $R^*$ is said to be \emph{graded left coherent} if
all finitely generated homogeneous left ideals in $R^*$ are finitely
presentable (as graded left $R^*$\+modules).
 In other words, a graded ring $R^*$ is graded left coherent if and
only if the abelian category of graded left $R^*$\+modules $R^*\sModl$
is locally coherent.

\begin{cor} \label{finitely-presentable-CDG-modules-coherent-CDG-rings}
 Let $\biR^\cu=(R^*,d,h)$ be a CDG\+ring and $\biM^\cu=(M^*,d_M)$ be a left
CDG\+module over~$\biR^\cu$.
 Then \par
\textup{(a)} $M^*$ is finitely presentable as a graded left
$R^*[\delta]$\+module (i.~e., as an object of the category
$\sZ^0(\biR^\cu\bModl)\simeq R^*[\delta]\sModl$) if and only if it is
finitely presentable as a graded left $R^*$\+module; \par
\textup{(b)} the graded ring $R^*[\delta]$ is graded left coherent
if and only if the graded ring $R^*$ is graded left coherent, and
if and only if the locally finitely presentable abelian DG\+category
$\biR^\cu\bModl$ is locally coherent.
\end{cor}

\begin{proof}
 The equivalences of abelian categories $\sZ^0(\biR^\cu\bModl)\simeq
R^*[\delta]\sModl$ and $\sZ^0((\biR^\cu\bModl)^\bec)\simeq R^*\sModl$
were discussed in Section~\ref{almost-involution-CDG-modules-subsecn}.
 In view of these category equivalences, part~(a) is
a particular case of Lemma~\ref{Phi-Psi-lambda-presentability},
and part~(b) is a particular case of
Proposition~\ref{locally-coherent-DG-category}.
 We recall that the DG\+category $\biR^\cu\bModl$ is always locally
finitely presentable by
Example~\ref{locally-presentable-abelian-DG-of-CDG-modules}.
 (See also~\cite[Example~9.15]{Pedg}.)
\end{proof}

\begin{ex} \label{locally-coherent-DG-category-of-factorizations}
 (1)~Let $\sA$ be a locally finitely presentable abelian category,
$\Delta\:\sA\rarrow\sA$ be an autoequivalence, and $v\:\Id_\sA\rarrow
\Delta$ be a potential.
 Consider the abelian DG\+category $\bF(\sA,\Delta,v)$ of
factorizations of~$v$ in $\sA$, as per
Example~\ref{abelian-DG-category-of-factorizations-example}.
 According to
Sections~\ref{DG-category-of-factorizations-defined-subsecn}
and~\ref{almost-involution-factorizations-subsecn}, we have
$\sZ^0(\bF(\sA,\Delta,v)^\bec)\simeq\sP(\sA,\Delta)\simeq\sA\times\sA$.
 By Example~\ref{locally-presentable-abelian-DG-of-factorizations}(1),
\,$\bF(\sA,\Delta,v)$ is a locally finitely presentable abelian
DG\+category.

 Obviously, a $2$\+$\Delta$-periodic object in $\sA$ is finitely
presentable if and only if all its grading components are finitely
presentable in~$\sA$ (there are essentially only two such grading
componets).
 By Lemma~\ref{Phi-Psi-lambda-presentability},
the finitely presentable objects of the abelian category of
factorizations $\sF(\sA,\Delta,v)=\sZ^0(\bF(\sA,\Delta,v))$ are
precisely all the factorizations of~$v$ with finitely presentable
grading components in~$\sA$.

 (2)~Now if $\sA$ is a locally coherent abelian category, then so is
$\sP(\sA,\Delta)\simeq\sA\times\sA$.
 By Proposition~\ref{locally-coherent-DG-category}, it follows that
$\bF(\sA,\Delta,v)$ is a locally coherent abelian DG\+category.
\end{ex}

\subsection{Absolute derived category}
\label{absolute-derived-subsecn}
 The following definition is taken from~\cite[Section~5.1]{Pedg},
but it goes back to~\cite[Section~2.1]{Psemi}
and~\cite[Sections~3.3 and~4.2]{Pkoszul}.

 Let $\bE$ be an abelian DG\+category.
 An object $X\in\bE$ is said to be \emph{absolutely acyclic} if it
belongs to the minimal thick subcategory of the homotopy category
$\sH^0(\bE)$ containing the totalizations $\Tot(K\to L\to M)$ of
all short exact sequences $0\rarrow K\rarrow L\rarrow M\rarrow0$
in the abelian category~$\sZ^0(\bE)$.
 The thick subcategory of absolutely acyclic objects is denoted by
$\sH^0(\bE)_\ac^\abs\subset\sH^0(\bE)$, and its full preimage
under the obvious functor $\sZ^0(\bE)\rarrow\sH^0(\bE)$ is denoted
by $\sZ^0(\bE)_\ac^\abs\subset\sZ^0(\bE)$.

 The triangulated Verdier quotient category $\sD^\abs(\bE)=
\sH^0(\bE)/\sH^0(\bE)_\ac^\abs$ of the homotopy category $\sH^0(\bE)$
by its thick subcategory of absolutely acyclic objects is called
the \emph{absolute derived category} of an abelian DG\+category~$\bE$.

\begin{lem} \label{absolutely-acyclic-closure-properties}
 The full subcategory of absolutely acyclic objects\/
$\sZ^0(\bE)_\ac^\abs$ is closed under the kernels of epimorphisms,
the cokernels of monomorphisms, extensions, and direct summands in
the abelian category\/~$\sZ^0(\bE)$
(i.~e.\ it is thick there in the sense of Section~\ref{abelian-model-structures-subsecn}).
\end{lem}

\begin{proof}
 By construction, the full subcategory of absolutely acyclic objects
is closed under shifts, cones, and direct summands in $\sH^0(\bE)$,
hence also in~$\sZ^0(\bE)$.
 Since the full subcategory $\sZ^0(\bE)_\ac^\abs$ contains
the totalizations of short exact sequences, it follows that it is
closed under the kernels of epimorphisms, the cokernels of
monomorphisms, and extensions in~$\sZ^0(\bE)$.
\end{proof}

\begin{prop} \label{absolute-acyclics-described}
 For any abelian DG\+category\/ $\bE$, the full subcategory of
absolutely acyclic objects\/ $\sZ^0(\bE)_\ac^\abs$ is precisely
the closure of the class of all contractible objects in\/ $\bE$
under extensions and direct summands in the abelian
category\/~$\sZ^0(\bE)$.
\end{prop}

\begin{proof}
 Denote by $\sC\subset\sZ^0(\bE)$ the closure of the class of
all contractible objects under extensions and direct summands.
 Since all contractible objects are absolutely acyclic, it follows
from Lemma~\ref{absolutely-acyclic-closure-properties} that
$\sC\subset\sZ^0(\bE)_\ac^\abs$.

 To prove the converse inclusion, we have to show that the class $\sC$
contains the totalizations of short exact sequences in $\sZ^0(\bE)$
and is closed under shifts, cones, homotopy equivalences, and
direct summands in the homotopy category~$\sH^0(\bE)$.
 Indeed, the totalizations of short exact sequences can be obtained
as extensions of contractible objects, as explained in the proof
of Lemma~\ref{becker-and-lp-coacyclic}.
 The class $\sC$ is closed under shifts, because the class of all
contractible objects is closed under shifts and the shift is
an auto-equivalence of the abelian category~$\sZ^0(\bE)$.
 The cone of a closed morphism of degree~$0$ in $\sZ^0(\bE)$ is
a particular case of an extension
(see Lemma~\ref{cone-kernel-cokernel}).

 It remains to explain that the class $\sC$ is closed under homotopy
equivalences and homotopy direct summands, and it suffices to
consider the homotopy direct summands.
 Let $C\in\sC$ and $X\in\sZ^0(\bE)$ be two objects and
$X\overset f\rarrow C\overset g\rarrow X$ be two morphisms in
$\sZ^0(\bE)$ such that the composition~$gf$ is homotopic to
the identity endomorphism of the object~$X$.
 Then the difference $\id_X-gf\:X\rarrow X$ is homotopic to zero
in~$\bE$.
 It follows that the morphism $\id_X-gf$ in the category $\sZ^0(\bE)$ 
factorizes through the natural morphism $X\rarrow\cone(\id_X)$.
 Thus the identity endomorphism of the object $X$ in $\sZ^0(\bE)$
factorizes as $X\rarrow C\oplus\cone(\id_X)\rarrow X$, so $X$ is
a direct summand of the object $C\oplus\cone(\id_X)$ in
the abelian category~$\sZ^0(\bE)$.
 Obviously, $C\oplus\cone(\id_X)\in\sC$, hence $X\in\sC$, and
we are done.
\end{proof}

\begin{prop} \label{coacyclics-as-varinjlim-of-abs-acyclics}
 Let\/ $\bA$ be a locally coherent abelian DG\+category.
 Then the class of all Becker-coacyclic objects $\sZ^0(\bA)_\ac^\bco$
consists precisely of all the directed colimits, taken in
the abelian category\/ $\sZ^0(\bA)$, of (directed diagrams of) objects
absolutely acyclic with respect to the abelian
DG\+category\/~$\bA_\bfp$.
 In other words, $\sZ^0(\bA)_\ac^\bco=
\varinjlim\bigl(\sZ^0(\bA_\bfp)_\ac^\abs\bigr)\subset\sZ^0(\bA)$.
\end{prop}

\begin{proof}
 Put $\sC=\sZ^0(\bA_\bfp)_\ac^\abs$ for brevity.
 According to Proposition~\ref{absolute-acyclics-described} applied
to the abelian DG\+category $\bE=\bA_\bfp$, the class $\sC$ is
the closure of the class of all contractible objects in $\bA_\bfp$
under extensions and direct summands in the abelian category
$\sZ^0(\bA_\bfp)$.

 By Proposition~\ref{varinjlim-of-class-of-finitely-presentables},
the class $\varinjlim\sC\subset\sZ^0(\bA)$ is closed under directed
colimits.
 By Proposition~\ref{varinjlim-closed-under-extensions}, the class
$\varinjlim\sC$ is also closed under extensions in~$\sZ^0(\bA)$.
 So $\varinjlim\sC$ is the closure of the class of all contractible
objects in $\bA_\bfp$ under extensions and directed colimits
in~$\sZ^0(\bA)$.
 (It is helpful to keep in mind that direct summands can be obtained
as countable directed colimits.)

 Furthermore, by Lemma~\ref{contractibles-described-lemma},
the contractible objects of $\bA$ are the direct summands of
the objects of the form $\Xi_\bA(A)$ with $A\in\bA$.
 The object $A$ is a directed colimit of finitely presentable objects
in $\sZ^0(\bA)$, and the functor $\Xi$ preserves directed colimits.
 So all the contractible objects of $\bA$ are direct summands
of some directed colimits of contractible objects of $\bA_\bfp$ in
the category~$\sZ^0(\bA)$.

 Hence the class $\varinjlim\sC\subset\sZ^0(\bA)$ contains all
the contractible objects of~$\bA$.
 Therefore, it can be described as the closure of the class of all
contractible objects of $\bA$ under extensions and directed
colimits in~$\sZ^0(\bA)$.
 Glancing into
Corollary~\ref{coacyclics-as-closure-extensions-directed-colimits},
we conclude that $\varinjlim\sC=\sZ^0(\bA)_\ac^\bco$, as desired.
\end{proof}

 A different proof of
Proposition~\ref{coacyclics-as-varinjlim-of-abs-acyclics}
will be indicated in
Remark~\ref{rem-another-proof-of-approachability-of-becker-coacyclics}
in the next Section~\ref{approachability-subsecn}.

\subsection{Approachability of Becker-coacyclic objects}
\label{approachability-subsecn}
 The following definition is taken from~\cite[Section~7.2]{Pedg}.
 
 Let $\sT$ be a triangulated category and $\sS$, $\sY\subset\sT$
be two full subcategories.
 An object $X\in\sT$ is said to be \emph{approachable from\/ $\sS$
via\/~$\sY$} if every morphism $S\rarrow X$ in $\sT$ with
$S\in\sS$ factorizes through an object of~$\sY$.

 Equivalently, an object $X$ is approachable from $\sS$ via $\sY$
if and only if, for every morphism $S\rarrow X$ as above
there exists an object $S'\in\sT$ and a morphism $S'\rarrow S$
with a cone belonging to $\sY$ such that the composition $S'\rarrow S
\rarrow X$ vanishes in~$\sT$.
 When (as it will be the case in our applications) $\sS$ is a full
triangulated subcategory in $\sT$ and $\sY\subset\sS$, the conditions
in this criterion imply that $S'\in\sS$.

\begin{lem} \label{approachables-are-thick-subcategory}
 Let\/ $\sT$ be a triangulated category and\/ $\sS$, $\sY\subset\sT$
be full triangulated subcategories such that\/ $\sY\subset\sS$.
 Then the full subcategory of all objects approachable from\/ $\sS$
via\/ $\sY$ is a thick subcategory in\/ $\sT$, i.~e., it is closed
under shifts, cones, and direct summands in\/~$\sT$.
\end{lem}

\begin{proof}
 This is~\cite[Lemmas~7.3 and~7.5]{Pedg}.
\end{proof}

If $\sS=\sH^0(\bA_\bfp)$, approachability is very closely related
to the closure under directed colimits as discussed in
Section~\ref{loc-fin-pres-abelian-subsecn}.

\begin{lem} \label{approachability-and-direct-limits}
Let\/ $\bA$ be a locally coherent abelian DG\+category and
$\sY\subset\sH^0(\bA_\bfp)$ be a full subcategory
closed under finite direct sums.
 Then $X\in\sH^0(\bA)$ is approachable from $\sH^0(\bA_\bfp)$
via $\sY$ if and only if $X\in\varinjlim\widetilde\sY$
in $\sZ^0(\bA)$, where $\widetilde\sY$ stands for the preimage
of\/ $\sY$ under the obvious functor
$\sZ^0(\bA_\bfp)\rarrow\sH^0(\bA_\bfp)$.
\end{lem}

\begin{proof}
The proof relies on Proposition~\ref{varinjlim-of-class-of-finitely-presentables}.
On the one hand, any morphism from a finitely presentable
object to an object $X\in\varinjlim\widetilde\sY$ factorizes
through an object of $\widetilde\sY$ in the abelian
category~$\sZ^0(\bA)$; hence it factorizes through an object
of $\sY$ in the homotopy category~$\sH^0(\bA)$.
In particular, any object $X\in\varinjlim\widetilde\sY$ is approachable.

On the other hand, suppose that $X$ is approachable
from $\sH^0(\bA_\bfp)$ via $\sY$.
 In view of Proposition~\ref{varinjlim-of-class-of-finitely-presentables},
it suffices to show that any morphism $f\:E\rarrow X$ in
the category $\sZ^0(\bA)$ from a finitely presentable object
$E\in\sZ^0(\bA_\bfp)$ factorizes through an object in $\widetilde\sY$.
 However, by approachability such a factorization exists
in the homotopy category~$\sH^0(\bA)$.
 So there is an object $Y\in\widetilde\sY$ and two
morphisms $g\:E\rarrow Y$ and $h\:Y\rarrow X$ in $\sZ^0(\bA)$
such that the morphism~$f$ is homotopic to~$hg$.
 Since the morphism $f-hg\:E\rarrow X$ is homotopic to zero in~$\bA$,
it follows that the morphism $f-hg$ in the category $\sZ^0(\bA)$
factorizes through the natural morphism $E\rarrow\cone(\id_E)$.
 Thus the morphism~$f$ in the abelian category $\sZ^0(\bA)$ factorizes
as $E\rarrow Y\oplus\cone(\id_E)\rarrow X$, and it remains to point
out that the contractible object $\cone(\id_E)$ as well as
the direct sum $Y\oplus\cone(\id_E)$ belong to $\widetilde\sY$.
\end{proof}

 As a consequence, we obtain the following proposition
which is the key technical result of
Section~\ref{locally-coherent-secn}.

\begin{prop} \label{becker-coacyclics-approachable-prop}
 Let\/ $\bA$ be a locally coherent abelian DG\+category.
 Then the class of Becker-coacyclic objects of\/ $\bA$
coincides precisely with the class of objects of\/
$\sH^0(\bA)$ which are approachable from the class $\sH^0(\bA_\bfp)$
of finitely presentable objects of\/ $\bA$ via the class 
$\sH^0(\bA_\bfp)_\ac^\abs$ of absolutely acyclic objects
with respect to the abelian DG\+category of finitely presentable
objects in\/~$\bA$.
\end{prop}

\begin{proof}
 The argument is based on
Proposition~\ref{coacyclics-as-varinjlim-of-abs-acyclics},
whose proof uses the theory of
Sections~\ref{directed-colimits-of-becker-coacyclics}
and~\ref{loc-fin-pres-abelian-subsecn}--%
\ref{locally-coherent-abelian-subsecn}.
 Indeed, by Proposition~\ref{coacyclics-as-varinjlim-of-abs-acyclics},
we have $\sZ^0(\bA)_\ac^\bco=
\varinjlim\bigl(\sZ^0(\bA_\bfp)_\ac^\abs\bigr)\subset\sZ^0(\bA)$,
and the conclusion follows from
Lemma~\ref{approachability-and-direct-limits}.
\end{proof}

\begin{rem} \label{rem-another-proof-of-approachability-of-becker-coacyclics}
One implication of the latter proposition, that each Becker-coacyclic
object is approachable from the finitely presentable objects
via absolutely acyclic finitely presentable objects, can be proved
by other means using the the theory of approachability
developed in~\cite[Sections~7.2--7.4]{Pedg}. 

 Consider the full subcategory $\sX$ in the homotopy category
$\sH^0(\bA)$ formed by all the objects approachable from
$\sH^0(\bA_\bfp)$ via $\sH^0(\bA_\bfp)_\ac^\abs$.
 By Lemma~\ref{approachables-are-thick-subcategory}, \,$\sX$ is
a thick subcategory in the triangulated category~$\sH^0(\bA)$.
 Denote by $\widetilde\sX\subset\sZ^0(\bA)$ the full preimage of $\sX$
under the obvious functor $\sZ^0(\bA)\rarrow\sH^0(\bA)$.
 By Lemma~\ref{approachability-and-direct-limits},
$\widetilde\sX=\varinjlim\sZ^0(\bA_\bfp)_\ac^\abs$ in $\sZ^0(\bA)$,
so $\widetilde\sX$ is closed under directed colimits by
Proposition~\ref{varinjlim-of-class-of-finitely-presentables}.
In fact, it is easy to prove the last fact directly from
the definition of approachability.

 There is a rather general technical result,
\cite[Proposition~7.10(a)]{Pedg}, implying that, in
the homotopy category $\sH^0(\bA)$, all the totalizations
of short exact sequences in $\sZ^0(\bA)$ are approachable from
$\sH^0(\bA_\bfp)$ via totalizations of short exact sequences in
$\sZ^0(\bA_\bfp)$ (since the full subcategory $\sZ^0((\bA_\bfp)^\bec)=
\sZ^0(\bA^\bec)_\fp$ is self-resolving in the abelian category
$\sZ^0(\bA^\bec)$, in the sense of~\cite[Section~7.1]{Pedg}).
 In the situation at hand, one could arrive at the same conclusion
by showing that all the short exact sequences in $\sZ^0(\bA)$ are
directed colimits of short exact sequences in
$\sZ^0(\bA_\bfp)=\sZ^0(\bA)_\fp$.
 Hence all the totalizations of short exact sequences in $\sZ^0(\bA)$
belong to~$\widetilde\sX$.


 It follows that the full subcategory $\widetilde\sX$ is closed
under extensions in~$\sZ^0(\bA)$ (as well as under kernels of
epimorphisms and cokernels of monomorphisms).
 Since we know that $\widetilde\sX$ is closed
under directed colimits in~$\sZ^0(\bA)$,
 it is also closed under transfinitely iterated
extensions (in the sense of the directed colimit).
 As obviously, all the contractible objects of $\bA$ belong to $\sX$,
and consequently to~$\widetilde\sX$,
 we have by Corollary~\ref{coacyclics-as-filtered-by-contractibles},
that $\sZ^0(\bA)_\ac^\bco\subset\widetilde\sX$; hence
$\sH^0(\bA)_\ac^\bco\subset\sX$.


The latter argument for a half of
Proposition~\ref{becker-coacyclics-approachable-prop}
is independent of
Proposition~\ref{coacyclics-as-varinjlim-of-abs-acyclics} and,
in fact, we can deduce 
Proposition~\ref{coacyclics-as-varinjlim-of-abs-acyclics}
back from it.
Indeed, we already know that $\sZ^0(\bA)_\ac^\bco\subset
\widetilde\sX=\varinjlim\sZ^0(\bA_\bfp)_\ac^\abs$.
On the other hand,
$\sZ^0(\bA_\bfp)_\ac^\abs\subset\sZ^0(\bA)_\ac^\bco$
and $\sZ^0(\bA)_\ac^\bco$ is closed under directed colimits
in $\sZ^0(\bA)$ by
Corollary~\ref{coacyclics-as-closure-extensions-directed-colimits}
and Proposition~\ref{absolute-acyclics-described};
hence the other inclusion.
%
%
\end{rem}

\subsection{Full-and-faithfulness and compactness}
\label{full-and-faithfulness-and-compactness-subsecn}
 We start with presenting an alternative proof of the theorem about
triangulated generation of the coderived category.

\begin{proof}[Second proof of
Corollary~\ref{generators-of-coderived-for-locally-presentable}]
 This argument is applicable in the particular case of a locally
coherent DG\+category~$\bA$.

 Let $A\in\bA$ be an object such that $\Hom_{\sD^\bco(\bA)}(E,A)=0$ for
all $E\in\bA_\bfp$.
 By the general property of the construction of the triangulated
Verdier quotient category, this means that any morphism $E\rarrow A$
in $\sH^0(\bA)$ factorizes as $E\rarrow X\rarrow A$ for some
Becker-coacyclic object $X\in\sH^0(\bA)_\ac^\bco$.
 According to Proposition~\ref{becker-coacyclics-approachable-prop},
the morphism $E\rarrow X$ in $\sH^0(\bA)$ factorizes as $E\rarrow Y
\rarrow X$ for some object $Y\in\sH^0(\bA_\bfp)_\ac^\abs$.
%
%
This, however, means that $A$ itself is approachable from $\sH^0(\bA_\bfp)$ via $\sH^0(\bA_\bfp)_\ac^\abs$ in $\sH^0(\bA)$.
Thus $A\in\sH^0(\bA)_\ac^\bco$ by Proposition~\ref{becker-coacyclics-approachable-prop}, as desired.
\end{proof}

 Let $\sT$ be a triangulated category with infinite coproducts.
 We recall than an object $S\in\sT$ is called \emph{compact} if
the functor $\Hom_\sT(S,{-})\:\sT\rarrow\Ab$ preserves coproducts.
 It is well-known~\cite[Theorems~4.1 and~2.1(2)]{Neem} that a set of
compact objects generates $\sT$ as a triangulated category with
coproducts if and only if it weakly generates~$\sT$ (in the terminology
of Section~\ref{generation-theorem-subsecn}).
 A triangulated category $\sT$ is said to be \emph{compactly generated}
if it has a set of generators consisting of compact objects.

 The following lemma describes our technique for proving
full-and-faithfulness of the triangulated functor
$\sD^\abs(\bA_\bfp)\rarrow\sD^\bco(\bA)$ and compactness of
objects in its essential image.

\begin{lem} \label{approachability-and-compactness}
 Let\/ $\sT$ be a triangulated category with infinite coproducts,
let\/ $\sS\subset\sT$ be a full triangulated subcategory, and let\/
$\sX\subset\sT$ be a strictly full triangulated subcategory closed
under coproducts.
 Let\/ $\sY\subset\sS\cap\sX$ be a full triangulated subcategory
in the intersection.
 Assume that all the objects of\/ $\sS$ are compact in\/ $\sT$ and
all the objects of\/ $\sX$ are approachable from\/ $\sS$ via $\sY$
in\/~$\sT$.
 Then the induced triangulated functor between the Verdier quotient
categories
$$
 \sS/\sY\lrarrow\sT/\sX
$$
is fully faithful, and the objects in its essential image are
compact in\/~$\sT/\sX$.
\end{lem}

\begin{proof}
 This is~\cite[Lemma~9.20]{Pedg}.
\end{proof}

 The following theorem is one of the main results of this paper.
 Its idea goes back to
Krause~\cite[Proposition~2.3]{Kra}, whose paper covers the case of 
the (DG\+category of) complexes in a locally Noetherian
abelian category.
 For complexes in a locally coherent abelian category, this result
was obtained by the second-named author of the present paper
in the preprint~\cite[Corollary~6.13]{Sto2}.

\begin{thm} \label{becker-coderived-compactly-generated}
 Let\/ $\bA$ be a locally coherent abelian DG\+category and\/
$\bA_\bfp\subset\bA$ be its (exactly embedded, full) abelian
DG\+subcategory of finitely presentable objects.
 Then the triangulated functor
$$
 \sD^\abs(\bA_\bfp)\lrarrow\sD^\bco(\bA)
$$
induced by the inclusion of abelian DG\+categories\/ $\bA_\bfp
\rightarrowtail\bA$ is fully faithful.
 The objects in its image are compact in\/ $\sD^\bco(\bA)$, and
form a set of compact generators for Becker's coderived category\/
$\sD^\bco(\bA)$.
\end{thm}

\begin{proof}
 It was mentioned already before
Corollary~\ref{coderived-well-generated} that infinite coproducts
exist in Becker's coderived category $\sD^\bco(\bA)$ and
the Verdier quotient functor $\sH^0(\bA)\rarrow\sD^\bco(\bA)$
preserves coproducts.
 Concerning the more elementary question of existence of coproducts
in the homotopy category $\sH^0(\bA)$, see
Section~\ref{co-products-in-DG-categories}.

 To prove the theorem, we apply
Lemma~\ref{approachability-and-compactness} to the triangulated
category $\sT=\sH^0(\bA)$ with its full triangulated subcategories
$\sX=\sH^0(\bA)_\ac^\bco\supset\sY=\sH^0(\bA_\bfp)_\ac^\abs\subset
\sS=\sH^0(\bA_\bfp)$.
 The full subcategory of Becker-coacyclic objects $\sH^0(\bA)_\ac^\bco$
is closed under coproducts in $\sH^0(\bA)$ by
Lemma~\ref{becker-and-lp-coacyclic}(b).
 All objects of $\sH^0(\bA)_\ac^\bco$ are approachable from
$\sH^0(\bA_\bfp)$ via $\sH^0(\bA_\bfp)_\ac^\abs$ by
Proposition~\ref{becker-coacyclics-approachable-prop}.

 In order to finish verifying the assumptions of
Lemma~\ref{approachability-and-compactness}, it remains to mention
the essentially trivial observation that all finitely presentable
objects are compact in the homotopy category~$\sH^0(\bA)$.
 This can be explained in many ways; e.~g., one can observe that,
for any $E\in\bA_\bfp$, the DG\+functor $\Hom_\bA^\bu(E,{-})\:
\bA\rarrow\bC(\Ab)$ preserves coproducts (cf.~\cite[Lemma~9.18]{Pedg}),
and the passage to the cohomology of a complex of abelian groups
preserves coproducts.
 Alternatively, the fact that the functor $\Hom_{\sH^0(\bA)}(E,{-})\:
\sH^0(\bA)\rarrow\Ab$ preserves coproducts is easily deduced
directly from the facts that the functor $\Hom_{\sZ^0(\bA)}(E,{-})\:
\sZ^0(\bA)\rarrow\Ab$ does and the natural functor $\sZ^0(\bA)
\rarrow\sH^0(\bA)$ is surjective on objects and morphisms.

 Now Lemma~\ref{approachability-and-compactness} tells that
the triangulated functor $\sD^\abs(\bA_\bfp)\rarrow\sD^\bco(\bA)$
is fully faithful and all the objects in its image are compact.
 The assertion that the image of this functor generates
the target category is provided by
Corollary~\ref{generators-of-coderived-for-locally-presentable}.
\end{proof}

 In the case of CDG\+modules over a graded Noetherian CDG\+ring,
the following result was worked out in the first-named author's
memoir~\cite[Section~3.11]{Pkoszul} (the main part of the argument
is due to Arinkin).
 For a CDG\+ring $\biR^\cu$ whose underlying graded ring $R^*$ is
left coherent and all homogeneous left ideals in $R^*$ have at most
$\aleph_n$ generators (where $n$~is a fixed integer), a proof
can be found in~\cite[Theorem~9.39 and Corollary~9.42]{Pedg}.
 One reason why the additional assumption was needed is that
the coderived categories in the sense of Positselski, rather than in
the sense of Becker, are considered in~\cite{Pedg};
see~\cite[Theorem~9.38]{Pedg} for a comparison theorem.

\begin{cor} \label{CDG-modules-becker-coderived-comp-gen}
 Let\/ $\biR^\cu=(R^*,d,h)$ be a curved DG\+ring such that the graded
ring $R^*$ is graded left coherent.
 Let $\biR^\cu\bModl$ be the locally coherent abelian DG\+category of
left CDG\+modules over $\biR^\cu$, and let $\biR^\cu\bModl_\bfp\subset
\biR^\cu\bModl$ be the (exactly embedded full abelian) DG\+subcategory
of all CDG\+modules with finitely presentable underlying graded
left $R^*$\+modules.
 Then the induced triangulated functor from the absolute derived
category to Becker's coderived category
$$
 \sD^\abs(\biR^\cu\bModl_\bfp)\lrarrow\sD^\bco(\biR^\cu\bModl)
$$
is fully faithful.
 The objects in its image are compact in\/ $\sD^\bco(\biR^\cu\bModl)$,
and form a set of compact generators for\/ $\sD^\bco(\biR^\cu\bModl)$.
\end{cor}

\begin{proof}
 This is the particular case of
Theorem~\ref{becker-coderived-compactly-generated} for
$\bA=\biR^\cu\bModl$.
 It is clear from
Corollary~\ref{finitely-presentable-CDG-modules-coherent-CDG-rings}
that Theorem~\ref{becker-coderived-compactly-generated} is applicable.
\end{proof}

\Section{Coherent Schemes and Matrix Factorizations}
\label{coherent-schemes-matrix-factorizations-secn}

 In this section we apply the results of
Section~\ref{locally-coherent-secn} to the locally coherent
abelian DG\+categories of complexes of quasi-coherent sheaves,
or more generally, quasi-coherent matrix factorizations on
coherent schemes.
 The main result is a generalization of the existence and description
of compact generators in the coderived category from the case of
a Noetherian scheme treated in~\cite[Section~3.11]{Pkoszul},
\cite[Proposition~1.5(d) and Corollary~2.3(l)]{EP} to that of
a coherent scheme.

\subsection{Coherent schemes and coherent sheaves}
\label{coherent-schemes-subsecn}
 The following lemma establishes the key fact of Zariski locality of
the coherence property for commutative rings.

\begin{lem} \label{coherence-locality}
\textup{(a)} Let $U$ be an an affine scheme and $V\subset U$ be
an affine open subscheme in~$U$.
 Assume that the ring $O(U)$ is coherent.
 Then the ring $O(V)$ is coherent, too. \par
\textup{(b)} Let $U$ be an affine scheme and $U=\bigcup_{\alpha=1}^d
V_\alpha$ be a finite affine open covering of~$U$.
 Assume that all the rings $O(V_\alpha)$ are coherent.
 Then the ring $O(U)$ is coherent as well.
\end{lem}

\begin{proof}
 Part~(a): notice that the restriction map $O(U)\rarrow O(V)$ is
a flat epimorphism of commutative rings (in the sense
of~\cite[Sections~XI.1\<2]{Sten}).
 For any flat epimorphism of commutative rings $R\rarrow S$,
coherence of the ring $R$ implies coherence of the ring~$S$
\,\cite[Proposition~3.7]{CEI}.

 Part~(b): notice that the restriction map $O(U)\rarrow
\bigoplus_{\alpha=1}^d O(V_\alpha)$ is a faithfully flat
homomorphism of commutative rings.
 For any faithfully flat homomorphism of commutative rings
$R\rarrow S$, coherence of the ring $S$ implies coherence of
the ring~$R$ \,\cite[Corollary~2.1]{Harr}, \cite[Propositions~I.5.9
and~I.6.11]{Bour}.
\end{proof}

 Given a scheme $X$, we denote by $X\rQcoh$ the abelian category of
quasi-coherent sheaves on~$X$.
 A scheme $X$ is said to be \emph{locally coherent}~\cite{CEI,EG} if,
for every affine open subscheme $U\subset X$, the commutative ring
$O(U)$ is coherent.
 It follows from Lemma~\ref{coherence-locality} that it suffices to
check this property for affine open subschemes $U_\alpha\subset X$
ranging over any given affine open covering of a scheme~$X$.

 We will say that a scheme $X$ is \emph{coherent} if $X$ is locally
coherent, quasi-compact, and quasi-separated.

\begin{prop} \label{quasi-coherent-locally-coherent}
 Let $X=\bigcup_\alpha U_\alpha$ be an affine open covering of
a coherent scheme~$X$.
 Then the abelian category $X\rQcoh$ is locally coherent.
 An object $M\in X\rQcoh$ is coherent (equivalently, finitely
presentable) if and only if, for every index~$\alpha$,
the $O(U_\alpha)$\+module $M(U_\alpha)$ is coherent (equivalently,
finitely presentable).
\end{prop}

\begin{proof}
 We refer to~\cite[Proposition~I.6.11]{Bour}
or~\cite[Lemma~2.1]{PS6} for a discussion of Zariski locality of
finite presentability of modules over commutative rings.
 For any quasi-compact quasi-separated scheme $X$, the category
of quasi-coherent sheaves $X\rQcoh$ is locally finitely presentable,
and an object $M\in X\rQcoh$ is finitely presentable if and only if
the $O(U_\alpha)$\+module $M(U_\alpha)$ is finitely
presentable for every index~$\alpha$ of an affine open covering
$X=\bigcup_\alpha U_\alpha$ \,\cite[0.5.2.5 and
Corollaire~I.6.9.12]{EGA1}.
 Now the condition that the kernel of any epimorphism of
finitely presentable quasi-coherent sheaves on $X$ is finitely
presentable can be checked locally, and it holds for a coherent
scheme $X$ because it holds for a coherent affine scheme.
 Finally, in a locally coherent abelian category an object is coherent
if and only if it is finitely presentable.
\end{proof}

 A quasi-coherent sheaf $M$ on a coherent scheme $X$ is said to be
\emph{coherent} if it satisfies the equivalent conditions of
the second assertion of
Proposition~\ref{quasi-coherent-locally-coherent}.
 We denote the full subcategory of coherent sheaves by
$X\rcoh=(X\rQcoh)_\fp\subset X\rQcoh$.
 According to Proposition~\ref{quasi-coherent-locally-coherent},
$X\rcoh$ is an abelian category for any coherent scheme~$X$.

%

\begin{cor}
 Let $X$ be a coherent scheme and\/ $\sH^0(\bC(X\rQcoh_\inj))$ be
the homotopy category of unbounded complexes of injective
quasi-coherent sheaves on~$X$.
 Then\/ $\sH^0(\bC(X\rQcoh_\inj))$ is a compactly generated
triangulated category.
 The full subcategory of compact objects in\/
$\sH^0(\bC(X\rQcoh_\inj))$ is equivalent to the bounded derived
category\/ $\sD^\bb(X\rcoh)$ of the abelian category of coherent
sheaves on~$X$.
\end{cor}

\begin{proof}
 This is a particular case of~\cite[Corollary~6.13]{Sto2}, which is
applicable in view of Proposition~\ref{quasi-coherent-locally-coherent};
cf.\ Theorem~\ref{locally-coherent-abelian-theorem} from
the introduction.
 Disregarding the direct summand closure issue, this is also
a particular case of Theorem~\ref{becker-coderived-compactly-generated}
applied to the locally coherent abelian DG\+category $\bA=\bC(X\rQcoh)$
of complexes in $X\rQcoh$ (as per
Examples~\ref{locally-coherent-DG-category-of-complexes});
cf.\ Theorem~\ref{locally-coherent-abelian-DG-category-theorem}.
\end{proof}

\subsection{Quasi-coherent and coherent factorizations}
 The setting in this section is a common particular case
of~\cite[Example~4.41]{Pedg} and~\cite[Example~4.42 with
Remark~2.7]{Pedg}.

 Let $X$ be a scheme, $L$ be a line bundle (invertible quasi-coherent
sheaf) on~$X$, and $w\in L(X)$ be a global section.
 In the context of
Example~\ref{Grothendieck-abelian-DG-of-factorizations}, put
$\sA=X\rQcoh$, and let $\Delta\:\sA\rarrow\sA$ be the functor of
twisting with $L$, that is $\Delta(M)=L\ot_{O_X}M$ for all
$M\in X\rQcoh$.
 Let $v\:\Id_\sA\rarrow\Delta$ be the natural transformation of
multiplication with~$w$, i.~e., the map $v_M(U)\:M(U)\rarrow
(L\ot_{O_X}M)(U)$ takes a local section $s\in M(U)$ to the local
section $w|_U\ot s\in(L\ot_{O_X}M)(U)$ for all $M\in X\rQcoh$ and
all open subschemes $U\subset X$.

 Then the objects of the DG\+category $\bF_\qc(X,L,w)=\bF(\sA,\Delta,v)$
are called \emph{quasi-coherent} (\emph{matrix}) \emph{factorizations}
of the potential $w\in L(X)$ on the scheme~$X$.
 Explicitly, a quasi-coherent factorization $N^\cu$ on $X$ is
a sequence of quasi-coherent sheaves $N^n\in X\rQcoh$, \,$n\in\boZ$,
together with a sequence of periodicity isomorphisms $\delta_N^{n+2,n}\:
L\ot_{O_X}N^n\simeq N^{n+2}$ and a sequence of differentials
$d_{N,n}\:N^n\rarrow N^{n+1}$ such that the composition
$d_{N,n+1}\circ d_{N,n}\:N^n\rarrow N^{n+2}$ is equal to the composition
of the multiplication map $w\:N^n\rarrow L\ot_{O_X}N^n$ with
the isomorphism~$\delta_N^{n+2,n}$ for every $n\in\boZ$.
 Example~\ref{Grothendieck-abelian-DG-of-factorizations} tells that
the DG\+category of quasi-coherent factorizations $\bF_\qc(X,L,w)$ is
a Grothendieck abelian DG\+category.

 Let us mention that, following the discussion in
Section~\ref{almost-involution-factorizations-subsecn} (commutative
diagram~\eqref{factorizations-Phi-diagram}), the full DG\+subcategory
of graded-injective objects $\bF_\qc(X,L,w)_\binj\subset
\bF_\qc(X,L,w)$ consists of all the \emph{injective quasi-coherent
factorizations}.
 The latter term means quasi-coherent factorizations $N^\cu$ of
the potential~$w$ on $X$ such that the quasi-coherent sheaf
$N^n$ is injective in $X\rQcoh$ for every $n\in\boZ$.
 So Becker's coderived category $\sD^\bco(\bF_\qc(X,L,w))\simeq
\sH^0(\bF_\qc(X,L,w)_\binj)$ is the homotopy category of injective
quasi-coherent factorizations.

 Now let us assume that $X$ is a coherent scheme (as defined in
Section~\ref{coherent-schemes-subsecn}).
 Then the full DG\+subcategory of \emph{coherent factorizations}
$\bF_\coh(X,L,w)\subset\bF_\qc(X,L,w)$ consists of all the factorizations
$N^\cu$ such that $N^n$ is a coherent sheaf on $X$ for every $n\in\boZ$.
 Following Example~\ref{abelian-DG-category-of-factorizations-example}
applied to the abelian category $\sE=X\rcoh$, the DG\+category
$\bF_\coh(X,L,w)$ is abelian; in fact, it is an exactly embedded full
abelian DG\+subcategory in $\bF_\qc(X,L,w)$ in the sense of
Section~\ref{exactly-embedded-subsecn}.

 Moreover,
Example~\ref{locally-coherent-DG-category-of-factorizations} together
with Proposition~\ref{quasi-coherent-locally-coherent} tell that
the DG\+category of quasi-coherent factorizations $\bF_\qc(X,L,w)$ on
a coherent scheme $X$ is a locally coherent abelian DG\+category, and
the full DG\+subcategory of coherent factorizations it its full
DG\+subcategory of finitely presentable/coherent objects, 
$\bF_\coh(X,L,w)=\bF_\qc(X,L,w)_\bfp$.

 The following corollary generalizes~\cite[Corollary~2.3(l)]{EP}
from the case of a Noetherian scheme to that of a coherent one.

\begin{cor} \label{qcoh-factorizations-coderived-comp-gen}
 Let $X$ be a coherent scheme, $L$ be a line bundle on $X$, and
$w\in L(X)$ be a global section.
 Let\/ $\bF_\qc(X,L,w)$ be the locally coherent DG\+category of
quasi-coherent factorizations of the potential $w\in L(X)$ on $X$,
and let $\bF_\coh(X,L,w)\subset\bF_\qc(X,L,w)$ be the (exactly
embedded full abelian) DG\+subcategory of coherent factorizations.
 Then the induced triangulated functor from the absolute derived
category to Becker's coderived category
$$
 \sD^\abs(\bF_\coh(X,L,w))\lrarrow\sD^\bco(\bF_\qc(X,L,w))
$$
is fully faithful.
 The objects in its image are compact in\/
$\sD^\bco(\bF_\qc(X,L,w))$, and form a set of compact generators for\/
$\sD^\bco(\bF_\qc(X,L,w))$.
\end{cor}

\begin{proof}
 This is the particular case of
Theorem~\ref{becker-coderived-compactly-generated} for
$\bA=\bF_\qc(X,L,w)$.
 Example~\ref{locally-coherent-DG-category-of-factorizations} with
Proposition~\ref{quasi-coherent-locally-coherent} imply that
Theorem~\ref{becker-coderived-compactly-generated} is applicable.
\end{proof}

\end{document}